\documentclass[12pt]{amsart}
\usepackage{amssymb,amsmath,amsthm}
\usepackage{latexsym}
\usepackage{epsfig}
\usepackage{graphicx}

\newtheorem{theorem}{Theorem}[section]
\newtheorem{lemma}[theorem]{Lemma}
\newtheorem{remark}[theorem]{Remark}
\newtheorem{definition}[theorem]{Definition}

\newtheorem{proposition}[theorem]{Proposition}
\newtheorem{corollary}[theorem]{Corollary}
\newtheorem{problem}[theorem]{Problem}
\newtheorem{example}[theorem]{Example}

\newtheorem{formulla}[theorem]{}
\long\def\symbolfootnote[#1]#2{\begingroup%
\def\thefootnote{\fnsymbol{footnote}}\footnote[#1]{#2}\endgroup}

\newcommand\Z{{\mathbb Z}}

\newcommand\R{{\mathbb{R}}}
\begin{document}
\title{Knots and distributive homology:\\ \footnotesize {from arc colorings to Yang-Baxter 
homology}}\footnote{This paper 
has its roots in two series of talks I gave:
in Russia (Lomonosov Moscow State University, May 29-June 1, 2012), where the visualization in Figure 6.2 
was observed,  Korea (TAPU Workshop on Knot Theory, 
July 23-27, 2012), and in a talk at Oberwolfach Conference (June 3-9, 2012). The short version of this paper was 
published in Oberwolfach Proceedings \cite{Prz-6}. While I keep novelty of the talks (many new ideas were crystallized 
then), I added a lot of supporting material so the paper is mostly self-sufficient. I also kept, to some extent, 
the structure of talks; it may lead to some repetitions but I hope it is useful for a reader.}
 \author{J\'ozef H. Przytycki} 

\maketitle
\markboth{\hfil{\sc Knots and distributive homology}\hfil}
\ \
\tableofcontents
\section{Introduction}
While homology theory of associative structures, such as groups and rings,
has been extensively studied in the past beginning with the work of Hurewicz, Hopf, Eilenberg, and
Hochschild, the non-associative structures, such as racks or quandles, were neglected until
recently. The distributive structures\footnote{The word {\it distributivity} was coined in 1814 
by French mathematician Francois Servois (1767-1847).} have been studied for a long time and already C.S.
Peirce (1839-1914) in 1880 \cite{Peir} emphasized the importance of (right) self-distributivity in algebraic
structures, and his friend E.~Schr\"oder, \cite{Schr} gave an example of a three element magma $(X;*)$ 
which is not associative\footnote{The example Schr\"oder (1841-1902) gave is 
\begin{center}
\begin{tabular}{c|ccc}
$*$&0&1&2\\ \hline
0&0&2&1\\
1&2&1&0\\
2&1&0&2\\
\end{tabular}
\end{center}
and we named elements of the magma by $0,1$ and $2$ as this example is the base for Fox three colorings 
of links (developed about 1956) \cite{C-F,Cro,Prz-3}, and the operation can be written as $x*y=2y-x$ modulo $3$; 
it happen to be self-distributive from both sides, that is $(x*y)*z= (x*z)*(y*z)$ and $x*(y*z)= (x*y)*(x*z)$ 
(see Example 2.3(5)).} (compare Section 2).

 However, homology for such universal algebras was introduced only between 1990 and 1995
by Fenn, Rourke, and Sanderson \cite{FRS-1,FRS-2,FRS-3,Fenn} .  We develop theory in the historical
context and propose a general framework to study homology of distributive structures.
We outline potential relations to Khovanov homology and categorification, via Yang-Baxter operators.
We use here the fact that Yang-Baxter equation can be thought of as a generalization of self-distributivity.

\subsection{Invariants of arc colorings}
Consider a link diagram $D$, say {\parbox{1.1cm}{\psfig{figure=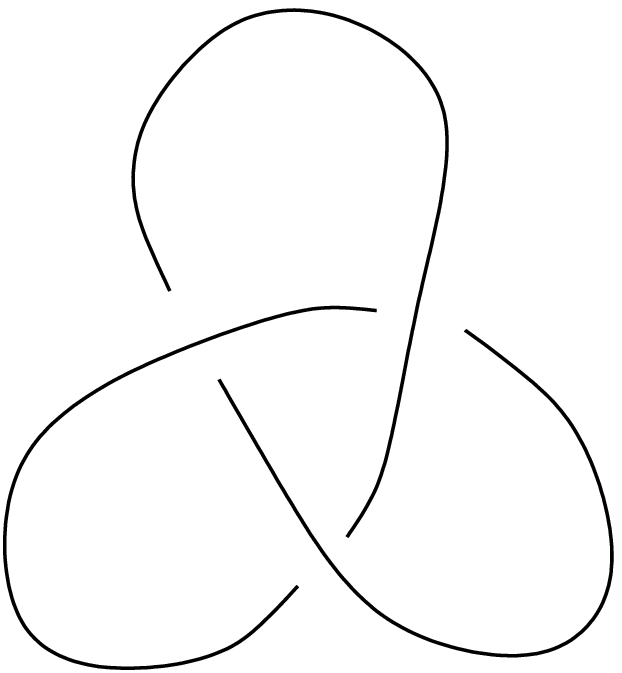,height=0.6cm}}, 
and a finite set $X$. We may define a diagram invariant to be the number of colorings of arcs\footnote{We use the 
term {\it arc} for a part of the diagram from an undercrossing to the next undercrossing (including possibility 
of a trivial component), and the term {\it semi-arc} for a part of the diagram from a crossing to the next  
crossing. Thus, for example, a standard trefoil knot diagram has three arcs and six semi-arcs.}
 of $D$ by elements of $X$,
$col_X(D)$. Even such a naive definition leads to a link invariant $col_X(L)= min_{D\in L} col_X(D)$, where 
$D\in L$ means that $D$ is a diagram of $L$.\footnote{One can say that it is nonsense but an invariant 
is nontrivial: $col_X(L)= |X|^{cr(L)+t(L)}$, where $cr(L)$ is the crossing number of $L$ and $t(L)$ the number 
of trivial components in $L$. This is the case as for a knot diagram $D$ with at least one crossing the 
number of arcs equals to the number of crossings.} 
More sensible approach would start with a magma $(X;*)$, that is a set with binary 
operation, and with the coloring convention of Figure 1.1.
\ \\
\centerline{{\psfig{figure=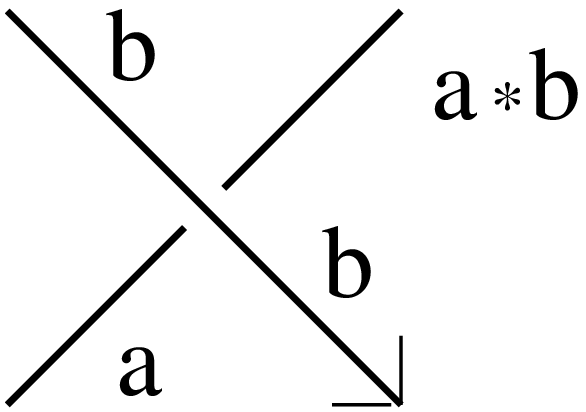,height=1.9cm}}}\ \\
\centerline{Figure 1.1; convention for a magma coloring of a crossing} \ \\
 Again, let for a finite $X$, $col_{(X;*)}(D)$ denote 
the number of colorings of arcs of $D$ by elements of $X$, according to the convention given in Figure 1.1,
 at every crossing.
We can define an oriented link invariant by considering $col_{(X;*)}(L)= min_{D\in L} col_{(X;*)}(D)$; Alternatively, 
we can minimize $col_{(X;*)}(L)$ over minimal crossing diagrams of $L$ only. Such an invariant would be very difficult to 
compute so it is better to look for properties of $(X;*)$ so that $col_{(X;*)}(D)$ is invariant under Reidemeister 
moves: $R_1$ ({\parbox{0.7cm}{\psfig{figure=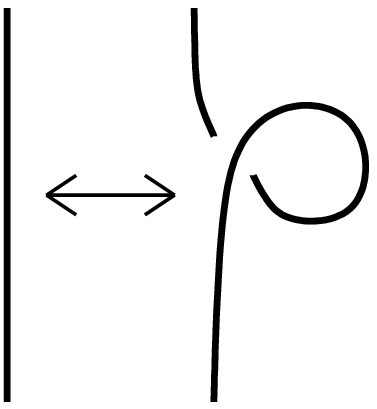,height=0.7cm}}) gives idempotent 
relation $a*a=a$.\footnote{The names {\it idempotent} (same power) and nilpotent (zero power) were  
introduced in 1870 by Benjamin Peirce \cite{Pei}, page 20.}
$R_2$ ({\parbox{0.9cm}{\psfig{figure=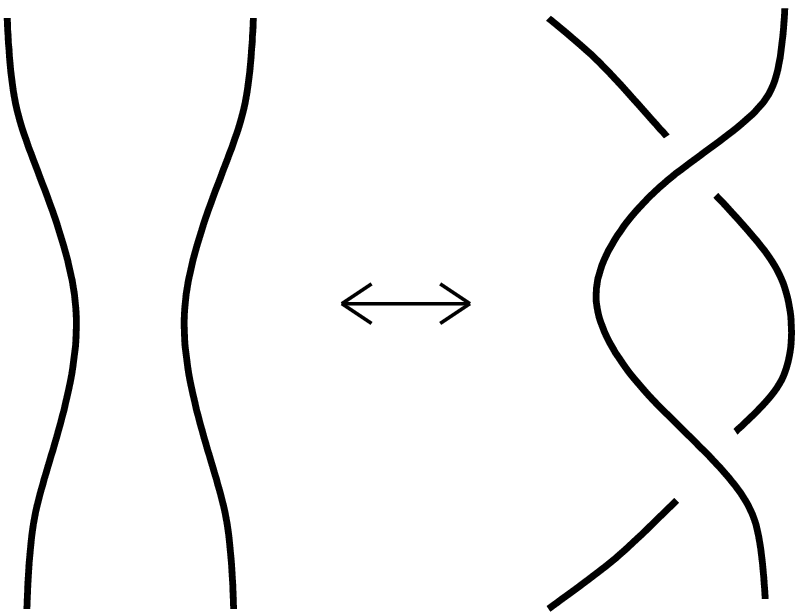,height=0.7cm}}) forces $*$ to be an invertible operation, and the third 
move illustrated in detail in Figure 1.2  below, forces on $*$ a right self-distributivity $(a*b)*c=(a*c)*(b*c)$.\\ \ \\
\centerline{\psfig{figure=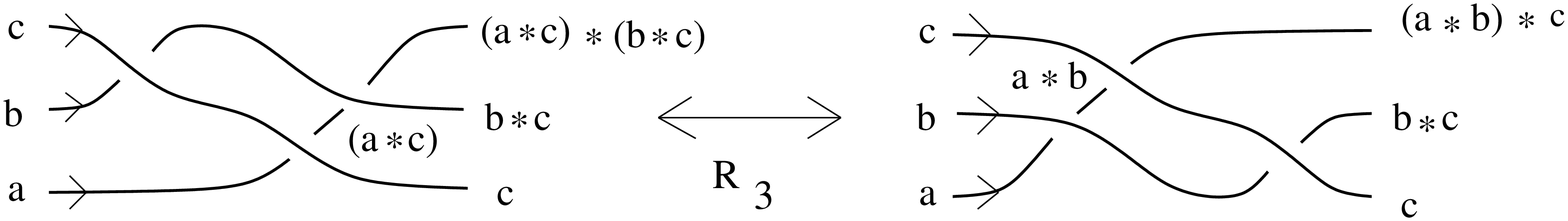,height=1.8cm}}\\ \ \\
\centerline{Figure 1.2; magma coloring of a third Reidemeister move and distributivity}
\ \ \\

The magma $(X;*)$ satisfying all three conditions is called a quandle, the last two -- a rack, and only 
the last condition -- a shelf or RDS (right distributive system). 
Thus, if $(X;*)$ is a quandle then $col_{(X;*)}(D)$ (which we denote from now on succinctly $col_X(D)$)
 is a link invariant. We also use the notation $Col_X(D)$ for the set of $X$-colorings of $D$, thus
$col_X(D)= |Col_X(D)|$.
 We also can try more generally to color semi-arcs of a diagram by elements of $X$ and 
declare for each colored crossing a weight of the crossing. This approach leads to state sums and 
Yang-Baxter operators (see Section \ref{Section 6} and Figures 6.2, 12.1, and 12.3).

We also can do more with distributive magmas (after S.Carter, S.Kamada, and M.Saito \cite{CKS}; compare also 
M.Greene thesis \cite{Gr}). 
We can sum over all crossings the pairs $\pm (a,b)$ according to the convention
{\parbox{3.3cm}{\psfig{figure=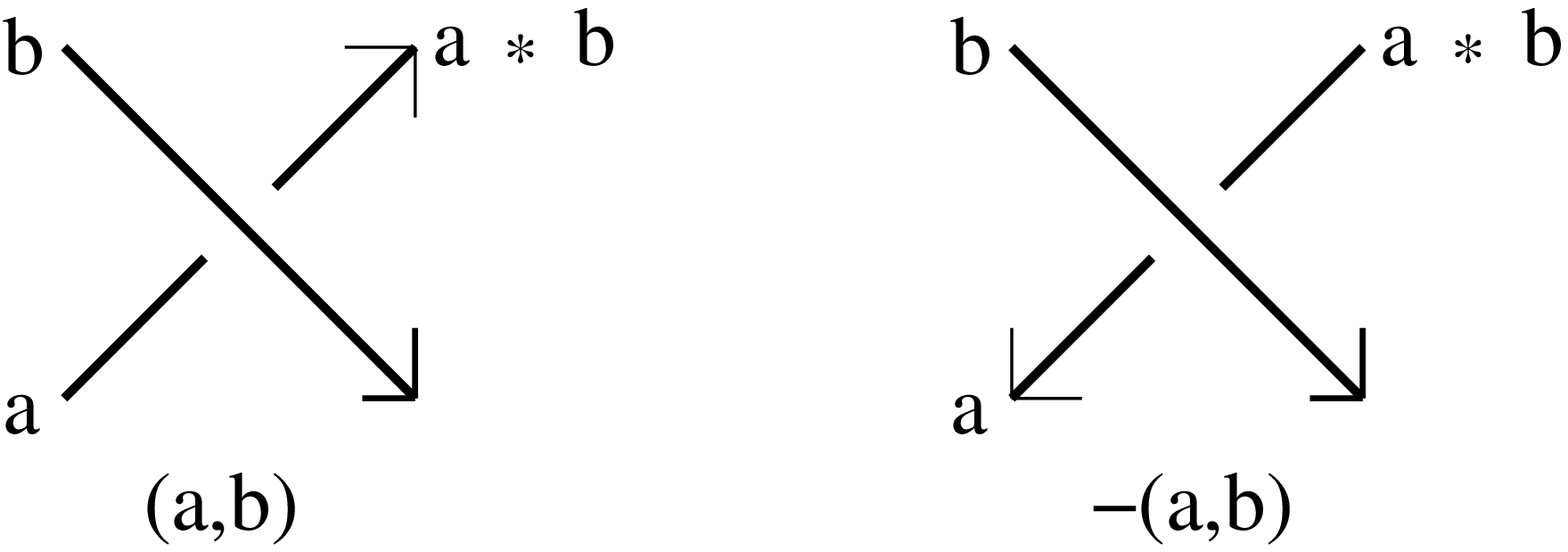,height=1.3cm}}; the investigation of invariance of $\sum \pm (a,b)$ 
under Reidemeister moves was a hint toward construction of (co)homology of quandles.

We also encounter right distributivity by asking the following question: for a given coloring $\phi \in Col_X(D)$ by a 
magma $(X;*)$, and an element $x\in X$ is coloring $\phi * x$ also a magma coloring?  
If, as before, coloring is given by $a$, $b$,
and $c=a*b$ then the new coloring of a crossing is given by $a*x$, $b*x$, and $c*x$. For a magma coloring we need 
$(a*b)*x = (a*x)*(b*x)$ which is exactly right self-distributivity. 
To put it on a more solid footing, we observe that the map $*_x:X \to X$ with $*_x(a)=a*x$ is a magma 
homomorphism: $*_x(a*b)=(a*b)*x = (a*x)*(b*x)= *_x(a)**_x(b)$. For any magma homomorphism $g: X \to X$ if 
$f: arcs(D) \to X$ is a magma coloring then $gf$ defined by $(gf)(arc)=g(f(arc))$ is a magma coloring of $D$. 
These are classical observations thus it is 
interesting to notice the slightly more general fact concerning the following question:

Let $(X;*_1)$ be a magma and $f$ and $g$ two $(X;*_1)$ magma colorings of a diagram $D$. Find the magma 
operation $*_2$ so that $f*_2g$ is also a $(X;*_1)$ magma coloring. The question reduces to the previous one if 
$g$ is a trivial (i.e. constant)
 coloring and $*_1=*_2$. The nice condition which answers the question was first discussed by 
M.Niebrzydowski at his talk at Knots in Washington XXXV conference in December of 2012 \cite{Nieb} 
(compare also \cite{C-N}).

\begin{lemma}\label{Lemma 1.1} Let  $f, g: arcs(D) \to X$  be two colorings of a diagram $D$ by $(X;*_1)$ 
(that is $f,g\in Col_{(X;*_1)}(D)$). Then $f*_2g$ where $*_2$ is another binary operation on $X$ and 
$(f*_2g)(arc)=f(arc)*_2g(arc)$, is an $(X;*_1)$ coloring of $X$ if operations
 $*_1$, and $*_2$ are entropic one with respect to the other, that is:
$(a*_1b)*_2(c*_1d) = (a*_2c)*_1(b*_2d)$.

\end{lemma}
\begin{proof}
 For every crossing with initial under-arc $a$ and over-arc $b$ we need
$$(f(a)*_2g(a))*_1(f(b)*_2g(b)) = (f(a)*_1f(b))*_2(g(a)*_1g(b))$$
which is exactly the entropic condition in Figure 1.3 (compare Subsection \ref{Subsection 8.2}).
\end{proof}

\ \\ \ \\
\centerline{\psfig{figure=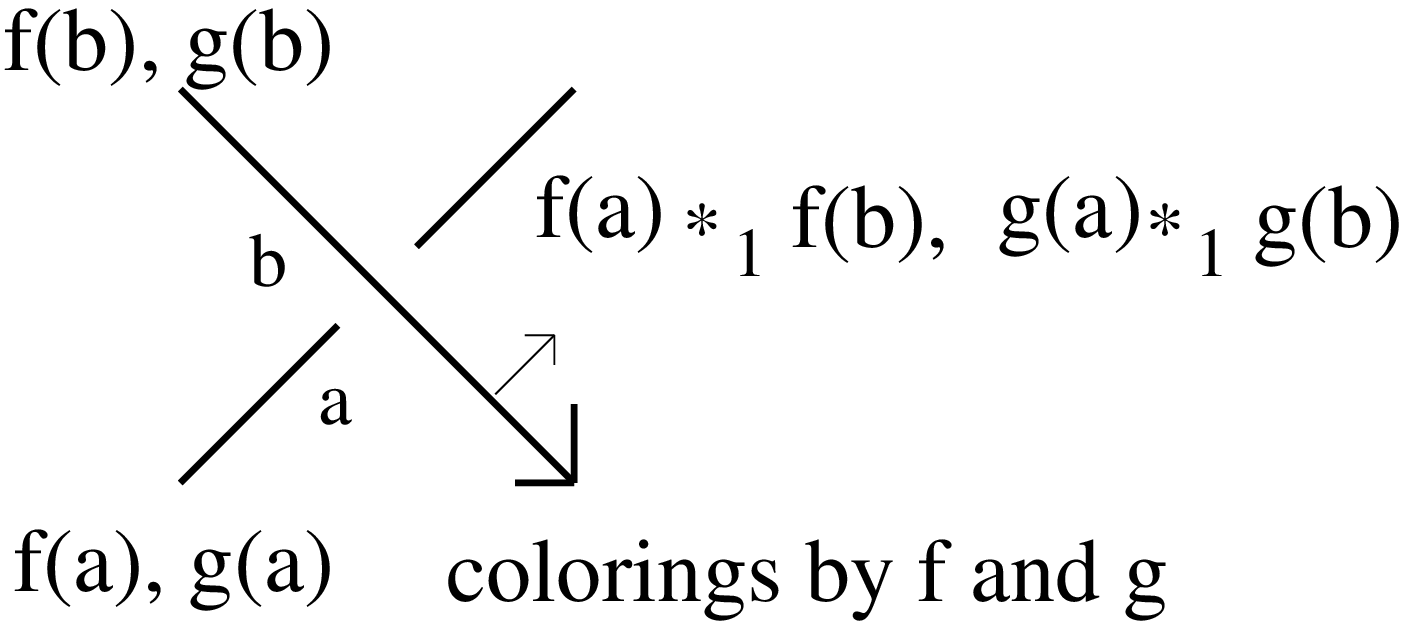,height=4.3cm}}\ \\ \ \\
\centerline{\psfig{figure=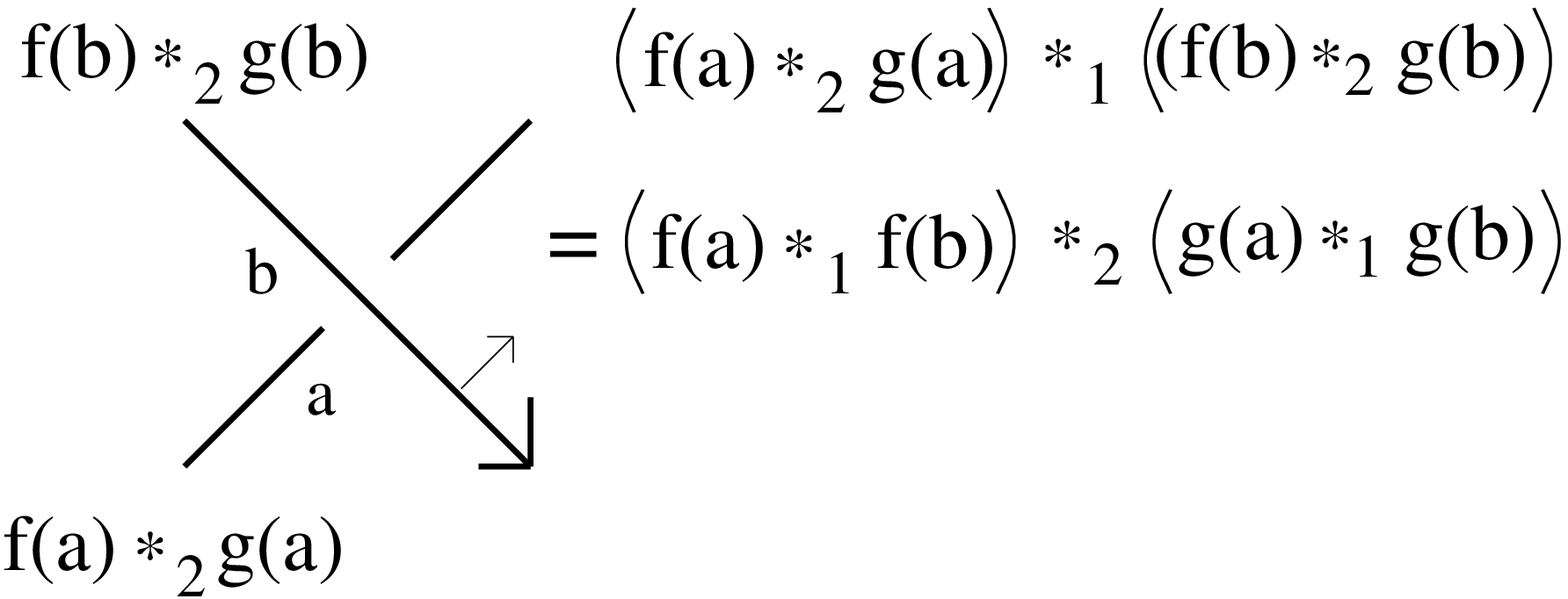,height=4.6cm}}
\\ \ \\
\centerline{Figure 1.3;  Entropy condition for composition of $f$ and $g$: $f*_2g$}
\ \\
Notice that if $g$ is a trivial coloring, say $g(arc)= x$ for any arc then any crossing forces $x*_1x=x$
and the entropic 
equation reduces to $(f(a)*_2x)*_1(f(b)*_2x) = (f(a)*_1f(b))*_2(x*_1x)= (f(a)*_1f(b))*_2x$ (right distributivity).
For use of entropic magmas in Knot Theory, see \cite{N-P-4,Prz-1,Prz-2,Prz-8,P-T-1,P-T-2,Si}; compare also 
Proposition \ref{Proposition 2.6}. 

We introduce now a monoid of binary operations and show that distributivity can be studied in the 
context of this monoid. Then we
compare homology for associative structures (semigroups) with that for distributive structures (shelves).
We also compare extensions in associative and distributive cases.

The paper is planned as a continuation of pioneering essay \cite{Prz-5} and for completeness we 
recall parts of the essay.

\section{Monoid of binary operations}\label{Section 2}
Let $X$ be a set and $*:X\times X \to X$ a binary operation. We call $(X;*)$ a 
{\it magma}\footnote{The term magma was used by J-P.~Serre \cite{Ser} and Bourbaki \cite{Bou}, replacing 
the older term {\it groupoid} which started to mean a category with every morphism invertible.}.
For any $b\in X$ the adjoint map $*_b: X\to X$, is defined by $*_b(a)=a*b$.
Let $Bin(X)$ be the set of all binary operations on $X$.

\begin{proposition}\label{Proposition 2.1} $Bin(X)$ is a monoid
(i.e. semigroup with identity)  with the composition
$*_1*_2$ given by $a*_1*_2b= (a*_1b)*_2b$, and the identity $*_0$ being the right trivial operation, that is,
$a*_0b=a$ for any $a,b\in X$.
\end{proposition}
If $* \in Bin(X)$ is invertible then $*^{-1}$ is usually denoted by $\bar *$.
One should remark that $*_0$ is distributive with respect to any other operation, that is,
$(a*b)*_0c= a*b= (a*_0c)*(b*_0c)$, and $(a*_0b)*c= a*c= (a*c)*_0(b*c)$.
\begin{definition}\label{Definition 2.2} Let $(X;*)$ be a magma, then:
\begin{enumerate}
\item[(i)] If $*$ is right self-distributive, that is, $(a*b)*c=(a*c)*(b*c)$,
then $(X;*)$ is called an RDS (right distributive structure) or a 
shelf (the term coined by Alissa Crans and used in knot theory \cite{Cr}).
\item[(ii)] If a shelf $(X;*)$ satisfies the idempotent condition, $a*a=a$ for any $a\in X$, then it
is called an RDI structure or {\it right spindle}, or just a spindle (again the term coined by Crans).
\item[(iii)] If a shelf $(X;*)$ has $*$ invertible in $Bin(X)$ (equivalently $*_b$ is a bijection for any $b\in X$),
then it is called a {\it rack} (the term wrack, like in ``wrack and ruin", of J.H.Conway from 1959 \cite{C-W},
was modified to rack in \cite{F-R}).
\item[(iv)] If a rack $(X;*)$ satisfies the idempotent condition, then it is called a {\it quandle} (the term
coined in Joyce's PhD thesis of 1979 \cite{Joy-1,Joy-2}). Axioms of a quandle were motivated by three
Reidemeister moves (idempotent condition by the first move, invertibility by the second, and right self-distributivity
by the third move).
\item[(v)] If a quandle $(X;*)$ satisfies $**=*_0$ (i.e. $(a*b)*b=a$) then it is called  {\it kei} or
an involutive quandle. The term kei (\psfig{figure=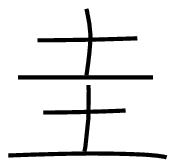,height=0.3cm})
was coined in a pioneering paper by M.~Takasaki\footnote{Mituhisa Takasaki worked at Harbin Technical
University,
likely as an assistant to K\^oshichi Toyoda. Both perished when Red Army entered Harbin in August 1945.
} in 1942 \cite{Tak}
\end{enumerate}
\end{definition}
The main early example of a rack (and a quandle) was a group $G$ with a $*$ operation given
by conjugation, that is, $a*b=b^{-1}ab$ (Conway thought of this as a ``wrack" of a group). 
Another example, considered already in Conway-Wraith correspondence \cite{C-W}, is defined for any group with 
$a*b=ba^{-1}b$ and called by Joyce (after Bruck \cite{Bru}) a core quandle. This example, for 
a group, $H$, abelian was already studied by Takasaki so we call it Takasaki kei (or quandle) and 
denote by $T(H)$ (in abelian notation we have $a*b=2b-a$), compare \cite{N-P-1}. 
$T(Z_n)$ is often called a dihedral quandle 
and denoted by $R_n$; it can be interpreted as composed of reflections of the dihedral group $D_{2n}$ 
(we can mention that rack and quandle homology of $T(Z_n)$ for prime $n$ has been computed in 
\cite{N-P-2,Cla,Nos}).
 
More general examples still starting from a group are given in Joyce paper \cite{Joy-2}:
\begin{example}\label{Example 2.3} 
Let $G$ be a group and $t:G\to G$ a group homomorphism then we have the following 
spindle structures on $G$:
\begin{enumerate}
\item[(1)] $a*_1b= t(ab^{-1})b$,
\item[(2)] $a*_2b= t(b^{-1}a)b$,
\item[(3)] If $t$ is invertible both examples give quandles  
where $\bar *_1$ and $\bar *_2$  are given by the formulas: \\
(i) $a\bar *_1 b= t^{-1}(ab^{-1})b$ thus $\bar *_1$ yielded by the automorphism $t$ 
is equal to $*_1$ yielded by the automorphism $t^{-1}$. \\
(ii) $a\bar *_2 b=bt^{-1}(ab^{-1})$, e.g. we check that \\
 $(a*_2b)\bar*_2b=bt^{-1}((t(b^{-1}a)bb^{-1})=a $. This example is related to the fundamental group of cyclic (branched) 
covers of $S^3$ along a link. Locally, at every crossing we have relations $C=\tau^{-1}(B^{-1}A)B$ and $A=B\tau(CB^{-1})$,
as illustrated in Figure 2.1 \cite{Prz-3,Pr-Ro-1,DPT}.
\item[(4)] If $G$ is an abelian group both examples lead to the same spindle called Alexander spindle 
(for $t$ invertible, Alexander quandle). In abelian notation we get $a*b= ta-tb+b= (1-t)b + ta$.
This two sided distributive structure was already considered in 1929 by C.Burstin\footnote{Celestyn Burstin (1888-1938)
was born in Tarnopol, where he obtained ``Matura" in 1907, he moved to Vienna where in 1911 he completed university. 
In 1929 he moved to Minsk where he was a member of the Belarusian National Academy of Sciences,
and a Director of the Institute of Mathematics of the Academy. In December 1937, 
he was arrested on suspicion of activity as a
spy for Poland and Austria. He died in October 1938, when interrogated in a prison in Minsk (``Minskaja Tjurma");
he was rehabilitated March 2, 1956 \cite{Bur-1,Bur-M,Mal,Mio}.} and 
W.Mayer\footnote{Walter Mayer (1887--1948) is well known for Mayer-Vietoris sequence and for being
assistant to A.Einstein at Institute for Advanced Study, Princeton \cite{Isa}.} \cite{B-M}.
\item[(5)] If $t=-1$ we get $a*b= 2b-a$ and this structure, as mentioned before, was the main example of Kei by Takasaki 
so we denote it by $T(G)$.  
\item[(6)] $a*_3b= t(ba^{-1})b$ with $t^2=t$. It is a quandle if and only if  $t=Id$ in which case we get 
a quandle called the core quandle of $G$.
\end{enumerate} 
\end{example}
\ \\
\centerline{{\psfig{figure=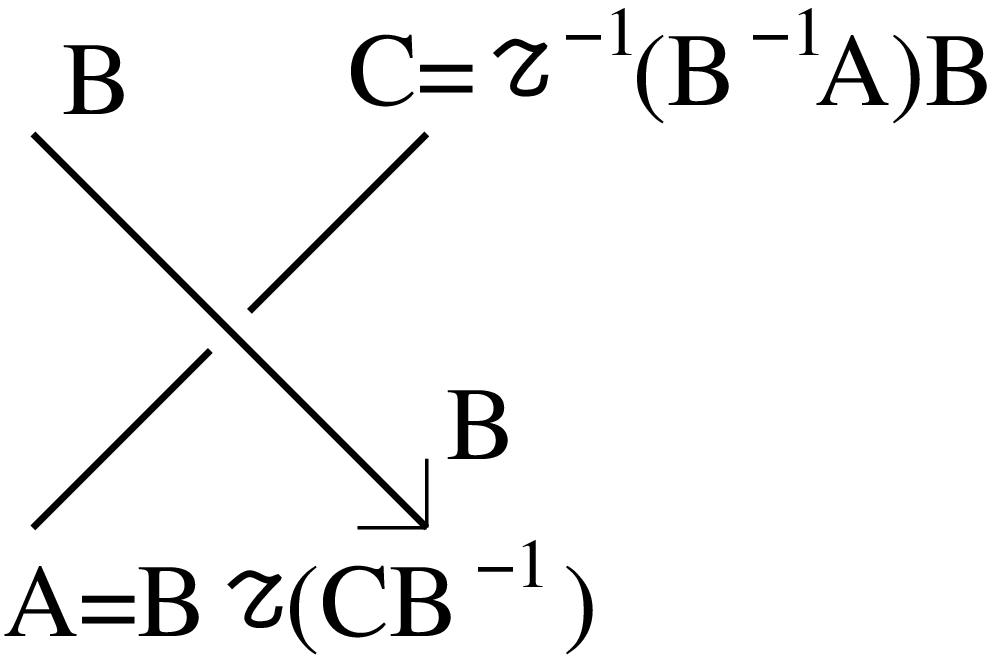,height=3.9cm}}}\ \\ \ \\
\centerline{Figure 2.1; relations for cyclic covering; see Example \ref{Example 2.3}(3)} \ \\

Definition \ref{Definition 2.2} describes properties of an individual magma $(X;*)$. It is also useful to
consider subsets or submonoids of $Bin(X)$ satisfying the related conditions (compare \cite{R-S,Mov,Deh-1,Prz-5}. 
\begin{definition}\label{Definition 2.3}
We say that a subset $S \subset Bin(X)$ is a distributive set if all
pairs of elements $*_{\alpha},*_{\beta} \in S$ are right distributive,
 that is, $ (a*_{\alpha}b)*_{\beta}c= (a*_{\beta}c)*_{\alpha}(b*_{\beta}c)$ (we allow  $*_{\alpha}=*_{\beta}$).
\begin{enumerate}
\item[(i)]
The pair $(X;S)$ is called a multi-shelf if $S$ is a distributive set.
If $S$ is additionally a submonoid (resp. subgroup) of $Bin(X)$, we say that it is a distributive monoid (resp. group).
\item[(ii)]
If $S \subset Bin(X)$ is a distributive set such that each $*$ in $S$ satisfies the idempotent condition, we call
$(X;S)$ a multi-spindle.
\item[(iii)] We say that $(X;S)$ is a multi-rack if $S$ is a distributive set, and  all elements of $S$ are invertible.
\item[(iv)] We say that $(X;S)$ is a multi-quandle if $S$ is a distributive set, and
elements of $S$ are invertible and satisfy the idempotent condition.
\item[(v)] We say that $(X;S)$ is a multi-kei if it is a multi-quandle with $**=*_0$ for any $*\in S$.
Notice that if $*_1^2=*_0$ and $*_2^2=*_0$ then $(*_1*_2)^2=*_0$; more generally if
$*_1^n=*_0$ and $*_2^n=*_0$ then $(*_1*_2)^n=*_0$. This follows from the fact that elements of a 
multi-quandle commute pairwise (this was observed by M.Jab{\l}onowski \cite{Prz-5}).
\end{enumerate}
\end{definition}

\begin{proposition}\cite{Prz-5}\label{Proposition 2.5}
\begin{enumerate}
\item[(i)] If $S$ is a distributive set and $*\in S$ is invertible, then $S\cup \{\bar *\}$ is also a
distributive set.
\item[(ii)] If $S$  is a distributive set and $M(S)$ is the monoid generated by $S$ then $M(S)$ is a
distributive monoid.
\item[(iii)] If $S$  is a distributive set of invertible operations and $G(S)$ is the group generated by $S$, then
$G(S)$ is a distributive group.
\end{enumerate}
\end{proposition}

We show, after G.Mezera \cite{Mez}, 
the fact that any group can be embedded in $Bin(X)$ for
some $X$, in particular the regular embedding of $G$ in $Bin(G)$ is given by $g \to *_g$ with
$a*_gb= ab^{-1}gb$ (compare \cite{Lar} and Example X.3.15 of \cite{Deh-2} where the operation $a*_gb= ab^{-1}gb$ 
is called a half-conjugacy). The expression $ab^{-1}gb$ was also discussed with respect to 
free rack by Fenn and Rourke (compare Remark \ref{Remark 8.2}).

Proposition \ref{Proposition 2.5} has its analogue for entropic magmas (that is magmas for which 
$(a*b)*(c*d)= (a*c)*(b*d)$). More precisely, we say that a subset $S\in Bin(X)$ is an entropic set if 
for any $*_\alpha, *_\beta \in S$ we have the entropic condition: 
$$(a*_{\alpha}b)*_{\beta}(c*_{\alpha}d)= (a*_{\beta}c)*_{\alpha}(b*_{\beta}d). \mbox{ Then we have:}$$
\begin{proposition}\cite{N-P-4}\label{Proposition 2.6}
\begin{enumerate}
\item[(i)] If $S$ is an entropic set and $*\in S$ is invertible, then $S\cup \{\bar *\}$ is also an 
entropic set.
\item[(ii)] If $S$ is an entropic set and $M(S)$ is the monoid generated by $S$ then $M(S)$ is 
 an entropic monoid.
\item[(iii)] If $S$  is  an entropic set of invertible operations and $G(S)$ is the group generated by $S$, then
$G(S)$ is an entropic group.
\end{enumerate}
\end{proposition}
 
In the next section we consider homology theory of various magmas, it is useful here to define, for any magma 
$(X;*)$ a supporting structure which we call an $X$-set (it is an old concept for (semi)groups and for quandles it was 
first considered by S.Kamada).

\begin{definition}\label{Definition 2.7}
Let $(X;*)$ a magma and $E$ a set. We say that $E$ is an $X$-set (or right $X$-set) if there is a 
function (right action) $*_E: E\times X \to E$. In general we do not put any conditions on $*_E$ but 
if our  magma satisfies some conditions (e.g. associativity or distributivity) then $*_E$ should satisfy
some related conditions. In particular, we will look for a magma structure on $X\sqcup E$ having similar structure.
\end{definition}

In the following few sections we discuss various homology theories for magmas (e.g. associative or distributive).
In broad approach we follow \cite{Prz-5} but we stress the use of $X$-sets in our definitions.

\section{Homology of magmas}\label{Section 3}

We survey in this section  various homology theories, starting from homology of abstract simplicial complexes; 
then we extract (old and new) properties to define a presimplicial module and a (weak) simplicial module.
Further we give two examples of homology for associative structures (semigroups), and,  an important 
in knot theory, example
of homology for right self-distributive structures (RDS or shelves). Later we go back to very general notion of
homology of a small category with coefficient in a functor to $R$-Mod, and recall the notion of a 
geometric realization in the case of a (pre)simplicial set, and (pre)cubic set. Reader interested only in distributive 
homology can go directly to Section \ref{Section 6}.

\subsection{Homology of abstract simplicial complexes}\label{Subsection 3.1}
\
Our goal is to introduce homology of distributive magmas but to keep a historical 
perspective we start with the standard (oriented and ordered) homology 
of abstract simplicial complexes as they provide the framework for all homology we consider.

\begin{definition}\label{Definition 3.1}
 The abstract simplicial complex ${\mathcal K}=(V,P)$ is a pair of sets where $V=V({\mathcal K})$ is called a set of
vertices and $P({\mathcal K})=P \subset 2^V$, called the set of simplexes of ${\mathcal K}$ and it satisfies:
 elements of $P$ are finite subsets of $V$, include all one-element subsets, and if $s'\subset s \in P$ then also $s'\in P$
(that is a subsimplex of a simplex is a simplex).\footnote{ Usually we do not allow $\emptyset$ as a simplex,
but in some situations it is convenient to allow also an empty simplex, say of dimension $-1$; 
it will lead naturally to augmented chain complexes.}
 A simplex of $n+1$ vertices is called $n$-dimensional simplex, or succinctly, $n$-simplex 
( we write $s=\{v_{i_0},v_{i_1},...,v_{i_n}\}$).
We define $dim ({\mathcal K})$ as the maximal dimension
of a simplex in ${\mathcal K}$ (may be $\infty$ if there is no bound).

We consider, additionally, the category of abstract simplicial complexes with a class of objects 
composed of abstract simplicial complexes. $Mor(K_1,K_2)$ is the set of maps from $V(K_1)$ to $V(K_2)$ 
which send a simplex to a simplex (that is if $f\in Mor(K_1,K_2)$, $s\in P(K_1)$ then $f(s)\in P(K_2)$).
\end{definition}

We recall here three classical (equivalent) definitions of a homology of an abstract simplicial complex:
ordered, normalized ordered, and oriented.
\begin{definition}\label{Definition 3.2}
Recall that a chain complex  $\mathcal C$ is a sequence of modules over a fixed ring $k$ 
(here always commutative with identity),
$${\mathcal C}:\ \  ...\stackrel{\partial_{n+2}}{\longrightarrow} C_{n+1} \stackrel{\partial_{n+1}}{\longrightarrow} 
C_n \stackrel{\partial_{n}}{\longrightarrow} C_{n-1} \stackrel{\partial_{n-1}}{\longrightarrow} ... $$
such that $\partial_n\partial_{n+1}=0$ (succinctly $\partial^2=0$). 
Thus we have $im \partial_{n+1}\subset ker \partial_{n}$, 
and we define homology $H_n({\mathcal C})=ker \partial_{n}/im\ \partial_{n+1}$.\\
 Now for an abstract simplicial complex $K=(V,P)$ one defines:
\begin{enumerate}
\item[(I)](Ordered homology)\\
Consider a chain complex ${\mathcal C}^{ord}$ with $k$-modules $C^{ord}_n= C^{Ord}_n({\mathcal C})$ a submodule
of $ kV^{n+1}$ generated by all sequences
$(x_0,x_1,...,x_n)$, allowing repetitions, such that the set $\{x_0,x_1,...,x_n\}$ is a simplex in $P$ (possibly 
of dimension smaller from $n$).
The boundary operation is given on the basis by:
$$\partial (x_0,x_1,...,x_n) = \sum_{i=0}^n (-1)^id_i=\sum_{i=0}^n (-1)^i (x_0,...,x_{i-1},x_{i+1},...,x_n).$$
The ordered homology of ${\mathcal K}$ is defined $$H^{ord}_n({\mathcal K},k)=ker \partial_{n}/im\ \partial_{n+1}.$$
If $k=\Z$ we write $H^{ord}_n({\mathcal K})$.\\
Notices, that the maps $d_i: C^{ord}_n \to C^{ord}_{n-1}, (0\leq i \leq n)$, $d_i(x_0,x_1,...,x_n)= 
(x_0,...,x_{i-1},x_{i+1},...,x_n)$, called the face maps, satisfy:
$$ \mbox{(1) }d_id_j=d_{j-1}d_i \mbox{ for any } i < j.$$
The system $(C_n,d_i)$ satisfying the above equality is called a presimplicial module\footnote{The
concept was introduced in 1950 by Eilenberg and Zilber under the name {\it semi-simplicial complex}, \cite{E-Z}.}
 and
if we limit ourselves to $(V^{n+1},d_i)$ it is called a presimplicial set (compare Definition \ref{Definition 3.3}).
The important basic observation is that if $(C_n,d_i)$ is a presimplicial module then $(C_n,\partial_n)$, for 
$\partial_n=\sum_{i=0}^n(-1)^id_i$, is a chain complex. 
\\
Motivation for the boundary operation:\\
it is coming from the geometrical realization of an abstract
simplicial complex as illustrated below (the general setting of geometric realization of a simplicial set 
is discussed in Section \ref{Section 13}):\

$$\partial(x_0,x_1,x_2)= \partial({\parbox{2.7cm}{\psfig{figure=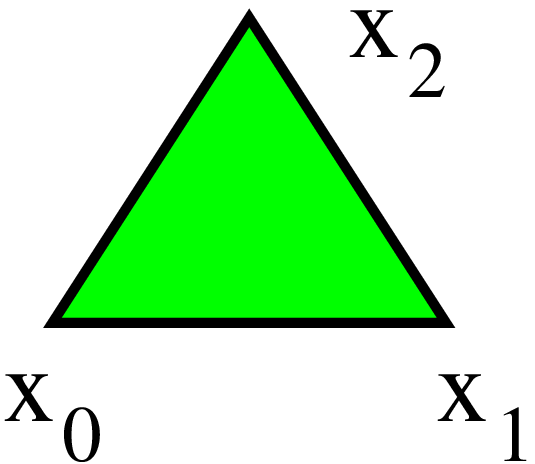,height=2.5cm}}}) =
{\parbox{2.7cm}{\psfig{figure=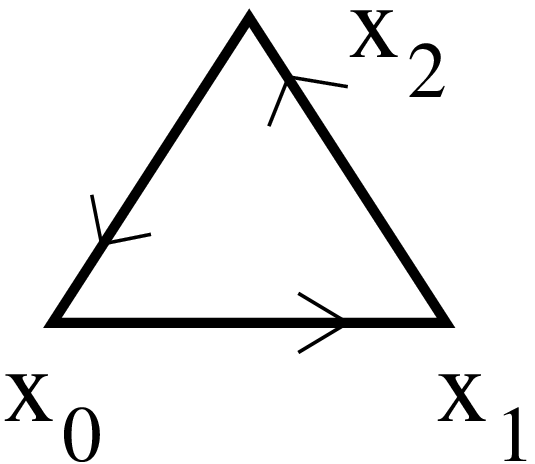,height=2.5cm}}} = $$
$$  (x_1,x_2)-(x_0,x_2)+(x_0,x_1).$$
 

\item[(II)] (Normalized ordered homology).\\
The ordered chain complex allows degenerate simplexes (when vertices repeat, in particular neighboring vertices repeat). 
We define $i$-degeneracy maps $s_i: C_n \to C_{n+1}$ ($0 \leq i \leq n$) by $s_i(x_0,...,x_n)=
(x_0,...,x_{i-1},x_i,x_i,x_{i+1},...,x_n)$. We can check here easily that:
$$\mbox{(2) } s_is_j=s_{j+1}s_i,\ \ 0\leq i \leq j \leq n, $$
$$ \mbox{(3) } d_is_j= \left\{ \begin{array}{rl}
 s_{j-1}d_i &\mbox{ if $i<j$} \\
s_{j}d_{i-1} &\mbox{ if $i>j+1$}
       \end{array} \right.
$$
$$ \mbox{(4) } d_is_i=d_{i+1}s_i= Id_{C_n}. $$
$(C_n,d_i,s_i)$ satisfying properties (1)-(4) is
called a simplicial module. The notion was introduced by Eilenberg and Zilber in 1950 under the name of complete
semi-simplicial complex \cite{E-Z}. It is convenient to rephrase the definition so it can be used to any
category:
\begin{definition}\label{Definition 3.3}
Consider a category $\mathcal C$, the sequence of objects $X_n$, $n\geq 0$ and for any $n$  morphisms
$d_i,s_i$, $0\leq i \leq n$, $d_i \in Mor(X_n, X_{n-1})$, and $s_i \in Mor(X_n, X_{n+1})$.
We call $(X_n;d_i,s_i)$ a simplicial category (e.g. simplicial set, simplicial module, or simplicial space)
 if the following four conditions hold.
\begin{enumerate}
\item[(1)] $d_id_j=d_{j-1}d_i$ for $i < j$,
\item[(2)] $s_is_j=s_{j+1}s_i$ for $i\leq j$,
\item[(3)] $$ s_is_j=s_{j+1}s_i,\ \ 0\leq i \leq j \leq n, $$
$$  d_is_j= \left\{ \begin{array}{rl}
 s_{j-1}d_i &\mbox{ if $i<j$} \\
s_{j}d_{i-1} &\mbox{ if $i>j+1$}
       \end{array} \right.
$$
\item[(4)] $d_is_i=d_{i+1}s_i= Id_{X_n}. $
\end{enumerate}
\end{definition}
Eilenberg and Mac Lane proved in 1947 that the degenerate part of a presimplicial module is an acyclic
chain complex (it has trivial homology) \cite{E-M-1} (the prove was more specific but the method worked for
all simplicial modules defined only 3 years later). We devote Subsection \ref{Subsection 3.2} to the proof,
 after \cite{Lod-1}, paying attention to which axioms of a simplicial module are used. In particular,
axiom (4) cannot be ignored as the degenerate chain complex of quandles, which satisfies the property (4)
only partially has often nontrivial homology.

Now back to normalized ordered homology: \\
Consider submodules $C^D_n({\mathcal K}$ (named degenerated modules) and defined by
$$C^D_n= span(s_0(C_{n-1}),s_1(C_{n-1}),...,s_{n-1}(C_{n-1})).$$ 
One check that $(C^D_n,\partial_n)$ is a subchain complex of $C^{ord}_n({\mathcal K})$. Details are given in 
Subsection \ref{Subsection 3.2}, were it is also proved that this chain complex is acyclic. 
We have also quotient chain complex, called normalized ordered chain complex with $C^N_n({\mathcal K})=
C_n({\mathcal K})/C^D_n({\mathcal K})$. As homology groups of $C^D_n({\mathcal K})$ are trivial we have isomorphism:

$$H^{ord}_n({\mathcal K},k) \to H^{N}_n({\mathcal K},k).$$

\item[(III)](Oriented homology).
We can consider smaller chain complex giving the same homology of $\mathcal K$ by taking the quotient of 
$C^N({\mathcal K})$ and considering only ``oriented simplexes". Formally, let $\bar C_n({\mathcal K})$ be 
a submodule of $C^N({\mathcal K})$ generated by ``transposition symmetrizers" 
$(x_0,...x_{i-1},x_i,x_{i+1},x_{i+2},...,x_n)+ (x_0,...x_{i-1},x_{i+1},x_i,x_{i+2},...,x_n)$.
One checks directly that  $(\bar C_n({\mathcal K}),\partial_n)$ is a subchain complex of $C^N({\mathcal K})$.
The oriented chain complex is the quotient: $C^{ori}({\mathcal K})= C^N({\mathcal K})/\bar C_n(({\mathcal K})$.
It require some effort to prove that the  quotient map $f: C^N({\mathcal K}) \to C^{ori}({\mathcal K})$ is 
a chain equivalence and thus $f_*: H^N({\mathcal K}) \to H^{ori}({\mathcal K})$ is is an isomorphism 
of homology modules. From this we conclude that all three definitions, ordered, normalized ordered, and 
oriented of homology of an abstract simplicial complex give the same result.

To have more concrete view of oriented chain complex and oriented homology we
  order vertices $V$ of ${\mathcal K}$ and interpret
 the chain group $C^{ori}_n({\mathcal C})$ as a subgroup of $ZV^{n+1}$ freely generated by
$n$-dimensional simplexes, $(x_0,x_1,...,x_n)$ (we assume that $x_0<x_1<...<x_n$ in our
ordering). In essence, with ordering, we are able to choose representatives of equivalence classes in 
$C^{ori}_n({\mathcal K})$ and the boundary operation
$$\partial (x_0,x_1,...,x_n) = \sum_{i=0}^n (-1)^id_i \mbox{ where }$$
$$ d_i(x_0,x_1,...,x_n) = (x_0,...,x_{i-1},x_{i+1},...,x_n),$$
preserves our choice so with given ordering we have a split chain map $g: C^{ori}_n({\mathcal K})\to C^N_n({\mathcal K})$.
The quotient map, with our ordering can be written as $f_n(x_0,...,x_n)=
(-1)^{|\pi|}(x'_0, ..., x'_n)$ where $\pi \in S_{n+1}$ is the permutation such that $x'_i=x_{\pi(i)}$ and $x'_i<x'_{i+1}$
(if $x_i=x_j$ for some $i\neq j$ then we put $f(x_0,...,x_n)=0$). Immediately, we have $f_ng_n=Id_{C^{ori}_n({\mathcal K})}$.
The proof that $gf$ is chain homotopic to identity on $C^{N}_n({\mathcal K})$ (and so $f$ is chain equivalence)  
requires more effort.\footnote{The standard 
Eilenberg-Mac Lane proof uses acyclic modules method \cite{E-M-2,Spa}, however in our case one can give shorter proof
(the idea is still that of Eilenberg-Mac Lane):
consider the chain map $f: {\mathcal C}^N \to {\mathcal C}^{ori}$ given by $f_n(x_0,...,x_n)=
(-1)^{|\pi|}(x'_0, ..., x'_n)$ where $\pi \in S_{n+1}$ is the permutation such that $x'_i=x_{\pi(i)}$ and $x'_i<x'_{i+1}$ 
(if $x_i=x_j$ for some $i\neq j$ then we put $f(x_0,...,x_n)=0$). The map $g: {\mathcal C}^{ori} \to {\mathcal C}^N$ 
is defined to be embedding; therefore $fg=Id_{\mathcal C}$. 
We show that $gf$ induces identity on homology of ${\mathcal C}^N$.
We construct a chain homotopy between $gf$ and the identity inductively, starting from $F_0=0$. The main 
ingredient of the proof is the fact that for a simplex $s=(x_0, ...,x_n)$ the subchain complex $\bar s=(s,2^s)$ 
 of ${\mathcal C}^N$ is acyclic ($H^N_n(\bar s)=0$ for $n>0$ and $H^N_0(\bar s) = k$).\\
Step n: assume that $F_{n-1},...,F_0$ are already constructed and we construct a map
  $F_n:C^N_n \to C^N_{n+1}$  such that $\partial_{n+1} F_n = -F_{n-1}\partial_n + Id -gf$.

We compute: that $$\partial_n(-F_{n-1}\partial_n + Id -gf)= -(\partial_nF_{n-1})\partial_n +\partial_n -\partial_n(gf)=$$
$$F_{n-2}\partial_{n-1}\partial_n -\partial_n + (gf)\partial_n \partial_n -\partial_n(gf)= 0.$$
Because chain complex $ C^N_n(\bar s)$ is exact at place $n$ and $-F_{n-1}\partial_n + Id -gf$ is in 
the kernel of this chain complex, it is also in the image, say, 
$\partial_{n+1}c_{n+1}=-F_{n-1}\partial_n + Id -gf$. Then we declare 
$F_n(s)= c_{n+1}$. In fact here $c_{n+1}$ can be obtained by putting any, fixed, vertex of $s$ in front of
$(-F_{n-1}\partial_n + Id -gf)$. Such constructed $F_n$ satisfies 
$\partial_{n+1}F_n + F_{n-1}\partial_n= Id-gf$.
Our map is well defined as we constructed it on the basis of $C^N_n$.} 
\end{enumerate}
\end{definition}


\subsection{Degenerate subcomplex}\label{Subsection 3.2} 
Consider a presimplicial module $(C_n,d_i)$ with degenerate maps $s_i$. We define degenerate modules 
$$C^D_n= span(s_0(C_{n-1}),s_1(C_{n-1}),...,s_{n-1}(C_{n-1})).$$ 
We check which conditions are needed so that $(C^D_n,d_i)$ is a sub-presimplicial module of $(C_n,d_i)$.
We have:
$$\partial_ns_p =  \sum_{i=0}^{n+1}(-1)^id_is_p =  \sum_{i=0}^{p-1}(-1)^id_is_p +(-1)^p(d_ps_p- d_{p-1}s_p +
  \sum_{i=p+2}^{n+1}(-1)^id_is_p \stackrel{(3)}{=} $$ 
$$\sum_{i=0}^{p-1}(-1)^is_{p-1}d_i +(-1)^p(d_ps_p- d_{p+1}s_p + \sum_{i=p+2}^{n+1}(-1)^is_pd_{i-1} \stackrel{(4')}{=} $$
$$\sum_{i=0}^{p-1}(-1)^is_{p-1}d_i \sum_{i=p+2}^{n+1}(-1)^is_pd_{i-1} \in C^D_{n-1},$$ where the (4') is the 
condition:
$$(4')\  \ d_ps_p = d_{p+1}s_p \mbox{ for any $p\leq n$.}$$
If $(C_n,d_i,s_i)$ satisfies conditions (1),(2),(3), and (4') it is called a weak simplicial module \cite{Prz-5}.
As condition (2) was not use in calculation it is useful also to consider $(C_n,d_i,s_i)$ satisfying the 
condition (1),(3), and (4'), we call this a weak-pseudo-simplicial module.

We strengthen the above calculation by considering the sequence of modules 
$ F^i_n=span(s_0(C_{n-1}),s_1(C_{n-1}),...,s_{i}(C_{n-1}))$ and the filtration:
$$ 0\subset F^0_n \subset F^1_n \subset ... \subset F^{n-1}_n=C^D_n.$$
Our calculation gives $\partial_n(s_p(C_{n-1}))\subset span (s_{p-1}(C_{n-1}),s_{p}(C_{n-1}))$ and 
consequently:
\begin{corollary}\label{Corollary 3.14}
If $(C_n,d_i,s_i)$ is a weak-pseudo-simplicial module then  $\partial_n$ is filtration preserving, that 
is $\partial_n(F^p_n) \subset F^{p-1}_n$.
\end{corollary}

We will prove now the Eilenberg-Mac Lane theorem that the degenerate complex $(C^D_n,\partial_n)$ is acyclic,
watching on the way which axioms are used.  For a filtration $(F^p_n)$  the associated graded module
is defined to be $\{Gr^p_n= F^p_n/F^{p-1}_n\}$. We prove first that the chain complex $\{Gr^p_n\}$ is acyclic 
for any $p$.

\begin{lemma} Let $(C_n,d_i,s_i)$ satisfies the conditions (1), (2''), (3), and (4) where 
$$\mbox{(2'') \ } s_{p-1}s_{p-1} = s_ps_{p-1} \mbox{ for every $0 < p \leq n$}.$$ 
We call such $(C_n,d_i,s_i)$ a co-almost-simplicial module. Then the chain complex  
$\{Gr^p_n= F^p_n/F^{p-1}_n\}$ is acyclic, in particular $H_n(\{Gr^p_n\}) = 0$, and homology of 
$F^p_n$ and $F^{p-1}_n$ are isomorphic.

\end{lemma}
\begin{proof} The classical idea of Eilenberg and Mac Lane is to use the degenerate map $s_p$ as a chain homotopy 
(we follow \cite{Lod-1}):\\
It suffices to show that $(\partial s_p +s_p\partial)s_p =(-1)^ps_p $ modulo $s_{p-1}C_{n-1}$, so the 
map $s_p$ is a chain homotopy between $(-1)^pId$ and the zero map on $Gr^p_n$. 
In the calculation we stress which axioms are used:
$$(\partial s_p +s_p\partial)s_p = (\sum_{i=0}^{n+1}(-1)^id_is_p +s_p\sum_{i=0}^{n}(-1)^id_i)s_p=$$
$$ \sum_{i=0}^{p-1}(-1)^id_is_ps_p +(-1)^p(d_ps_p -d_{p+1}s_p)s_p +  \sum_{i=p+2}^{n+1}(-1)^id_is_ps_p +$$
$$\sum_{i=0}^{p-1}(-1)^is_pd_is_p + (-1)^ps_p(d_ps_p -d_{p+1}s_p) +\sum_{i=p+2}^{n}(-1)^is_pd_is_p  \stackrel{(3)}{=}$$
$$ \sum_{i=0}^{p-1}(-1)^is_{p-1}s_{p-1}d_i + (-1)^p(d_ps_p -d_{p+1}s_p)s_p + \sum_{i=p+2}^{n+1}(-1)^is_pd_{i-1}s_p +$$
$$\sum_{i=0}^{p-1}(-1)^is_ps_{p-1}d_i + (-1)^ps_p(d_ps_p -d_{p+1}s_p) +\sum_{i=p+2}^{n}(-1)^is_pd_is_p  = $$
$$\sum_{i=0}^{p-1}(-1)^is_{p-1}s_{p-1}d_i + (-1)^p(d_ps_p -d_{p+1}s_p)s_p + (-1)^ps_pd_{p+1}s_p +$$
$$\sum_{i=0}^{p-1}(-1)^is_ps_{p-1}d_i + (-1)^ps_p(d_ps_p -d_{p+1}s_p) \stackrel{d_ps_p=d_{p+1}s_p}{=}$$
$$\sum_{i=0}^{p-1}(-1)^is_{p-1}s_{p-1}d_i + \sum_{i=0}^{p-1}(-1)^is_ps_{p-1}d_i +(-1)^ps_pd_{p+1}s_p \stackrel{(2)}{=}$$
$$2\sum_{i=0}^{p-1}(-1)^is_{p-1}s_{p-1}d_i +(-1)^ps_pd_{p+1}s_p \stackrel{\mod 2s_{p-1}M_{n-1}}{=}$$
$$(-1)^ps_pd_{p+1}s_p \stackrel{d_{p+1}s_p=Id}{=} (-1)^ps_p $$
Now consider the short exact sequence $0 \to F^{p-1} \to  F^p \to  F^p/F^{p-1} \to 0$ 
and the corresponding long exact sequence of
homology: 
$$...\to H_{n+1}(F^p/F^{p-1}) \to H_n(F^{p-1}) \to H_n(F^p) \to H_n(F^p/F^{p-1})\to ...$$
Thus because homology of $ F^p/F^{p-1}$ is trivial we obtain isomorphism  $H_n(F^{p-1}) \to H_n(F^p)$.
In conclusion, by induction on $p$ we get the Eilenberg-Mac Lane result: $H^D_n(C)=0$.
\end{proof}
From our  proof follows that working modulo $2s_{p-1}M_{n-1}$ , e.g. modulo $2$, gives directly $H_n(F^p)=0$.
Also from axiom (2) we took only $s_ps_{p-1}=s_{p-1}s_{p-1}$ that is axiom (2'').

The above consideration do not work  for a distributive case (the axiom (4) usually does not hold as explained 
in Section \ref{Section 6} (see \cite{CPP,Pr-Pu-1,P-S}. 
We proved however that the degenerate part of quandle homology can be 
obtained from the normalized one via K\"unneth type formula, see \cite{Pr-Pu-2}).

\subsection{Bicomplex for a degenerate subcomplex $C^D_n$}\label{Subsection 3.6}
One more important observation follows from our calculations. If $(C_n,d_i,s_i)$ is a weak simplicial 
module\footnote{In fact a pseudo weak simplicial module suffices, i.e. conditions (1),(3),(4').}
 (i.e. conditions (1)-(3),(4') hold then the formula 
$$\partial_ns_p =\sum_{i=0}^{p-1}(-1)^is_{p-1}d_i + \sum_{i=p+2}^{n}(-1)^is_{p}d_{i-1}$$
allows us to define a bicomplex with entries $E^0_{p,q}=Gr_{n,p}=F_n^p/F_n^{p-1}$, $n=p+q$, with horizontal and 
vertical  boundary operation: $\partial_{p,q}^h= \sum_{i=0}^{p-1}(-1)^i$ and 
$\partial_{p,q}^v=\sum_{i=p+2}^{n}(-1)^is_{p}d_{i-1}$ with 
$\partial^h_{p,q-1}\partial^v_{p,q}=-\partial^v_{p-1,q}\partial^h_{p,q}: E^0_{p,q} \to E^0_{p-1,q-1}$; see 
Figure 3.1.

\begin{displaymath}
\begin{array}{ccccccl}
&\downarrow \partial^v& & \downarrow \partial^v  &  & \downarrow \partial^v & \\
\ldots \stackrel{d^h}{\leftarrow} & E^0_{p-1,q+1} & \stackrel{\partial^h}{\leftarrow} &E^0_{p,q+1} & 
\stackrel{\partial^h}{\leftarrow} & E^0_{p+1,q+1} & \stackrel{\partial^h}{\leftarrow}\ldots \\
 &\downarrow \partial^v& & \downarrow \partial^v  &  & \downarrow \partial^v & \\
\ldots \stackrel{\partial^h}{\leftarrow} & E^0_{p-1,q} & \stackrel{\partial^h}{\leftarrow} &E^0_{p,q} & 
\stackrel{\partial^h}{\leftarrow} &
E^0_{p+1,q} & \stackrel{\partial^h}{\leftarrow}\ldots \\
 &\downarrow \partial^v& & \downarrow \partial^v  &  & \downarrow \partial^v & \\
\ldots \stackrel{\partial^h}{\leftarrow} & E^0_{p-1,q-1} & \stackrel{\partial^h}{\leftarrow} & E^0_{p,q-1} & 
\stackrel{\partial^h}{\leftarrow} & E^0_{p+1,q-1} & \stackrel{\partial^h}{\leftarrow}\ldots \\
 &\downarrow \partial^v& & \downarrow \partial^v  &  & \downarrow \partial^v & \\
\end{array}
\end{displaymath}
\centerline{Figure 3.1; Bicomplex $(E^0_{p,q},\partial^v,\partial^h)$}

The bicomplex $(E^0_{p,q},\partial^v,\partial^h)$ yields a spectral sequence, in fact two
spectral sequences: starting from  columns, that is 
${}^cE^1_{pq}=\frac{\ker(\partial^v:E^0_{pq}\to E^0_{p,q-1})}{im (\partial^v:E^0_{p,q+1}\to E^0_{pq})}$,
and the spectral sequence starting from rows:
${}^rE^1_{pq}=\frac{\ker(\partial^h:E^0_{pq}\to E^0_{p-1,q})}{im(\partial^h:E^0_{p+1,q}\to E^0_{pq})}$
which can be used to analyze homology of $Gr(C^D)$ and $(C_n)$, see \cite{Pr-Pu-2} for an application in 
the distributive case.
 
\subsection{Homology with coefficients in a $k-$Mod functor}\label{Subsection 3.7}

Each individual abstract simplicial complex $K=(V,{\mathcal P})$ is a small category\footnote{Category is 
called small if objects form a set.}  with simplexes as objects 
and inclusions of simplexes, $s\subset s'$, as morphisms. As usually $K^{op}$ will denote the opposite category 
so restrictions, $s\supset s'$, are morphisms, more precisely $Mor_{K^{op}}(s,s')$ is empty 
if $s$ does not contain $s'$ and 
otherwise $Mor_{K^{op}}(s,s')$ has one morphism denoted by $(s\supset s')$. 
Now for any (covariant) functor ${\mathcal F}: K^{op}\to k-$Mod, where $k$-Mod is a category of 
modules over a commutative ring $k$, we can define oriented homology 
$H^{ori}_n(K,{\mathcal F})$ of an abstract simplicial complex $K$ with coefficients in ${\mathcal F}$,
as follows:
\begin{definition}\label{Definition 3.6}
 Let $K=(V,P)$ be an abstract simplicial complex with ordered vertices\footnote{It suffices 
to have $V$ partially ordered  as long as for any simplex $s=(x_0,...,x_n)$ the partial order on $V$ 
restricts to linear order on vertices of $s$. Even better we do not need a partial order, it suffices that 
vertices of every simplex are ordered in such a way that if $s_1 \subset s_2$ then the ordering of vertices 
of $s_1$ is a restriction of the ordering of vertices of $s_2$.}
 and ${\mathcal F}: K^{op}\to k-$Mod a functor. We define the presimplicial 
module $(C^{ori}_n(K,{\mathcal F}),d_i)$ as follows:
$$C^{ori}_n(K,{\mathcal F})= \bigoplus_{dim(s)=n}{\mathcal F}(s)$$
the face map $d_i:C^{ori}_n(K,{\mathcal F})\to C^{ori}_{n-1}(K,{\mathcal F})$ is defined by 
$$d_i = {\mathcal F}(s\supset (s-x_i)) \mbox{ where } s=(x_0,...,x_n) \mbox{ and } x_i < x_{i+1}$$
as usually for presimplicial modules $\partial_n= \sum_{i=1}^n(-1)^id_i$ and 
 $(C^{ori}_n(K,{\mathcal F}),\partial_i)$ is a chain complex whose homology is denoted by $H^{ori}_n(K,{\mathcal F})$.
\end{definition}
The above definition can be thought of as a twisted version of an oriented  homology of abstract simplicial complexes.
Similarly we can define ordered homology of $(K,{\mathcal F})$ but oriented and ordered homology with coefficient 
in a functor are not necessarily isomorphic.

Definition \ref{Definition 3.6} is related to more general Definition \ref{Definition 3.7} 
on homology of a small category with a functor coefficient,
 usually thought to be first given by \cite{Wat}, who in turn refers to the earlier paper \cite{Dehe} 
 in the case of the category of posets.

\begin{definition}\label{Definition 3.7}
Let ${\mathcal P}$ be as small category (i.e. objects, $P=Ob({\mathcal P)}$ form a set), 
and let ${\mathcal F}:{\mathcal P} \to k$-Mod
be a functor from ${\mathcal P}$ to the category of modules over a commutative ring $k$.
We call the sequence of objects and functors, $x_0 \stackrel{f_0}{\to} x_1 \stackrel{f_1}{\to} \ldots
\stackrel{f_{n-1}}{\to} x_n$ an $n$-chain  (more formally $n$-chain in the nerve of the category).
We define the chain complex $C_*({\mathcal P},{\mathcal F})$ as follows:
$$ C_n= \bigoplus_{x_0 \stackrel{f_0}{\to} x_1 \stackrel{f_1}{\to} \ldots \stackrel{f_{n-1}}{\to} x_n} {\mathcal F}(x_0)$$
where the sum is taken over all $n$-chains.

The boundary operation $\partial_n: C_n({\mathcal P},{\mathcal F}) \to C_{n+1}({\mathcal P},{\mathcal F})$ is an
alternative sum of  face maps, $\partial_n=\sum_{i=0}^n (-1)^id_i$, where $d_i$ are given by:
$$d_0(\lambda;x_0\stackrel{f_0}{\to} x_1 \stackrel{f_1}{\to} \ldots \stackrel{f_{n-1}}{\to} x_n)=
({\mathcal F}(x_0\stackrel{f_0}{\to} x_1)(\lambda);  x_1 \stackrel{f_1}{\to} \ldots \stackrel{f_{n-1}}{\to} x_n),$$
$$d_i(\lambda;x_0\stackrel{f_0}{\to} x_1 \stackrel{f_1}{\to} \ldots \stackrel{f_{n-1}}{\to} x_n)=
(\lambda;x_0\stackrel{f_0}{\to} x_1 \stackrel{f_1}{\to}\ldots \to
x_{i-1}\stackrel{f_if_{i-1}}{\to} x_{i+1}\to \ldots \stackrel{f_{n-1}}{\to} x_n)$$ for $0<i<n$, and 
$$d_n(\lambda;x_0\stackrel{f_0}{\to} x_1 \stackrel{f_1}{\to} \ldots \stackrel{f_{n-1}}{\to} x_n)=
(\lambda;x_0\stackrel{f_0}{\to} x_1 \stackrel{f_1}{\to}\ldots \stackrel{f_{n-2}}{\to} x_{n-1}).$$
We denote by $H_n({\mathcal P},{\mathcal F})$ the homology yielded by the above chain complex, and call this 
the homology of a small category ${\mathcal P}$ with coefficients in a functor ${\mathcal F}$.\\
Similarly, if ${\mathcal F}': {\mathcal P} \to k-$Mod is a contravariant functor we may define 
a homology $H_n({\mathcal P};{\mathcal F}')$, starting from 
$$C_n({\mathcal P};{\mathcal F}')= 
\bigoplus_{x_0 \stackrel{f_0}{\to} x_1 \stackrel{f_1}{\to} \ldots \stackrel{f_{n-1}}{\to} x_n} {\mathcal F}'(x_n).$$
In particular, $d_n(x_0\stackrel{f_0}{\to} x_1 \stackrel{f_1}{\to} \ldots \stackrel{f_{n-1}}{\to} x_n;\lambda)=\\
(x_0\stackrel{f_0}{\to} x_1 \stackrel{f_1}{\to}\ldots \stackrel{f_{n-2}}{\to} x_{n-1});
{\mathcal F}'(x_{n-1}\stackrel{f_{n-1}}{\to} x_n)(\lambda))$, where $\lambda \in {\mathcal F}'(x_n)$.\\
One can also consider both functors, ${\mathcal F}$ and ${\mathcal F}'$ in the definition starting from
$$C_n({\mathcal P};{\mathcal F},{\mathcal F}')= 
\bigoplus_{x_0 \stackrel{f_0}{\to} x_1 \stackrel{f_1}{\to} \ldots \stackrel{f_{n-1}}{\to} x_n} 
{\mathcal F}'(x_n)\otimes {\mathcal F}(x_0);$$
compare Definition 4.6. We can also start from from a bifunctor 
${\mathcal D}: {\mathcal P}^{op}\times {\mathcal P} \to k-Mod$ 
and mimic the definition of the Hochschild homology (Section 5) \cite{Lod-1}.

\end{definition}

\begin{remark}\label{Remark 3.8} 
Any subcategory ${\mathcal P}'$ of ${\mathcal P}$ has its chain complex, and homology (we use the
functor ${\mathcal F}'= {\mathcal F}/{\mathcal P'}$, that is, the restriction of ${\mathcal F}$ to ${\mathcal P'}$).
$C_*({\mathcal P'},{\mathcal F}')$ is a subchain complex of $C_*({\mathcal P},{\mathcal F})$
so we can consider the short exact sequence of chain complexes:
$$0 \to C_n({\mathcal P},{\mathcal F}) \to C_n({\mathcal P}',{\mathcal F}')\to 
C_n({\mathcal P},{\mathcal F})/C_n({\mathcal P}',{\mathcal F}') \to 0$$
and yielded by it the long exact sequence of homology. 
\end{remark}

The pair $(C_n,d_i$) form a presimplicial module by associativity of morphisms of a category and properties of 
a functor. More generally we have:
\begin{proposition}\label{Proposition 3.9}
 Let $s_i:C_n \to C_{n+1}$ be a map inserting identity morphism on the $i$th place in 
the $n$th chain of of the nerve of the category, that is 
$$s_i((\lambda;x_0\stackrel{f_0}{\to} \ldots \stackrel{f_{n-1}}{\to} x_n)= 
(\lambda;x_0\stackrel{f_0}{\to} x_i \stackrel{Id_{x_i}}{\to} x_i \ldots \stackrel{f_{n-1}}{\to} x_n).$$
Then $(C_n,d_i,s_i)$ is a simplicial module.
\end{proposition}

The classical example is the homology of a simplicial complex with constant coefficients, that
is ${\mathcal F}(s)= k$ and ${\mathcal F}(f)=Id_k$; in that case we write $H_n(K,{\mathcal F})= H_n(K,k)$ or just 
$H_n(K)$ if $k=\Z$. Related to this example is homology of posets: if ${\mathcal P}$ is a small category
and for any objects $x$ and $y$, $Mor(x,y)$ has at most one element and additionally if $Mor(x,y)\neq \emptyset$
and $Mor(y,x)\neq \emptyset$ then $x=y$, then ${\mathcal P}$ is a poset with $x\leq y$ if{f} $Mor(x,y)\neq \emptyset$.

Another classical example concerns homology of groups, where the category has one object and $G$ morphisms
(interpreted as multiplication by elements of $G$ \cite{Bro}, that is the morphism $g: G\to G$ 
is given by $g(h)=hg$); compare Section 4.

More recent example is motivated by Khovanov homology so we call a related functor ${\mathcal F}_{(D,A,M)}$, 
a Khovanov functor. The functor depends on a choice of an $k$-Frobenius algebra $A$, $A$-Frobenius bimodule $M$ and
a link diagram (possibly virtual link, or a link diagram on a surface, $L$ (equivalently we can
work with graphs on a surface). Here for simplicity we assume that $M=A$ is an abelian Frobenius 
$k$-algebra\footnote{$A$ is a $k$ module with associative and commutative multiplication, $\mu$, 
with co-associative and co-commutative co-multiplication, $\Delta$, satisfying the Frobenius 
condition, that is
$\Delta\mu =  (\mu \times Id)(Id\times \Delta)$; graphically: {\parbox{0.9cm}{\psfig{figure=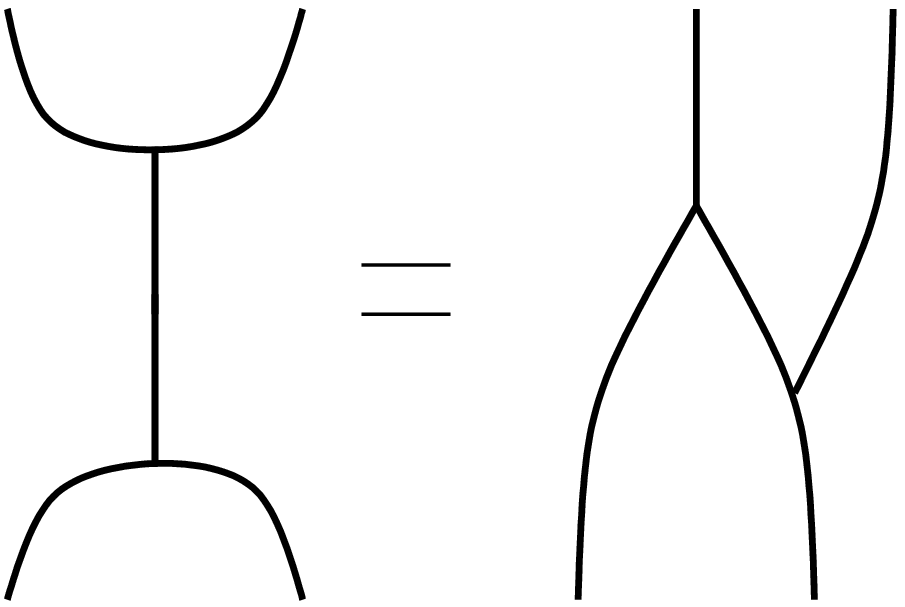,height=0.6cm}}}.
There is no need for unit and counit.}
and $D$ is a classical link diagram.
\begin{definition}\label{Definition 3.10}
For a link diagram $D$, let $V$ be the set of its crossings (in some order), and $P=2^V$.
Thus $K=(V,P)$ is a simplex (we allow also the empty, $-1$-dimensional simplex). Let $A$ be a Frobenius algebra with
a multiplication $\mu$ and a co-multiplication $\Delta$ (e.g. $A=Z[x]/(x^m)$, $\Delta(1)=\sum_{i+j=m-1}x^i\otimes x^j$).
We define a functor ${\mathcal F}_{D,A}: K \to k$-Mod as follows. For any $s\in P$ we identify $s$ with
a Kauffman state, where $s(v)=1$ (i.e. {\parbox{0.9cm}{\psfig{figure=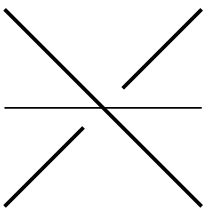,height=.5cm}}})
iff $v\in s$. We denote by $D_s$ the collection of circles obtained
from $D$ by smoothing it according to $s$, and by $|D_s|$ the number of circles in $D_s$.
Then we define ${\mathcal F}(s)= A^{\otimes |D_s|}$. To define ${\mathcal F}(s\supset (s-v_i))$ we first
decorate circles of $D_s$ by algebra $A$, (that is each circle by one copy of $A$); then we have two cases:\\
($\mu$) $|D_{s-v_i}|=  |D_s|-1$, thus two circles are glued together when we switch the state at $v_i$.
In this case we multiply the element associated with glued circles (we use the fact that $A$ is commutative).\\
($\Delta$) $|D_{s-v_i}|=  |D_s|+1$, thus a circle of $D_s$ is split so we apply to the element of $A$
associated to this circle a co-multiplication (we use the fact that $\Delta$ is co-commutative).
${\mathcal F}$ is a functor as $A$ is a commutative Frobenius algebra.
\end{definition}
This approach to Khovanov homology was first sketched in \cite{Prz-4}, where Khovanov homology was 
connected to Hochschild homology.

It is a classical result that homology of a baricentric subdivision of an abstract simplicial complex 
is isomorphic to the homology of the complex. This was an ingredient of original proofs of topological 
invariance of homology. The generalization of the result holds also for 
a homology of a simplicial complex $\mathcal K$ with a coefficient in a functor and the 
homology of $\mathcal K$ treated as a small category. I was informed by S.~Betley and Jolanta S{\l}omi\'nska
about at least three proof of the fact, compare \cite{Slo}.
We are writing, with my student Jing Wang detailed survey with the  proof following 
closely the classical proof with constant coefficients (in essence it is another case of acyclic model theorem 
of Eilenberg and Zilber \cite{E-Z}).


In the next few sections we discuss homology related to various magmas (e.g. associative and right distributive)
and look for the common traits, for example presimplicial or simplicial module structure, geometric 
realization etc.

\section{Group homology of a semigroup}\label{Section 4.1}

In the homology of abstract simplicial complexes, the set of vertices, $X$, has no algebraic structure
or, as we will see later, we can treat $X$ as a magma with the trivial operation $*_0$ $, x*_0y=x$. 
We will now discuss homology of magma $(X,*)$ equipped with some specific structure, e.g. associativity,
Jacobi identity, or distributivity.

According to \cite{Bro}: {\it The cohomology theory of groups arose from both topological and algebraic sources.
The starting point for the topological aspect of the theory was the work of Hurewicz (\cite{Hur}, 1936
on ``aspherical spaces". About a year earlier, Hurewicz had introduced the higher homotopy groups
$\pi_n X$ of a space $X$ ($n\geq 2$). He now singled out for study those path-connected spaces $X$
whose higher homotopy groups are all trivial, but whose fundamental group
$\pi=\pi_1X$ need not be trivial. Such spaces are called aspherical. Hurewicz proved, among other things,
that the homotopy type of an aspherical pace $X$ is completely determined by its fundamental group $\pi$.
... Hopf (\cite{Hop}, 1942)... expressed $H_2\pi$ in purely algebraic terms...}

Let $(X,*)$ be a semigroup that is a set with associative binary operation. We associate with $(X,*)$ 
a presimplicial set, presimplicial module, chain complex, group homology and geometric realization as follows:

\begin{definition}\label{Definition 4.1}
\begin{enumerate}
\item[(i)] Let $X_n=X^n$ and $d_i:X_n \to X_{n-1}$ for $0\leq i \leq n $ is given by:
$$d_0(x_1,x_2,...,x_n)=(x_2,,...,x_n),$$
$$d_i(x_1,...,x_n)= (x_1,...,x_{i-1},x_i*x_{i+1},x_{i+2},...,x_n) \mbox{ for } 0<i<n,$$
$$d_0(x_1,...,x_{n-1},x_n)=(x_1,,...,x_{n-1}).$$
Then $(X_n,d_i)$ is a presimplicial set.
\item[(ii)] If we choose a commutative ring $k$ and consider $C_n=kX^n$ and $d_i: C_n \to C_{n-1}$ the unique 
extension of the map $d_i$ from (i) then $(C_n,d_i)$ is a presimplicial module.
\item[(iii)] If $\partial_n = \sum_{i=0}^n(-1)^id_i$, then $(C_n,\partial_n)$ is a chain complex;
its homology are called group homology of a semigroup $X$ and denoted by $H_n(X;k)$ or just 
$H_n(X)$ if $k=\Z$.
\item[(iv)] A presimplicial set has a standard geometric realization, $BX$ (as a CW-complex\footnote{$BX$ can be 
made into geometric simplicial complex by second baricentric subdivision because $BX$ by the construction 
is glued from simplexes (such a space is called a $\Delta$-complex in \cite{Hat}), 
see Section \ref{Section 13}, e.g. Definition \ref{Definition 13.1}}).
Thus the semigroup homology has a natural interpretation as a homology of a CW-complex \cite{Bro}.
\end{enumerate}
\end{definition}
Definition \ref{Definition 4.1} has a classical generalization, when a semigroup $(X;*)$ is 
augmented by  an $X$-right-semigroup-set $E$,
that is a set with the  right action (also denoted by $*$) of $X$ on $E$ such that $(e*a)*b=e*(a*b)$.
\begin{definition}\label{Definition 4.2}
\begin{enumerate}
\item[(i)] Let $X_n= E\times X^n$ and $d_i:X_n \to X_{n-1}$ for $0\leq i \leq n $ is given by:
$$d_0(e,x_1,x_2,...,x_n)=(e*x_1,x_2,,...,x_n),$$
$$d_i(e,x_1,...,x_n)= (e,x_1,...,x_{i-1},x_i*x_{i+1},x_{i+2},...,x_n) \mbox{ for } 0<i<n,$$
$$d_n(e,x_1,...,x_{n-1},x_n)=(e,x_1,,...,x_{n-1}).$$
Then $(X_n,d_i)$ is a presimplicial set.
\item[(ii)] For $C_n= k(E\times X^n)$, $(C_n,d_i)$ is a presimplicial module and $(C_n,\partial_n)$ is a chain complex
with homology denoted by $H_n(X,E)$ and geometric realization $B(X,E)$.
\end{enumerate}
If $E$ has one element then we get the case of Definition \ref{Definition 4.1}.
\end{definition}
Definition \ref{Definition 4.2} has further generalization if, in addition to a semigroup $(X;*)$ we have 
the right $X$-set $E_0$ and the left $X$ set $E_w$ (here we need ($a*(b*e)=(a*b)*e$), compare \cite{Ca-E}, Chapter X.
\begin{definition}\label{Definition A3.6}
\begin{enumerate}
\item[(i)] Let $X_n= E_0\times X^n\times E_w$ and $d_i:X_n \to X_{n-1}$ for $0\leq i \leq n $ is given by:
$$d_0(e_0,x_1,x_2,...,x_n,e_{n+1})=(e*x_0,x_2,,...,x_n,e_{n+1})),$$
$$d_i(e_0,x_1,...,x_n,e_{n+1})= (e_0,x_1,...,x_{i-1},x_i*x_{i+1},x_{i+2},...,x_n,e_{n+1}) \mbox{ for } 0<i<n,$$
$$d_n(e_0,x_1,...,x_{n-1},x_n,e_{n+1})=(e_0,x_1,...,x_{n-1},x_n*e_{n+1}).$$
Then $(X_n,d_i)$ is a presimplicial set. We call this pre-simplicial set a ``two walls" presimplicial set due to 
visualization in Figure 4.1.
\item[(ii)] For $C_n= k(E_0\times X^n\times E_w)$, $(C_n,d_i)$ is a presimplicial module 
and $(C_n,\partial_n)$ is a chain complex
with homology denoted by $H_n(X,E_0,E_w)$, and geometric realization $B(X,E_0,E_w)$.

\end{enumerate}
If $E_w$ has one element then we get the case of Definition \ref{Definition 4.2}
\end{definition}
\centerline{\psfig{figure=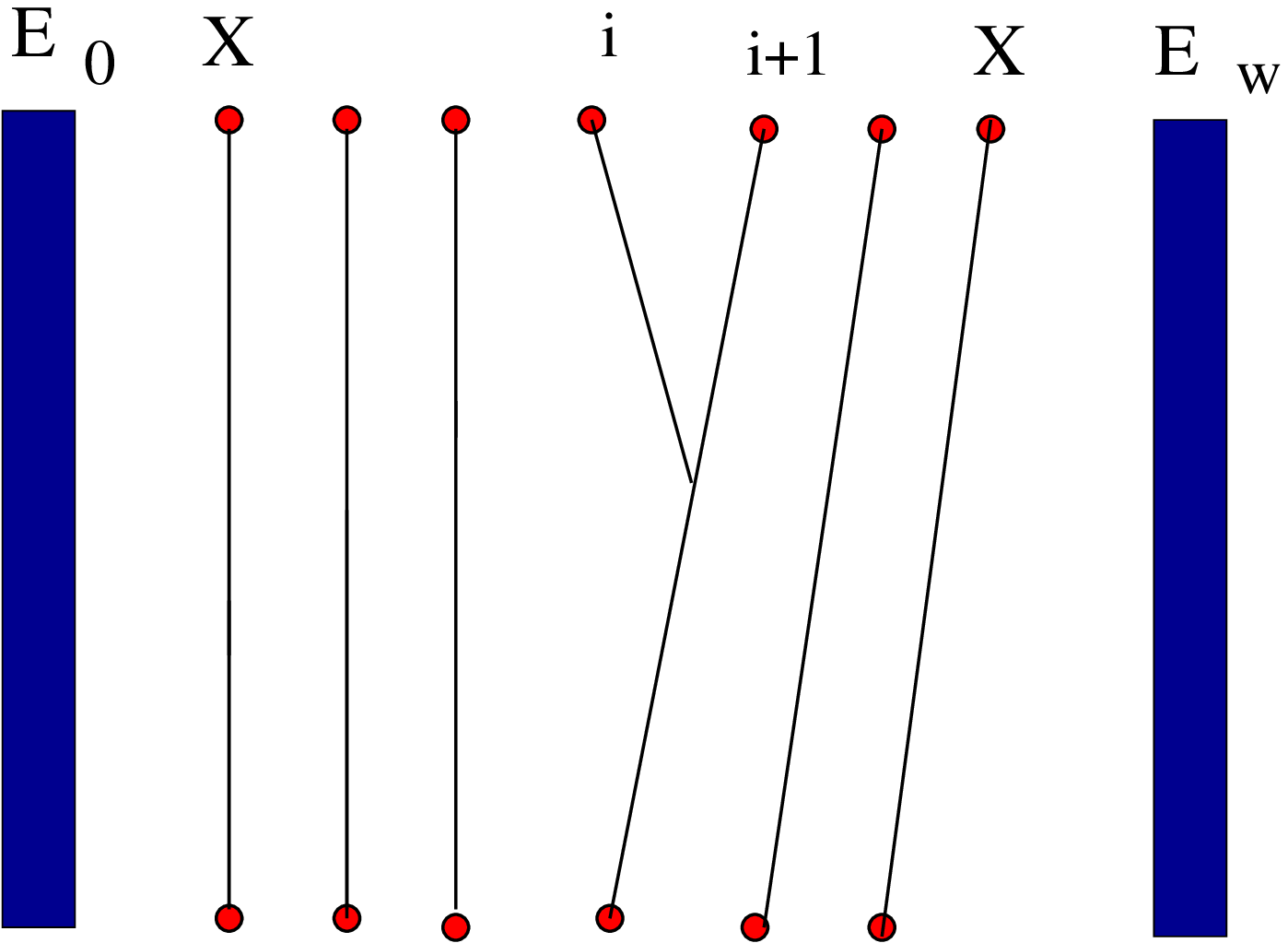,height=3.8cm}}\ \ \\
\centerline{Figure 4.1; i'th face map in ``two walls" presimplicial set for a semigroup}\ \\ \ \

A version of Definition \ref{Definition A3.6} when we assume that $E_0=E_w$ and $(e_1*x)*e_2=e_1*(x*e_2)$ that 
is $E_0$ is an $X$-biset leads to the Hochschild homology. In particular, for a semigroup $(X;*)$ we have:
\begin{definition}\label{Definition A3.7}
Let $(X;*)$ be a semigroup and $E$ an $X$-biset then:
\begin{enumerate}
\item[(i)]  Let $X_n= E\times X^n$ and $d_i:X_n \to X_{n-1}$ for $0\leq i \leq n $ is given by:
$$d_0(e,x_1,x_2,...,x_n)=(e*x_1,x_2,,...,x_n),$$
$$d_i(e,x_1,...,x_n)= (e,x_1,...,x_{i-1},x_i*x_{i+1},x_{i+2},...,x_n) \mbox{ for } 0<i<n,$$
$$d_n(e,x_1,...,x_{n-1},x_n)=(x_n*e,x_1,,...,x_{n-1}).$$
Then $(X_n,d_i)$ is a presimplicial set.
\item[(ii)] For $C_n= k(E\times X^n)$, $(C_n,d_i)$ is a presimplicial module and $(C_n,\partial_n)$ is a chain complex
with (Hochschild) homology denoted by $HH_n(X,E)$ and geometric realization $BH(X,E)$.
\item[(iii)] If $E_0$ is an $X$-right-semigroup-set and $E_w$ is an $X$-right-semigroup-set
then we can take $E= E_w\times E_0$ and $E$ has a natural structure of an $X$-biset.
Thus the concepts of a ``two-wall" semigroup homology and Hochschild semigroup homology are equivalent.
\end{enumerate}
\end{definition}

If $(X;*)$ is a monoid (with a unit element $1$) then we say that the set $E$ is $X$-right-monoid-set 
if it is  $X$-right-monoid-set and additionally $e*1=e$ for any $e\in E$ (that is $1$ acts trivially 
on $E$ from the right. Similarly we define $X$-left-monoid-set (e.g. $1*e=e$).
For a monoid  the presimplicial sets (modules) described in Definitions \ref{Definition 4.2}, \ref{Definition A3.6}, 
and \ref{Definition A3.7}
are in fact simplicial sets (modules) with the degeneracy maps $s_i$ placing $1$ between $x_i$ and $x_{i+1}$,
for example  in the case of \ref{Definition A3.6}:
$$s_0(e_0,x_1,x_2,...,x_n,e_{n+1})=(e_0,1,e_0,x_1,x_2,...,x_n,e_{n+1}),$$
$$s_i(e_0,x_1,x_2,...,x_n,e_{n+1})= (e_0,x_1,...,x_i,1,x_{i+1},...,x_n,e_{n+1}) \mbox{ for } 0<i<n,$$
$$s_n(e_0,x_1,x_2,...,x_n,e_{n+1})= (e_0,x_1,x_2,...,x_n,1,e_{n+1}).$$
 
\begin{example}\label{Example 4.5}
We can check that $ d_id_{i+1}=d_id_i$ ($0< i < n$)  if and only if $*$ is associative.\\
Furthermore, $d_0d_1=d_0d_0$ iff $(e_0*x_1)*x_2=e_0*(x_1*x_2)$ that is $E_0$ is an $X$-right-semigroup-set.\\
Similarly $d_{n-1}d_n=d_{n-1}d_{n-1}$ iff $E_w$ is an $X$-left-semigroup-set.
\end{example}

Let $\partial^{(\ell)}$ be a boundary map obtained from the group homology boundary operation by dropping
the first term from the sum. Analogously, let $\partial^{(r)}$ be a boundary map obtained
from the group homology boundary operation by dropping the last term from the sum. It is a classical
observation that $(C_n,\partial^{(\ell})$ and $(C_n,\partial^{(r)})$ are acyclic for a group (or a monoid).
We show this in a slightly more general context of weak simplicial modules (used later in the distributive case)
in Section 6. 
It would be of 
interest to analyze homology of $(C_n,\partial^{(\ell})$ for a semigroup without identity. Can it have a torsion?.

Our definition (in the presented form) can be generalized to any $k$-algebra $V$ not only $V=kX$
Below we give the definition for ``two wall" $k$-algebra, and in the next section 
we describe the mainstream Hochschild homology of $k$-algebra 
closely related to group homology.
\begin{definition}\label{Definition 4.6}
Let $A$ be a $k$-algebra which acts from the right on a $k$-module $M_0$ ($(m*x)*y=m*(x*y)$) and 
from the left on a $k$-module $M_w$ ($x*(y*m)=(x*y)*m$), that is $M_0$ is a right $A$-module and
$M_w$ a left $A$-module. We define chain groups $C_n= M_0\otimes A^{\otimes n }\otimes M_w$ and face 
maps $d_i(x_0,x_1,...,x_n,x_{n+1})=(x_0,...,x_i*x_{i+1},...,x_{n+1})$, $0\leq i \leq n$, $x_0\in M_0$, $x_{n+1}\in M_w$,
and $x_i\in A$ for $0<i \leq n$. Then $(C_n,d_i)$ is a presimplicial module and $(C_n,\partial_n)$, with 
$\partial_n=\sum_{i=0}^n(-1)^id_i$ is a chain complex, whose homology is denoted by $H_n(A,M_0,M_w)$.
If $A$ is a unitary algebra (with unit $1$) then we define degenerate maps $s_i(x_0,x_1,...,x_n,x_{n+1})=
(x_0,x_1,...,x_i,1,x_{i+1},...,x_n,x_{n+1})$, $0\leq i \leq n$, and one checks directly that 
$(C_n,d_i,s_i)$ is a simplicial module.

If we glue together $M_0$ and $M_w$ to get 2-sided module ($A$-bimodule) $M=M_w\otimes M_0$ we obtain 
Hochschild homology $H_n(A,M)$, \cite{Hoch,Lod-1}; see Section \ref{Section 5}.
\end{definition}

\section{Hochschild homology of a semigroup and an algebra}\label{Section 5}

Hochschild homology was created to have a homology theory of algebras, as before homology was defined 
only for (semi)groups, $G$, and (semi)group algebras $kG$ (Definition \ref{Definition 4.6} is only 
afterthought with Hochschild homology in mind). The history of discovering homology for algebra is 
described in Mac Lane autobiography \cite{Mac}: {\it Given his topological background and enthusiasm, Eilenberg 
was perhaps the first person to see this clearly. He was in active touch with Gerhard Hochschild, who was 
then a student of Chevalley at Princeton. Eilenberg suggested that there ought to be a cohomology (and 
a homology) for algebras. This turned out to be the case, and the complex used to describe the cohomology of groups 
(i.e. the bar resolution) was adapted to define the Hochschild cohomology of algebras.}  

Nevertheless, we start from Hochschild homology of semigroups as it leads to a presimplicial set, while 
the general Hochschild homology gives a presimplicial module.

Let $(X;*)$ be a semigroup and $E$ a two sided $X$-semigroup-set that is 
$(e*a)*b=e*(a*b)$, $(a*e)*b= a*(e*b)$, and $(a*b)*e=a*(b*e)$
 We define a Hochschild presimplicial module $\{C_n(X,E),d_i\}$ as follows \cite{Hoch,Lod-1}:
$C_n(X)= k(E\times X^{n})$ and the Hochschild face map is given by $d_i:k(E\times X^n) \to k(E\times X^{n-1}$
where $d_0(e_0,x_1,...x_n)= (e_0*x_1,x_2,...,x_n)$,\\
$d_i(e_0,x_1,...x_n)=(e_0,x_1,...,x_{i-1},x_i*x_{i+1},...,x_n$ for $0<i<n$, and \\
$d_n(e_0,x_1,...x_n)= (x_n*e_0,x_1,...,x_{n-1}).$\\
$\partial_n : \Z X^n \to \Z X^{n-1}$ is defined by:
$$\partial(x_0,x_1,...x_n)= $$
$$\sum_{i=0}^{n-1}(-1)^i(x_0,...,x_{i-1},x_i*x_{i+1},x_{i+2},...,x_n) +$$
$$(-1)^n(x_n*x_{0},x_1,...x_{n-1})$$

The resulting homology is called the Hochschild homology of a semigroup $(X,*)$
and denoted by $H\!H_n(X)$ (introduced by Hochschild in 1945 \cite{Hoch}).
It is useful to define $C_{-1}=\Z$ and define $\partial_0(x)=1$ to obtain the augmented Hochschild
chain complex and augmented Hochschild homology.

Again if $(X,*)$ is a monoid then dropping the last term gives an acyclic chain complex.

Notice that $\partial_n=\sum_{i=0}^n(-1)^id_i$,
where $d_i(x_0,...,x_n)= (x_0,...,x_{i-1},x_i*x_{i+1},x_{i+2},...,x_n)$, for $0\leq i<n$ and \\
$d_n(x_0,...,x_n)= (x_n*x_0,...,x_{n-1})$. \\
Again, $(C_n,d_i)$ is a presimplicial module. 
If $(X,*)$ is a monoid, one can define $n+1$ homomorphisms $s_i: C_n \to C_{n+1}$, called 
degeneracy maps, by $s_i(x_0,...,x_n)= (x_0,...,x_i,1,x_{i+1},...,x_n)$ 
(similarly, in the case of group homology of a semigroup, we put, 
$s_i(x_1,...,x_n)= (x_1,...,x_i,1,x_{i+1},...,x_n)$). We check that in both cases the 
following conditions hold:
$$ (1) \ \ \  d_id_j = d_{j-1}d_i\ for\ i<j. $$
$$(2)\ \ \ s_is_j=s_{j+1}s_i,\ \ 0\leq i \leq j \leq n, $$
$$ (3) \ \ \ d_is_j= \left\{ \begin{array}{rl}
 s_{j-1}d_i &\mbox{ if $i<j$} \\
s_{j}d_{i-1} &\mbox{ if $i>j+1$}
       \end{array} \right.
$$
$$ (4) \ \ \ d_is_i=d_{i+1}s_i= Id_{C_n}. $$

$(C_n,d_i,s_i)$ satisfying conditions (1)-(4) above is called a simplicial module\footnote{The concept  
of a simplicial set was introduced by Eilenberg and Zilber who called it {\it complete semi-simplicial complex}; 
their semi-simplicial complex is now usually called {\it presimplicial set} \cite{E-Z,May}.} 
(e.g. $Z$-module/abelian group).
If we replace (4) by a weaker condition $d_is_i=d_{i+1}s_i$ we deal with a weak simplicial module, 
the concept useful in the theory of homology of distributive structures (spindles or quandles).

As we already mentioned before Hochschild homology (and presimplicial module) can be defined for 
any algebra $A$ and two-sided $A$-module $M$. We put $C_n(A;M)= M\otimes A^{\otimes n}$ and 
$d_i(m,x_1,...,x_n)$ is given by:
$$d_0(m,x_1,...,x_n)= (mx_1,x_2,...,x_n).$$
$$d_i(m,x_1,...,x_n)= (m,x_1,...,x_{i-1},x_ix_{i+1},x_{i+2},...,x_n) \mbox{ for } 0<i<n, \mbox{ and}$$ 
$$d_n(m,x_1,...,x_n)=(x_nm,x_1,...,x_{n-1}).$$
From associativity of $A$ and our action of $A$ on $M$ follows that $(C_n,d_i)$ is a presimplicial module.
Furthermore, if $A$ is unitary we can define a simplicial module structure $(C_n,d_i,s_i)$, 
by putting $s_i(m,x_1,...,x_n)= (m,x_1,...,x_{i-1},1,x_{i},...,x_n)$.

\section{Homology of distributive structures}\label{Section 6}

Recall that a shelf (or right distributive system (RDS)) $(X;*)$ is a set $X$ with  
a right self-distributive  binary operation $*:X \times X \to X$ (i.e. $(a*b)*c= (a*c)*(b*c)$).

We work, for simplicity, with chain complexes and homology over $\Z$ but we can replace $\Z$ by any 
commutative ring $k$ in our considerations.

We start from atomic definition, one term distributive homology,  introduced in 2010 just before
Knots in Poland III conference \cite{Prz-5}.
 
\subsection{One-term distributive homology}

\begin{definition}\label{Definition 6.1}
We define a (one-term) distributive chain complex ${\mathcal C}^{(*)}$ as
follows:  $C_n=\Z X^{n+1}$ and the boundary operation $\partial^{(*)}_n: C_n \to C_{n-1}$ is given by:
$$\partial^{(*)}_n(x_0,...,x_n)= (x_1,...,x_n) +$$
 $$ \sum_{i=1}^{n}(-1)^i(x_0*x_i,...,x_{i-1}*x_i,x_{i+1},...,x_n).$$
The homology of this chain complex is called
a one-term distributive homology of $(X;*)$ (denoted by $H_n^{(*)}(X)$).
\end{definition}
We directly check that $\partial^{(*)}\partial^{(*)}=0$.

We can put $C_{-1}=\Z$ and $\partial_0(x)=1$. We have $\partial_0\partial_1^{(*)}=0$, so we
obtain an augmented distributive chain complex and an
augmented (one-term) distributive homology, $\tilde H^{(*)}_n$. As in the classical case we get:
\begin{proposition}\label{Proposition 6.4}
$
 H_n^{(*)}(X)=
 \begin{cases}
 \Z \oplus \tilde H^{(*)}_n(X) & n = 0 \\
 \tilde H^{(*)}_n(X) & \text{otherwise}
 \end{cases}
$
\end{proposition}

If  $(X;*)$ is a rack then the complex $(C^{(*)}_n, \partial^{(*)})$ is acyclic, but in the general case 
of a shelf or spindle homology can be nontrivial with nontrivial free and torsion parts (joint work with 
A.Crans, K.Putyra and A.Sikora \cite{CPP,Pr-Pu-1,P-S}).

If we define $d_i:C_n \to C_{n-1}$, $0\leq i \leq n$,  by $d_i(x_0,...,x_n)=(x_0*x_i,...,x_{i-1}*x_i,x_{i+1},...,x_n)$,
then $(C_n,d_i)$ is a presimplicial module and $X^{n+1},d_i)$ is a presimplicial set. 
If we define degeneracy maps $s_i(x_0,...,x_n)=(x_0,..,x_{i-1},x_i,x_i,x_{i+1},...x_n)$ then one checks 
that $(C_n,d_i,s_i)$ is a very weak simplicial module. If we assume idempotency, that is $(X;*)$ is a 
spindle, then $(C_n,d_i,s_i)$ is a weak simplicial module and the degenerate part $(C^D_n,\partial_n)$ 
is a subchain complex which splits from $C_n,\partial_n)$ (see \cite{Prz-5}. This split is analogous to the one 
conjectured in \cite{CJKS} and proved in \cite{L-N} for classical quandle homology. In \cite{N-P-2} we 
gave very short, easy to visualize and to generalize, proof using the split map $C^N_n \to C_n$ given by 
$(x_0,x_1,...,x_n) \to (x_0,x_1-x_0,...,x_n-x_{n-1})$.

We can repeat our definitions if $(X;*)$ is a shelf and $Y$ is a shelf-set ($*:Y\times X \to Y$ with 
$(y*x_1)*x_2 = (y*x_2)*(x_1*x_2)$ see Figure 6.1 for visualization). The presimplicial set  
 $(Y\times X^{n+1},d_i)$ has face maps $d_i$ defined by $d_i(y,x_0,...,x_n)=(y*x_i,x_0*x_i,...,x_{i-1}*x_i,x_{i+1},..,x_n)$.
The face map $d_i$ is visualized in Figure 6.2; this visualization will play an important role when 
distributive homology will be generalized to Yang-Baxter homology.

\ \\
\centerline{\psfig{figure=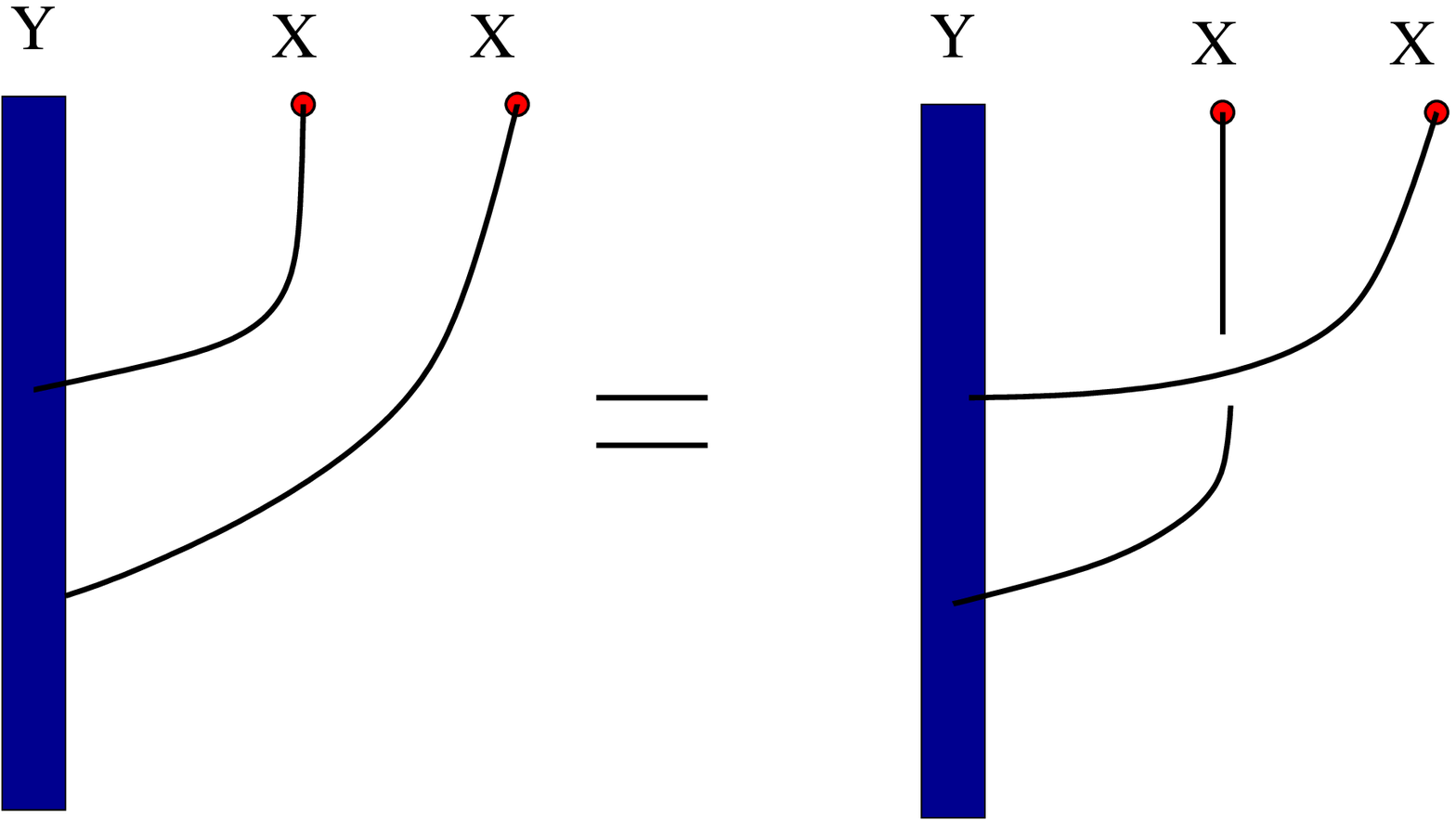,height=3.3cm}}
\centerline{Figure 6.1; Graphical interpretation of the axiom for $X$-shelf-set $Y$}
\centerline{$(y*x_1)*x_2 = (y*x_2)*(x_1*x_2)$}

\ \\
\centerline{\psfig{figure=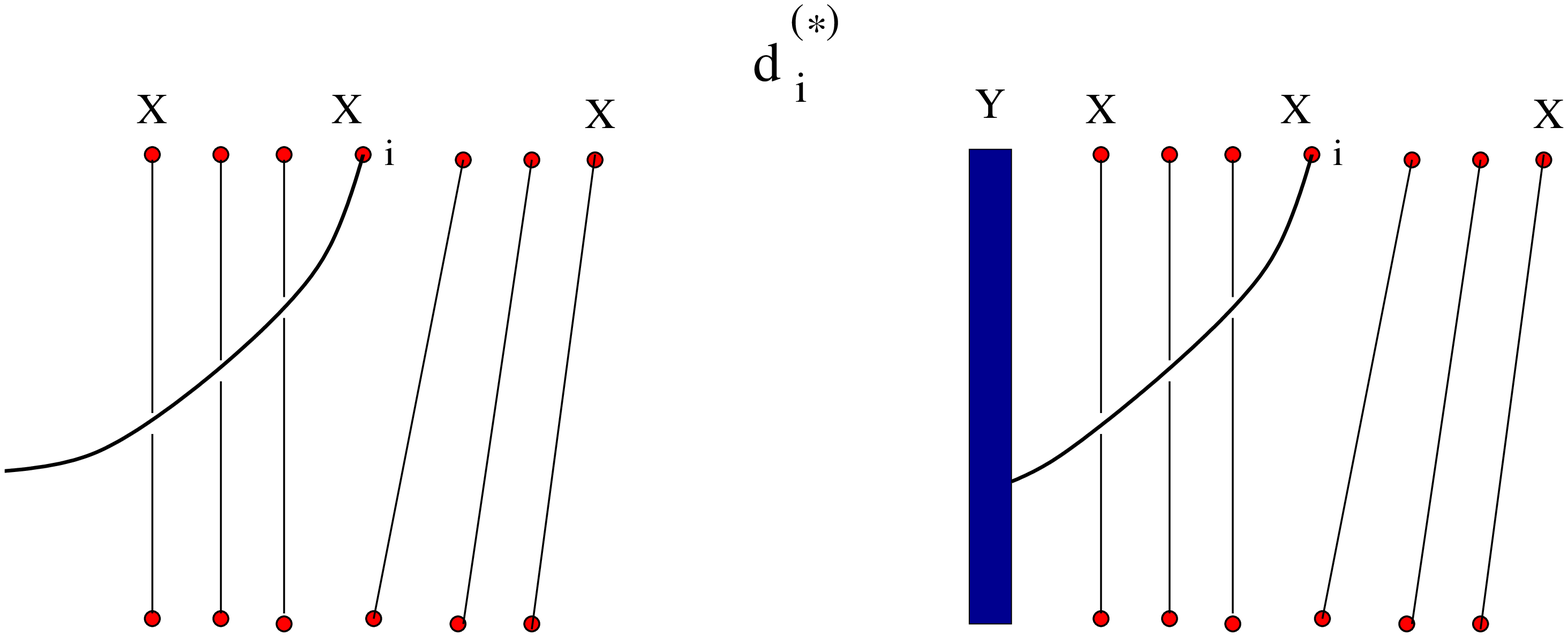,height=4.3cm}}
\centerline{Figure 6.2; Graphical interpretation of the face map $d^{(*)}_i$}
 
\subsection{Multi-term distributive homology}\label{Section 7}

The first homology theory related to a self-distributive structure was constructed in early 1990s by
Fenn, Rourke, and Sanderson \cite{FRS-2} and motivated by (higher dimensional) knot theory\footnote{The recent
paper by Roger Fenn, \cite{Fenn} states:
"Unusually in the history of mathematics, the discovery of the homology and classifying
space of a rack can be precisely dated to 2 April 1990."}.
For a rack $(X,*)$, they defined rack homology $H_n^R(X)$ by taking $C^R_n=\Z X^n$ and
$\partial_n^R: C_n \to C_{n-1}$ is given by $\partial_n^R = \partial_{n-1}^{(*)}-\partial_{n-1}^{(*_0)}$.
Our notation has grading shifted by 1, that is, $C_n(X)= C^R_{n+1}= \Z X^{n+1}$. It is  routine to
check that $\partial^R_{n-1}\partial_n^R=0$. However, it is an interesting question what properties
of $*_0$ and $*$ are really used. With relation to the paper \cite{N-P-3} we noticed that
it is distributivity again which makes $(C^R(X),\partial_n^R)$ a chain complex. More generally we
observed that if $*_1$ and $*_2$ are right self-distributive and distributive  with respect to each other,
then $\partial^{(a_1,a_2)}= a_1\partial^{(*_1)}+a_2\partial^{(*_2)}$ leads to a chain complex
(i.e. $\partial^{(a_1,a_2)}\partial^{(a_1,a_2)}=0$).
Below I answer a more general question: for a finite set $\{*_1,...,*_k\}\subset Bin(X)$ and integers $a_1,...,a_k \in \Z$,
when is $(C_n,\partial^{(a_1,...,a_k)})$ with $\partial^{(a_1,...,a_k)}=a_1\partial^{(*_1)}+...+a_k\partial^{(*_k)}$
a chain complex?
When is $(C_n,d_i^{(a_1,...,a_k)})$ a presimplicial set?
We answer these questions in Lemma \ref{Lemma 6.3}. In particular, for a distributive set $\{*_1,...,*_k\}$
the answer is affirmative.

\begin{lemma}\label{Lemma 6.3}\
\begin{enumerate}
\item[(i)] If $*_1$ and $*_2$ are right self-distributive operations, then  $(C_n,\partial^{(a_1,a_2)})$ is
a chain complex if and only if the operations $*_1$ and $*_2$ satisfy:
\begin{formulla}\label{Formulla 6.4}
$$(a*_1b)*_2c + (a*_2b)*_1c = (a*_2c)*_1(b*_2c) + (a*_1c)*_2(b*_1c) \mbox{ in ${\Z}X$}.$$
\end{formulla}
We call this condition  {\it weak distributivity}. If Condition \ref{Formulla 6.4} does not hold we can take $C_0(X)$
to be the quotient by Equation 6.4: $C_0(X)= {\Z}X/{\bf 6.4}$ and then we take $C_n=C_0^{\otimes n+1}$.
\item[(ii)] We say that a set $\{*_1,...,*_k\}\subset Bin(X)$ is weakly distributive if each operation is
right self-distributive and each pair of operations is weakly distributive (with two main cases:
distributivity $(a*_1b)*_2c = (a*_2c)*_1(b*_2c)$ and chronological distributivity\footnote{I did not see this concept
considered in literature, but it seems to be important in K.Putyra's work
on odd Khovanov homology \cite{Put}.}
$(a*_1b)*_2c=(a*_1c)*_2(b*_1c)$).
We have: $(C_n,d_i^{(a_1,...,a_k)})$ is a presimplicial set if and only if
the  set $\{*_1,...,*_k\}\subset Bin(X)$ is weakly distributive.
\item[(iii)] $(C_n,\partial_n^{(a_1,...,a_k)})$ is a chain complex if and only if
the  set $\{*_1,...,*_k\}\subset Bin(X)$ is weakly distributive.
\end{enumerate}
\end{lemma}

We complete this section by showing 
that for a rack homology of a quandle or spindle,
$H^R_n$ embeds in $H^R_{n+1}$; we construct monomorphic ``homology operation" of degree one\footnote{For 
a quandle it is a well know fact that rack homology in dimension $n$ is isomorphic to ``early degenerate" 
homology in dimension $n+1$.}. 
We place it in a more general context of weak simplicial modules.
\begin{lemma}\label{Lemma 6.5}
Let $(C_n,d_i,s_i)$ be a weak simplicial module then $\partial s_0+ s_0\partial = s_0d_0$; in effect  
$s_0d_0$ induces a trivial map on homology. In particular: 
\begin{enumerate}
\item[(i)] if the map $s_0d_0$ is the identity then $s_0d_0s_0=s_0$ (as in the 
case of a simplicial module) and then the chain complex $(s_0(C_{n-1}),\partial_n)$ is acyclic, 
\item[(ii)] if $d_0=0$ then $s_0$ is a chain map (e.g. this hold for 2-term rack homology),
\item[(iii)] In the case of one term distributive homology, we conclude that the map replacing $x_0$ by $x_1$ in 
$(x_0,x_1,...,x_n)$ is a chain map, chain homotopic to zero map.
\end{enumerate}
\end{lemma}
\begin{proof} We have $\partial s_0+ s_0\partial =$
$$ d_0s_0-d_1s_0 +\sum_{i=2}^{n+1}(-1)^id_is_0 + s_0d_0 + \sum_{i=1}^{n}(-1)^is_0d_i =$$
$$ d_0s_0-d_1s_0 +\sum_{i=2}^{n+1}(-1)^is_0d_{i-1} + s_0d_0 + \sum_{i=1}^{n}(-1)^is_0d_i = s_0d_0.$$
\end{proof}
If $s_0$ has a left inverse map, say $p_n:C_{n+1}\to C_n$, $p_ns_0=Id_{C_n}$, as is the case for 
a weak simplicial module in (multi) spindle case, we can say more.

\begin{lemma}\label{Lemma 6.6}
 Let $(C_n,d_i,s_i)$ be a weak simplicial module $p_n:C_{n+1}\to C_n$ is a left inverse of $s_0$
and additionally $p_nd_i = d_{i-1}p_n $ for $i>0$ then $p\partial + \partial p= pd_0$; in effect 
$pd_0$ induces a trivial map on homology. In particular: 
\begin{enumerate}
\item[(i)] If we deal with (multi)term distributive homology, $p$ may be taken to be the map deleting the 
first coordinate of $(x_0,x_1,...,x_n)$ (in one term distributive homology $p=d_0$, so $d_0d_0$ is 
a chain map trivial on homology).
\item[(ii)] if $d_0=0$ then $p$ is a chain map ((e.g. this hold for 2-term rack homology),
\item[(iii)] if $d_0=0$ then $p$ induces an epimorphism on homology and $s_0$ induces a monomorphism on homology;
in particular $s_0$ induces monomorphic ``homology operation" of degree one ($s_0(C_{n-1},\partial_n)$ is called 
an early degenerate chain complex).
\end{enumerate}
\end{lemma}
\begin{proof} We have $p \partial + \partial p   =$ 
$$ pd_0 +\sum_{i=1}^{n}(-1)^id_{i-1}p + \sum_{i=0}^{n-1}(-1)^id_{i}p =  pd_0$$
Part (iii) follow from the fact that $p_ns_0=Id_{C_n}$ and $p$ and $s_0$ are chain maps.
\end{proof}
Thus we proved that rack homology of quandles (or spindles) cannot decrease with $n$ 
($H^R_n \subset H^R_{n+1}$).

We computed with K.Putyra \cite{Pr-Pu-1} various multi-term homology, including that for finite distributive 
lattices (including Boolean algebras).

\section{Bloh-Leibniz-Loday algebra}
Lie algebra was probably the first nonassociative structure for which homology was defined \cite{Ch-E}.
The idea of Chevalley and Eilenberg was to translate homology of a (Lie) group to homology of its Lie 
algebra\footnote{The paper starts from: {\it The present paper lays no claim to deep originality.
Its main purpose is to give a systematic treatment of the methods by which topological
questions concerning compact Lie groups maybe reduced to algebraic questions concerning Lie algebras}.}. 
We should stress, in particular the role of conjugacy in Lie algebra, as 
conjugacy was the motivation for wracks (racks) and quandles.

We discuss here homology theory of Bloh-Loday-Leibniz algebras introduced by Bloh an Loday \cite{Blo-1,Blo-2,Lod-2},
and which can be informally thought to be a linearization of distributive homology.\footnote{BLL algebras are 
often call Leibniz algebras as the version of Jacobi identity they satisfy can be treated as a Leibniz rule. The 
history of the discovery is described by Loday as follows \cite{Lod-1}:\\
 ``In the definition of the Chevalley-Eilenberg complex of a Lie algebra ${\mathcal G}$ the module
of chains is the exterior module. The non-commutative analog of the exterior module $\Lambda \mathcal G$ is
the tensor module $T\mathcal G$. If one replaces $\Lambda$ by $\bigotimes$ in the classical formula for the boundary map
$d$ of the $CE$-complex, then one gets a well defined map $T\mathcal G$ but the relation $d^2$ is not
valid anymore. However I discovered that, if one writes $d$ so as as to put the commutator $[x_i,x_j]$ at the place
$i$ when $i<j$,..., then the relation $d^2=0$ is satisfied in the tensor (i.e. non-commutative) context.
So, this give rise to a new complex $T{\mathcal G},d)$ for the Lie algebra ${\mathcal G}$. The homology
groups of this complex are denoted are denoted $HL_*({\mathcal G})$ and called the non-commutative homology
groups of ${\mathcal G}$. In the proof of the relation $d^2=0$ in the tensor module case, I noticed that
the only property  of the Lie bracket, which is needed, is the Leibniz relation $[x,[y,z]] = [[x,y],z] - [[x,z],y]$.
 So the complex $(T{\mathcal G},d)$
and its homology are  defined for more general objects than Lie algebras, for the {\it Leibniz algebras}."
 }
We follow Loday and Lebed here \cite{Lod-1,Lod-2,Leb-1,Leb-2}. Because BLL (Bloh-Leibniz-Loday) algebra is a generalization
of a Lie algebra we use a bracket $[-,-]$ for a bilinear map:
\begin{definition}\label{Definition 7.1}
\begin{enumerate}
\item[(1)] Let $V$ be a $k$-module equipped with a bilinear map $[-,-]: V\times V \to V$ satisfying
the relation (Leibniz version of the Jacobi identity):
$$[x,[y,z]]= [[x,y],z] - [[x,z],y], \mbox{ for all $x,y,z\in V$}.$$
I see the linearization of distributivity as  
$$(x*y)*z=(x*z)*(y*z)  \Longrightarrow (x*y)*z= (x*z)*y+ x*(y*z) \mbox{ BLL condition}.$$
V.Lebed formalized this ``linearization" in the case $V$ has a central element $1$ that is 
$[x,1]=0[1,x]$  we color crossing as follows: {\parbox{3.6cm}{\psfig{figure=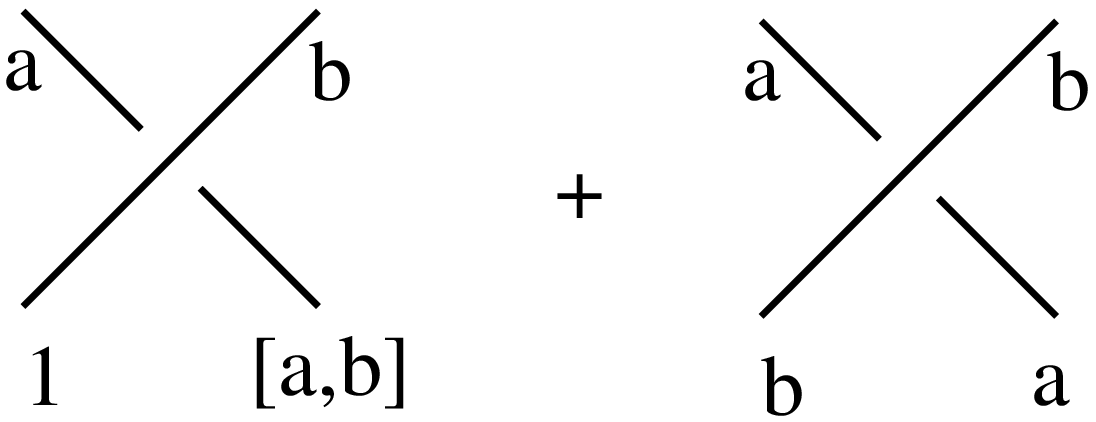,height=1.3cm}}}. 
This led Lebed to describe face  map for chain complex of BLL algebras as in Figure 7.1; compare the figure with 
Definition \ref{Definition 7.2}. 
\item[(2)] A BLL module $M$ over $V$ is a $k$-module with a bilinear action (still denoted $[-,-]:M\times V \to M$,
satisfying the formula from (1) for any $x\in M$ and $y,z\in V$.\\
In a special case of $M=k$ ($k$ a ring with identity), the map $[-,-]:M\times V \to M$ is replaced by
map $\epsilon: V\to k$ which is zero on commutators (Lie character), \cite{Leb-2}).
\end{enumerate}
\end{definition}
\centerline{\psfig{figure=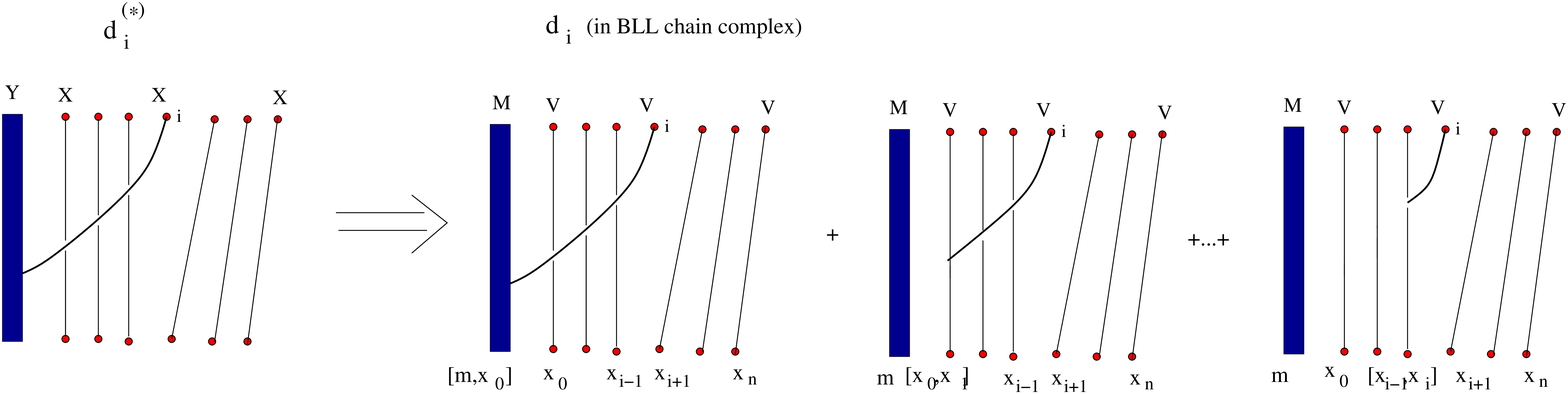,height=3.8cm}}\ \ \\
\centerline{Figure 7.1; Comparing face map $d_i$ in distributive and unital BLL algebra }
\ \\ \ \\
Homology of $(V,M)$ was constructed by Loday based partially on the work of C.Cuvier \cite{Cuv,Lod-1}, they
were unaware of the earlier work by Bloh \cite{Blo-2}.
\begin{definition}\label{Definition 7.2}
Let $V$ be a BLL algebra, and $M$ a  BLL module over $V$.
\begin{enumerate}
\item[(1)] We define $C_n= M\otimes V^{\otimes n+1}$, that is $C_*= TV$ (tensor algebra).
For  $0\leq i \leq n$, $d_i: C_n \to C_{n-1}$ is given by:
$$d_i(x_{-1},x_0,...,x_n)=\sum_{j: -1\leq j<i}(x_{-1},x_0,...,x_{j-1},[x_j,x_i],x_{j+1},x_{i-1},x_{i+1},...,x_n).$$

If we take $\Lambda V$, the exterior algebra in place of $ TV$. We will be in the classical case of
homology developed for Lie algebras by  Chevalley and Eilenberg \cite{Ch-E}.

\item[(2)] Assume that $V$ is a split unital BLL algebra (that is $V=V'\oplus k 1$,
and $M=k$, $[1,x]=\epsilon(x)$ ($1 \in k$ $x\in V$), then the ``primitive" degenerate maps $s_i^p: C_n \to C_{n+1}$
is defined by
$$s^p_i(x_{-1},x_0,...,x_n)=
 \left\{ \begin{array}{rl}
(x_{-1},x_0,...,x_{i-1},(1,x_i)+(x_i,1),x_{i+1},...,x_n) &\mbox{ if $x_i\in V'$} \\
  (x_{-1},x_0,...,x_{i-1},1,1,x_{i+1},...,x_n) &\mbox{ if $x_i=1$ }
       \end{array} \right.
$$
\item[(3)] Assume $V$ is a free $k$-module with basis $X$ and define degenerate maps $s_i^g: C_n \to C_{n+1}$ on $X$
by doubling $i$th coordinate:
$$s_i^g(m,x_0,..,x_n)= (m,x_0,..,x_{i-1},x_i,x_i,x_{i+1},...,x_n).$$

\end{enumerate}
\end{definition}
\begin{lemma}\label{Lemma 8.3}
\begin{enumerate}
\item[(1)]
$(C_n(V,M),d_i)$ is a presimplicial module, that is $d_id_j=d_{j-1}d_i$ for $0\leq i<j\leq n$.
\item[(2)]
$(C_n(X,M),d_i,s_i^p)$ is a weak simplicial module, and its homology and degenerated homology are
annihilated by $\epsilon (1)$.
\item[(3)]
$(C_n(X,M),d_i,s_i^g)$ satisfies conditions (1) (2) and partially (3) ($d_is_j=s_{j-1}d_i$ for $i<j$) of
a simplicial module. Condition (4'),  $d_is_i - d_{i+1}s_{i}= 0$ holds iff $[x,x]$ for all $x\in X$.
The condition (3) of a simplicial module for $i>j+1 $ requires the following equality:
$$[x_j,x_{i}]\otimes [x_j,x_{i}] = [x_j,x_{i}]\otimes x_j +
x_j \otimes [x_j,x_{i}]$$ for basic elements $(x_j,x_{i}) \in X^2$ for $i>j+1$. Thus if we take
$C_*=VT/{\mathcal I}$ divided by an ideal containing the above equation, then $(C_n(X,M),d_i,s_i^g)$
is a very weak simplicial module. It is a weak simplicial module
iff additionally $[x,x]=0$, for all $x\in X$.\footnote{Notice that in exterior algebra $\Lambda V$,
the equation $[x_j,x_{i}]\otimes [x_j,x_{i}] = [x_j,x_{i}]\otimes x_j + x_j \otimes [x_j,x_{i}]$
holds ($0=0$), but in that case our degenerate map $s^g_i$ would be a zero map.}
\end{enumerate}
\end{lemma}
We leave the proof as an exercise for the reader however we make the calculation in two small but
typical cases which show that our axioms are needed:\\
(i) Comparison of $d_0d_1$ with $d_0d_0$ (they should be equal):\\
$d_0d_1(m;x_0,x_1)=d_0(([m,x_1];x_0) + (m;[x_0,x_1]))= ([m,x_1],x_0) + [m,[x_0,x_1]]$,\\
$d_0d_0(m;x_0,x_1)=d_0([m,x_0];x_1))= [[m,x_0)],x_1]$,\\
Thus $d_0d_1=d_0d_0$ if and only if $([m,x_1],x_0] + [m,[x_0,x_1]= [[m,x_0],x_1]$ which is the axiom of BLL-module.\\
(ii) Comparison of $d_1d_2$ with $d_1d_1$ (they should be equal):\\
$d_1d_2(m;x_0,x_1,x_2) = d_1(([m,x_2];x_0,x_1) + (m;[x_0,x_2],x_1) + (m;x_0,[x_1,x_2]))=$\\
$([[m,x_2],x_1];x_0) + (([m,x_2];[x_0,x_1]) + $\\
$([m,x_1];[x_0,x_2]) + (m;[[x_0,x_2],x_1]) + $\\
$([m,[x_1,x_2]];x_0)  + (m;[x_0,[x_1,x_2]]),$\\
and\\
$d_1d_1(m;x_0,x_1,x_2) = d_1([m,x_1];x_0,x_2) + (m;[x_0,x_1],x_2))=$\\
$([[m,x_1],x_2];x_0) + ([m,x_1];[x_0,x_2]) + $\\
$([m,x_2];[x_0,x_1]) + (m;[[x_0,x_1],x_2]).$\\
Thus $d_1d_2=d_1d_1$ if and only if the following sum is equal to zero:
$$([[m,x_2],x_1];x_0) + ([m,[x_1,x_2]];x_0) - ([[m,x_1],x_2];x_0) +$$
$$(m;[[x_0,x_2],x_1]) +  (m;[x_0,[x_1,x_2]]) - (m;[[x_0,x_1],x_2]).$$

The first part is equal to zero iff $M$ is BLL-module and the second part is equal to zero
iff $V$ is BLL-algebra.

\begin{remark}\label{Remark 7.4}
Lie algebra homology, as proved by Cartan and Eilenberg \cite{Ca-E} can be obtained from homology 
of the universal enveloping algebra $UV=TV/(a\otimes b - b\otimes a=[a,b])$ of the Lie algebra $V$. 
One hopes for a similar connection between distributive homology and homology of the group  associated 
to a wrack or quandle. One hint in this direction is that in every group the following ``distributivity" 
holds:
$$[[x,y^{-1}],z]^y[[y,z^{-1}],x]^z[[z,x^{-1}],y]^x=1,$$
where $[x,y]=x^{-1}y^{-1}xy$ and $x^y=y^{-1}xy$. This leads to the graded Lie algebra associated to the group,
via lower central series of the group \cite{Va}. 
\end{remark}
 
\section{Semigroup extensions and shelf extensions}\label{Section 8}
The theory of extension of structures and related cocycles started from two important examples from group theory:\\
(i) The extension of $PSL_n(C)$ by $SL_n(C)$ by I.Schur (1904), with related short exact sequence of groups \cite{B-T}
$$0 \to \Z_2 \to SL_n(C) \to PSL_n(C) \to 1  \mbox{ and}$$
the study of crystallographic groups $\Gamma$  where we consider a short exact sequence
$$0 \to \Z^n \to \Gamma \to \Gamma/\Z^n \to 1. $$ 
An extension of a group $X$ by a group $N$ is a short exact sequence of groups
$$1\to N \stackrel{i}{\rightarrow} E \stackrel{\pi}{\rightarrow} X \to 1$$
(some people call this an extension of $N$ by $X$ \cite{B-T,Bro,Ma-Bi}).\\
Consider a set-theoretic section $s:X \to E$ (that is $\pi s=Id_X$). Every element of $E$ is a unique product 
$as(x)$ for $a\in N$ and $x\in X$ (coset decomposition), thus  we have $E= N\times X$ as sets 
(here $e \to (es(\pi(e^{-1})),\pi(e))$ and $es(\pi(e^{-1})) (s\pi(e))=e$ as needed. The inverse 
map is $(a,x) \to as(x)$.

This motivates study of extension of magmas as study of projections $\pi: A\times X \to X$ with various structures 
preserved. In particular, we compare  semigroup extension of a semigroup by an abelian group 
with the shelf extension of a shelf 
by an Alexander quandle.
We start from a general concept of a dynamic cocycle in a magma case and then in associative and distributive
cases and in both we relate to  the (co)homology of our structures. 
Extension of modules, groups and Lie algebras is described in the classical book by Cartan and Eilenberg \cite{Ca-E},
distributive case was developed in \cite{CES-2,A-G,CKS}.

\begin{definition}\label{Definition 8.1}
Let $(X;*)$ be a magma, $A$ a set, and $\pi: A\times X \to X$ the projection to 
the second coordinate. Any magma structure on $A\times X $ for which $\pi$ is an epimorphism,
can be given by a system of functions $\phi_{a_1,a_2}(x_1,x_2): X\times X \to A$ by:
$$(a_1,x_1)*(a_2,x_2)= (\phi_{a_1,a_2}(x_1,x_2), x_1*x_2).$$ 
Functions $\phi_{a_1,a_2}(x_1,x_2)$ are uniquely  defined by the multiplication on $A\times X $, 
thus binary operations on $A\times X$ agreeing with $\pi$ are in bijection with choices of functions 
$\phi_{a_1,a_2}$.
If we require some special structure on  $(X;*)$ (e.g. associativity or right-distributivity) we obtain
some property of $\phi_{a_1,a_2}(x_1,x_2)$ which we call a dynamical co-cycle property for the structure.
\begin{enumerate}
\item[(1)] Let $(X;*)$ be a semigroup; in order that an action on $A\times X$ is associative we 
need:
$$((a_1,x_1)*(a_2,x_2))*(a_3,x_3)=(\phi_{a_1,a_2}(x_1,x_2), x_1*x_2)*(a_3,x_3)=$$
$$ (\phi_{\phi_{a_1,a_2}(x_1,x_2),a_3}(x_1*x_2,x_3), (x_1*x_2)*x_3)$$
to be equal to 
$$(a_1,x_1)*((a_2,x_2))*(a_2,x_3))= (a_1,x_1)*(\phi_{a_2,a_3}(x_2,x_3), x_2*x_3)= $$
$$ (\phi_{a_1,\phi_{a_2,a_3}}(x_2,x_3))(x_1,x_2*x_3),x_1*(x_2*x_3)).$$
Thus the dynamical cocycle condition in the associative case has a form:
$$(\phi_{\phi_{a_1,a_2}(x_1,x_2),a_3}(x_1*x_2,x_3)= \phi_{a_1,\phi_{a_2,a_3}(x_2,x_3)}(x_1,x_2*x_3)). $$
\item[(2)] Let $(X;*)$ be a shelf; in order that an action on $A\times X$ is right self-distributive 
we need:
$$((a_1,x_1)*(a_2,x_2))*(a_2,x_3)=(\phi_{a_1,a_2}(x_1,x_2), x_1*x_2)*(a_3,x_3)=$$
$$ (\phi_{\phi_{a_1,a_2}(x_1,x_2),a_3)}(x_1*x_2,x_3), (x_1*x_2)*x_3)$$
to be equal to 
$$((a_1,x_1)*(a_3,x_3))*((a_2,x_2)*(a_3,x_3))=$$ 
$$  (\phi_{a_1,a_3}(x_1,x_3), x_1*x_3)* (\phi_{a_2,a_3}(x_2,x_3), x_2*x_3)=$$
$$ (\phi_{\phi_{a_1,a_3}(x_1,x_3),\phi_{a_2,a_3}(x_2,x_3)}( x_1*x_3,x_2*x_3), (x_1*x_3)*(x_2*x_3)). $$
Thus the dynamical cocycle condition in right-distributive case has a form:
$$\phi_{\phi_{a_1,a_2}(x_1,x_2),a_3}(x_1*x_2,x_3)=
\phi_{\phi_{a_1,a_3}(x_1,x_3),\phi_{a_2,a_3}(x_2,x_3)}( x_1*x_3,x_2*x_3).$$
\item[(3)] We assume now that $(X;*)$ is an entropic (inner turn) magma, that is 
$(a*b)*(c*d)= (a*c)*(b*d)$ for any $a,b,c,d \in X$. We look for condition on the dynamical co-cycle 
so that $A\times X$ is entropic. We need
$$((a_1,x_1)*(a_2,x_2))*((a_3,x_3)*(a_4,x_4))=$$
$$  (\phi_{a_1,a_2}(x_1,x_2), x_1*x_2)* (\phi_{a_3,a_4}(x_3,x_4), x_3*x_4)=$$
$$ (\phi_{\phi_{a_1,a_2}(x_1,x_2),\phi_{a_3,a_4}(x_3,x_4)}( x_1*x_2,x_3*x_4), (x_1*x_2)*(x_3*x_4))$$
to be equal to
$$((a_1,x_1)*(a_3,x_3))*((a_2,x_2)*(a_3,x_3))=$$
$$  (\phi_{a_1,a_3}(x_1,x_3), x_1*x_3)* (\phi_{a_2,a_4}(x_2,x_4), x_2*x_4)=$$
$$ (\phi_{\phi_{a_1,a_3}(x_1,x_3),\phi_{a_2,a_4}(x_2,x_4)}( x_1*x_3,x_2*x_4), (x_1*x_3)*(x_2*x_4)).$$
Thus the dynamic cocycle condition in entropic case has the form:
$$\phi_{\phi_{a_1,a_2}(x_1,x_2),\phi_{a_3,a_4}(x_3,x_4)}( x_1*x_2,x_3*x_4)=
  \phi_{\phi_{a_1,a_3}(x_1,x_3),\phi_{a_2,a_4}(x_2,x_4)}( x_1*x_3,x_2*x_4).$$
\end{enumerate}
\end{definition}

We illustrate the above by several examples, starting from a classical group extension by an abelian group.
Consider the extension $E$ of a group $X$ by an abelian group $A$; this is described by a short exact sequence 
of groups:
$$ 0\to A \to E \stackrel{\pi}{\rightarrow} X \to 1 $$
As noted before,
 $E= A\times X$ as a set and bijection depends on a section $s: X  \to E$. Furthermore $X$ acts on $A$ 
(we have $X\times A \to A$) and the action is given by conjugation: $x(a)= s(x)a(s(x)^{-1}$ and does not 
depend on the choice of $s$ as $A$ is commutative).\\
For a semigroup this is the starting point.\\
Let $X$ be a semigroup, $\pi: A\times X$ a projection and a semigroup $X$ act on a set $A$. 
We define a product on $A\times X$ by the formula: 
$$(a_1,x_1)(a_2,x_2)= (a_1 + x_1(a_2) + f(x_1,x_2), x_1x_2).$$ 
The  function $f:X\times X \to A$, as in the group case,  arise by comparing section of a multiplication 
with multiplication of sections, that is $s(x_1x_2)=f(x_1,x_2)s(x_1)s(x_2)$.  
We assume that the action $x: A \to A$ is 
a group homomorphism for any $x$ and it is associative ($x_1(x_2(a)= (x_1x_2)(a)$). The associativity 
of the product on $A\times X$ is equivalent to the 
condition on $f:X\times X \to A$ of the form\footnote{Calculation is as follows: Associativity,
$$((a_1,x_1)(a_2,x_2))(a_3,x_3)= (a_1,x_1)((a_2,x_2))(a_3,x_3))\mbox{ gives, after expanding each side:}$$
$$((a_1,x_1)(a_2,x_2))(a_3,x_3)= (a_1+ x_1(a_2)+ f(x_1,x_2), x_1x_2)(a_3,x_3)=$$
$$(a_1+ x_1(a_2)+ f(x_1,x_2) + (x_1x_2)(a_3) + f(x_1x_2,x_3), (x_1x_2)x_3) \mbox{ and}$$
$$(a_1,x_1)((a_2,x_2))(a_3,x_3))= (a_1,x_1)(a_2+ x_2(a_3)+ f(x_2,x_3),x_2x_3)=$$
$$(a_1+ x_1(a_2+ x_2(a_3)+ f(x_2,x_3) + f(x_1,x_2x_3), x_1(x_2x_3))$$ 
thus the associativity reduces to:
$$f(x_1,x_2)+ f(x_1x_2,x_3) =x_1(f(x_2,x_3)) + f(x_1,x_2x_3) \mbox{ which is our 2-cocycle condition.}$$
}:
$x_1(f(x_2,x_3)) - f(x_1x_2,x_3) + f(x_1,x_2x_3) - f(x_1,x_2)=0$ which we call a second cocycle 
condition (relation to homology of groups defined before, will be explained).
Thus $\phi_{a_1,a_2}(x_1,x_2)= a_1 + x_1(a_2) + f(x_1,x_2)$, is an example of a dynamical cocycle 
for an associative structure.
We should stress that for a semigroup there may be choice for a dynamical cocycle but for a group it is 
unique (see e.g. \cite{Bro}).
$f:X\times X \to A $ is a cocycle for a chain complex introduced in Definition 4.1 for the trivial action 
(and generally Definition 4.2); we have:
$$\partial^2(f)(x_1,x_2,x_3)=$$
$$ f(\partial_2((x_1,x_2,x_3))= f((x_2,x_3)-(x_1x_2,x_3)+ (x_1,x_2x_3)-(x_1,x_2))=$$
$$f(x_2,x_3) - (x_1x_2,x_3) + f(x_1,x_2x_3) - f(x_1,x_2)=0.$$ 
If action of $X$ on $A$ is not necessarily trivial, we define cohomology $H^n(G,C)$ with a cochain complex
$C_n=Hom (ZG^n \to A)$ and \\
 $\partial^n: C^n \to C^{n+1}$ is given by $\partial^n(f)(x_1,...,x_n,x_{n+1})=$ 
$$x_1f(x_2,...,x_{n+1}) +\sum_{i=1}^{n}(-1)^if(x_1,...,x_ix_{i+1},...x_{n+1}) + (-1)^{n+1}(x_1,...,x_{n}).$$

\subsection{Extensions in right distributive case} 
We give here two examples of extension in right distributive case:\\
I. 
Let $(X;*)$ be a shelf and $A$ an abelian group with a given homomorphism $t: A \to A$ (equivalently, 
$A$ is a ${\Z}[t]$ module). We define a shelf structure (called Alexander extension \cite{CES-2,CKS}) on $A\times X$ 
by the formula: 
$$(a_1,x_1)(a_2,x_2)= (ta_1 + (1-t)a_2+ f(x_1,x_2), x_1*x_2)$$ and right distributivity is equivalent to the
condition on $f:X\times X \to A$ which satisfies
 twisted cocycle condition:
$$t(f(x_2,x_3)-f(x_1,x_3) + f(x_1,x_2)) - f(x_2,x_3) + f(x_1*x_2,x_3) - f(x_1*x_3,x_2*x_3)=0.$$ 
The calculation is as follows: Right self-distributivity
$$((a_1,x_1)*(a_2,x_2))*(a_3,x_3)= ((a_1,x_1)*(a_3,x_3))*((a_2,x_2)*(a_3,x_3))$$
 gives, after expanding each side:
$$((a_1,x_1)*(a_2,x_2))*(a_3,x_3)= (ta_1+ (1-t)a_2 + f(x_1,x_2), x_1*x_2)*(a_3,x_3)=$$
$$(t(ta_1+ (1-t)a_2+ f(x_1,x_2)) + (1-t)a_3 + f(x_1*x_2,x_3), (x_1*x_2)*x_3) \mbox{ and}$$
$$((a_1,x_1)*(a_3,x_3))*((a_2,x_2)*(a_3,x_3))= $$
$$(ta_1 + (1-t)a_3+ f(x_1,x_3), x_1*x_3)*(ta_2 + (1-t)a_3+ f(x_2,x_3), x_2*x_3)=$$
$$(t(ta_1 + (1-t)a_3+ f(x_1,x_3))+ (1-t)(ta_2 + (1-t)a_3+ $$
$$ f(x_2,x_3))+ f(x_1*x_3,x_2*x_3), (x_1*x_3)*(x_2*x_3))$$
This is equivalent to: $$tf(x_1,x_2) + f(x_1*x_2,x_3) = tf(x_1,x_3) + (1-t)f(x_2,x_3)) + f(x_1*x_3,x_2*x_3),$$
and further to a cocycle in a (twisted) rack homology:
$$(\partial^Rf)(x_1,x_2,x_3)=$$
$$ -t(f(x_2,x_3)-f(x_1,x_3))+ f(x_1,x_2)) + $$
$$f(x_2,x_3) - f(x_1*x_2,x_3) +f(x_1*x_3,x_2*x_3)=0.$$
If
there are two right self-distributive binary operations, $*_1$ and $*_2$ on $A\times X$ represented by $f_1$ and $f_2$
respectively (that is $(a_1,x_1)*_i(a_2,x_2)= (a_1*a_2 + f_i(x_1,x_2),x_1*x_2$, $i=1,2$), and there is
a homomorphism $H:A\times X \to A\times X$ given by $H(a,x)= (a+c(x),x)$ for some $c:X\to A$  then
the homomorphism condition\\
 $H((a_1,x_1)*_1(a_2,x_2))= H(a_1,x_1) *_2  H(a_2,x_2)$
is equivalent to
$$(a_1*a_2 + f_1(x_1,x_2) +c(x_1*x_2), x_1*x_2) = $$
$$ ((a_1+c(x_1))*(a_2+c(x_2))+  f_2(x_1,x_2),x_1*x_2) \mbox { thus}$$
$ta_1+ (1-t)a_2 + c(x_1*x_2) + f_1(x_1,x_2)= t(a_1+c(x_1) + (1-t)(a_2+c(x_2))+ f_2(x_1,x_2)$ so
$f_1(x_1,x_2) - f_2(x_1,x_2)= tc(x_1) + (1-t)c(x_2)- c(x_1x_2)= (\partial c)(x_1,x_2)$.
We can say that the second cohomology (here (twisted) rack cohomology) $H^2(X,A)$ describes (shelf)
extensions of $X$ by $A$ of type described above, modulo described above equivalence, for an abelian group $A$. 

The dynamical cocycle is given by $\phi_{a_1,a_2}(x_1,x_2)= ta_1 + (1-t)a_2+ f(x_1,x_2)$, \cite{CKS}.\\
II. Another family of extensions is given by the hull construction for a multi-shelf (multi-RD-system)
 of Patrick Dehornoy and David Larue \cite{Deh-2,Lar},
and its ``$G$-group" generalization (which we call a {\it twisted hull}) by Ishii, Iwakiri, Jang, and Oshiro \cite{IIJO}.\\
For a hull construction we need a distributive set of binary operations on $A$ indexed by elements of $X$,
that is $(a*_xb)*_yc=(a*_yc)*_x(b*_yc)$, and the ``hull" shelf structure on $A\times X$ is given by:
$$(a_1,x_1)*(a_2,x_2)=(a_1*_{x_2}a_2,x_1).$$ To see our construction as obtained from a dynamical 
cocycle we put trivial operation on $X$ ($x*y=x$), and the dynamical cocycle is given by
$\phi_{a_1,a_2}(x_1,x_2)= a_1*_{x_2}a_2$. \\

\begin{remark}\label{Remark 8.2}
 If $A=F(X)$ is a fee group on free generators $X$, then the hull $A\times X$ related to 
the distributive set of operations $*_x$ on $A$ given by $a_1*_xa_2=  a_1a_2^{-1}xa_2$ is a free rack 
generated by $X$ (denoted by $FR(X)$) as defined by Fenn and Rourke \cite{F-R} (see also \cite{CKS}). 
To summarize, we have then $FR(X)=F(X)\times X$ with $(a_1,x_1)*(a_2,x_2)=(a_1*_{x_2}a_2,x_1)=(a_1a_2^{-1}x_2a_2,x_1)$. 
\end{remark}

The $G$-group generalization of hull to twisted hull, relaxes 
condition that $X$ is indexing distributive set of operation and we allow ``twisted distributivity".
That is: $(a*_xb)*_yc=(a*_yc)*_{x*y}(b*_yc)$. Thus $X$ indexes operations on $A$ satisfying 
``twisted distributivity". In this case the shelf structure on $A\times X$ is given by:
$$(a_1,x_1)*(a_2,x_2)=(a_1*_{x_2}a_2,x_1*x_2).$$ The fundamental example leading to ``twisted distributivity"
was already given by Joyce: Let $G$ be a group and $X$ be a subgroup of $hom(G,G)$. Then we define 
$g_1*_{x}g_2= x(g_1g_2^{-1})g_2$ and we get:
$$(g_1*_{x_2}g_2)*_{x_3}g_3=(g_1*_{x_3}g_3)*_{x_2*x_3}(g_2*_{x_3}g_3)$$ 
where $x_2*x_3=x_3x_2x_3^{-1}$. \\
Twisted distributivity is illustrated in Figure 8.1 below:\\ \ \\
\centerline{\psfig{figure=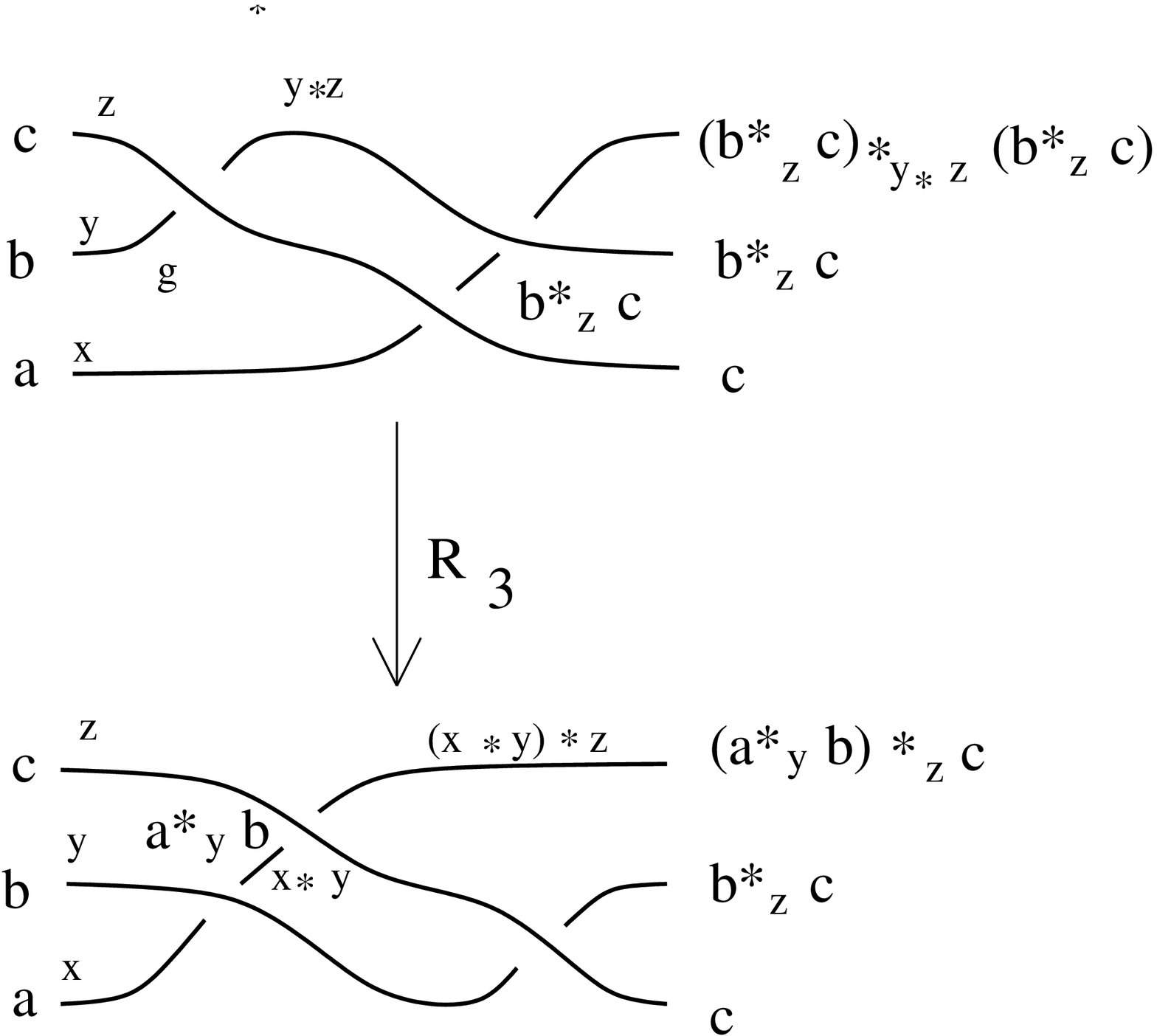,height=8.8cm}}\\ \ \\
\centerline{Figure 8.1; Twisted distributivity} 
\subsection{Extensions in entropic case}\label{Subsection 8.2}
Let $(X;*)$ be an entropic magma, that is $*$ satisfies the entropic identity:
$(a*b)*(c*d)= (a*c)*(b*d)$. Let also  $A$ be an abelian group with a given pair of commuting 
homomorphisms $t,s: A \to A$ and a constant $a_0\in A$; we consider $(A;*)$ as an entropic magma 
with an affine action $a*b= ta +sb + a_0$.
Then we define a binary operation on $A\times X$ by 
$(a_1*x_1)*(a_2*x_2) = (a_1*a_2 + f(x_1,x_2),x_1*x_2)$.
In order for $A\times X$ to be entropic magma we need entropic condition, or equivalently 
$\phi_{a_1,a_2}(x_1,x_2) = a_1*a_2 + f(x_1,x_2)$ should be and entropic dynamic cocycle.
This leads to entropic cocycle condition:
$$tf(x_1,x_2) - tf(x_1,x_3) +sf(x_3,x_4) - sf(x_2,x_4) + f(x_1*x_2,x_3*x_4) - f(x_1*x_3,x_2*x_4)
=0.\footnote{Calculation
is as follows: Entropic condition,
$((a_1,x_1)*(a_2,x_2))*((a_3,x_3)*(a_4,x_4))= ((a_1,x_1)*(a_3,x_3))*((a_2,x_2)*(a_4,x_4))$
 gives, after expanding each side:\\
$((a_1,x_1)*(a_2,x_2))*((a_3,x_3)*(a_4,x_4))=$
$ (a_1*a_2 + f(x_1,x_2),x_1*x_2)*(a_3*a_4 + f(x_3,x_4),x_3*x_4) =$
$((a_1*a_2)*(a_3*a_4) + tf(x_1,x_2) + sf(x_3,x_4) + f(x_1*x_2,x_3*x_4),(x_1*x_2)*(x_3*x_4))$\\
Similarly $((a_1,x_1)*(a_3,x_3))*((a_2,x_2)*(a_4,x_4))=$
$((a_1*a_3)*(a_2*a_4) + tf(x_1,x_3) + sf(x_2,x_4) + f(x_1*x_3,x_2*x_4),(x_1*x_3)*(x_2*x_4))$, which
reduces to (entropic) cocycle condition
$ tf(x_1,x_2) + sf(x_3,x_4) + f(x_1*x_2,x_3*x_4) = tf(x_1,x_3) + sf(x_2,x_4) + (x_1*x_3,x_2*x_4).$
}
$$

The above formula may serve as a hint how to define (co)homology in entropic case \cite{N-P-4}.

In particular $\partial: RX^4 \to RX^2$ may be given by:
$$\partial(x_1,x_2,x_3,x_4)= $$
$$t(x_1,x_2) - t(x_1,x_3) +s(x_3,x_4) - s(x_2,x_4) + (x_1*x_2,x_3*x_4) - (x_1*x_3,x_2*x_4)$$
and $\partial: RX^2 \to RX$  may be given by:
$\partial(x_1,x_2)= tx_1 - x_1*x_2 + sx_2$ (here it agrees with the rack case for $s=1-t$).

\begin{remark}
An important, but not yet fully utilized, observation from \cite{N-P-4}, is that we can 
consider atomic boundary functions\\
 $\partial^{(*)}((x_1,x_2,x_3,x_4)=  (x_1*x_2,x_3*x_4) - (x_1*x_3,x_2*x_4)$
and $\partial^{(*)}(x_1,x_2)= - x_1*x_2$, and consider also the left trivial binary operations $x*_0y=x$, 
 and  the right trivial binary operations $x*_{\sim}y=y$ and then to recover $\partial$ as 
a three term entropic boundary function, for the multi-entropic system $(*,*_0,*_{\sim})$,
 by the formula $\partial = \partial^{(*)}- t\partial^{(*_0)}- s\partial^{(*_{\sim})}$.
\end{remark}


\section{Degeneracy for a weak and very weak simplicial module}\label{Section 9}

We expand here on Subsections 3.2 and 3.3 and discuss degenerate part of distributive homology 
in the general context of weak and very weak simplicial modules.

Quandle homology is build in analogy to group homology or Hochschild homology of associate structures.
In the unital associative case we deal with simplicial sets (or modules) and it is a classical result
of Eilenberg and Mac Lane
that the degenerate part of a chain complex is acyclic so homology and normalized homology are isomorphic
(see Subsection 3.2).
It is not the  case for distributive structures, e.g. for quandles or spindles. Quandle homology or even one
term distributive homology of spindles may have nontrivial degenerate part. The underlining homological 
algebra structure is a weak simplicial module and in this case the degenerate part is not necessarily 
acyclic and the best one can say is that the
degenerate part has a natural filtration so yields a spectral sequence which can be used to study
degenerate homology. In the concrete case of quandle homology (motivated by and applicable to knot theory)
it is proven that the homology (called the rack homology) splits into degenerate and normalized (called the
quandle homology) parts \cite{L-N}. Otherwise no clear general connection between degenerate and quandle
part were observed. We prove
in the joint paper with Krzysztof Putyra \cite{Pr-Pu-2}
that the degenerated homology of a quandle is fully determined by quandle homology via a K\"unneth type formula.


\section{Degeneracy for a weak simplicial module}\label{Section 10}

Here we give a few general observations about degenerate part of a weak simplicial module. They are related 
to concrete work in the distributive case done in \cite{Pr-Pu-2}.

Consider a weak simplicial module $(C_n,d_i,s_i)$ (see Subsection 3.2 and \cite{Prz-5}).
As checked in Corollary 3.4, the filtration by degenerate elements $F^p_n=span(s_0(C_{n-1}),...,s_p(C_{n-1}))$ 
is preserved by the boundary operation $\partial_n=\sum_{i=0}^n(-1)^id_i$. In Subsection 3.3 we constructed 
a degenerate bicomplex $(E^0_{i,j},d^v,d^h)$

We discuss here the fact that a weak simplicial complex has also dual filtration (or better to say 
it has  left and write filtrations). 
We define the dual (or opposite) filtration $\hat F^p_n=span(s_{n-1}(C_{n-1}),...,s_{n-p}(C_{n-1}))$.

We start our dual description from a presimplicial module:

If $(C_n;d_i)$ is a presimplicial module then we define $\hat d_i = d_{n-i}$ and notice that
$(C_n;\hat d_i)$ is also a presimplicial module with $\hat\partial_n = (-1)^n\partial_n$ and
unchanged homology. More generally we have:
\begin{proposition} \label{Proposition 10.1}
\item[(i)] If $(C_n;d_i)$ is a presimplicial module then $(C_n;\hat d_i)$ is also a presimplicial module.
\item[(ii)] If $s_i: C_n \to C_{n+1}$, $0\leq i \leq n$ are degenerate map, define $\hat s_i= s_{n-i}$.
Then if $(C_n;d_i,s_i)$ is a (weak, or very weak) simplicial module then $(C_n;\hat d_i,\hat s_i)$ is also a a
(weak or very weak) simplicial module.
\end{proposition}
\begin{proof} (i) For $i<j$ we have $n-j<n-i$, so:
$$\hat d_i\hat d_j= \hat d_i d_{n-j}=d_{n-1-i}d_{n-j}\stackrel{(1)}{=} d_{n-j}d_{n-i}= \hat d_{j-1}\hat d_i.$$
(ii) For a better presentation let us list conditions of a simplicial module for $(C_n,\hat d_i, \hat s_i)$,
 one by one:
$$ (\hat 1) \ \ \  \hat d_i\hat d_j = \hat d_{j-1}\hat d_i\ for\ i<j. $$
$$(\hat 2)\ \ \ \hat s_i\hat s_j=\hat s_{j+1}\hat s_i,\ \ 0\leq i \leq j \leq n, $$
$$ (\hat 3) \ \ \ \hat d_i\hat s_j= \left\{ \begin{array}{rl}
 \hat s_{j-1}\hat d_i &\mbox{ if $i<j$} \\
\hat s_{j}\hat d_{i-1} &\mbox{ if $i>j+1$}
       \end{array} \right.
$$
$$ (\hat 4') \ \ \ \hat d_i\hat s_i=\hat d_{i+1}\hat s_i.$$
$$ (\hat 4) \ \ \ \hat d_i\hat s_i=\hat d_{i+1}\hat s_i= Id_{M_n}. $$
Proposition \ref{Proposition 10.1} follows from the following lemma.
\begin{lemma}\label{Lemma 10.2}
Consider $(C_n;d_i,s_i)$ and its complementary (dual) $(C_n;\hat d_i,\hat s_i)$, then
conditions $(x)$ and $(\hat x)$ are equivalent
\end{lemma}
\begin{proof} Equivalence of (1) and $(\hat 1)$ was already established. Other parts are equally simple but
we prove them for completeness:\\
$(2) \Leftrightarrow (\hat 2)$ (we assume $i\leq j$ or equivalently $n-j\leq n-i$): \\
$$\hat s_i\hat s_j= \hat s_i s_{n-j}=s_{n+1-i}s_{n-j}\stackrel{(2)}{=} s_{n-j}s_{n-i}= \hat s_{j+1}\hat s_i.$$
$(3) \Leftrightarrow (\hat 3)$ First assume that $i<j$ (i.e. $n+1-i > n-j+1$) then:\\
$$\hat d_i\hat s_j= d_{n+1-i}s_{n-j}\stackrel{(3)}{=} s_{n-j}d_{n-i}= \hat s_{j-1}\hat d_i.$$
Second assume that $i>j+1$ (i.e. $n+1-i <n-j$), then:\\
$$\hat d_i\hat s_j= d_{n+1-i}s_{n-j}\stackrel{(3)}{=} s_{n-j-1}d_{n-i+1}= \hat s_{j}\hat d_{i-1}.$$
$(4') \Leftrightarrow (\hat 4')$. We have: \\
$$\hat d_i\hat s_i= d_{n+1-i}s_{n-i}\stackrel{(4')}{=} d_{n-i}s_{n-i}= \hat d_{i+1}\hat s_i.$$
$(4) \Leftrightarrow (\hat 4)$. We have: \\
$$\hat d_i\hat s_i= d_{n+1-i}s_{n-i}\stackrel{(4)}{=} Id.$$ \\
\end{proof}
\end{proof}

\begin{remark}\label{Remark 10.3}
If $C_n=ZX^{n+1}$ then we can consider the map
$\hat I: (C_n;d_i) \to (C_n;\hat d_i)$ given by $\hat I (x_0,x_1,...,x_n)=(x_n,...,x_1,x_0)$
(or succinctly $\hat I ({\bf x})=\hat {\bf x})$).
\end{remark}

Our results ( Proposition \ref{Proposition 10.1} and Lemma \ref{Lemma 10.2} hold for very weak simplicial modules, 
weak simplicial modules, and simplicial modules.
In particular, for a weak simplicial module the dual filtration of $C^D_n$,
 $\hat F^p_n =span (\hat s_0(C_{n-1}), \hat s_1(C_{n-1}),..., \hat s_p(C_{n-1})$ leads to 
a spectral sequence and a bicomplex.
Here we give a few general remarks to summarize basic facts:\\
A weak simplicial module yields two filtrations:
$F^p_n$ and the dual (complementary) one $\hat F^p_n$.
By the definition we have
$$\partial_n= \sum_{i=0}^n (-1)^id_i= (-1)^n \sum_{i=0}^n (-1)^i\hat d_i=(-1)^n\hat\partial_n$$
Furthermore, on $s_p(C_{n-1})$ we have:
$$\partial_n s_p=\sum_{i=0}^n (-1)^id_is_p = \sum_{i=0}^{p-1}(-1)^id_is_p + (-1)^pd_ps_p + (-1)^{p+1}d_{p+1}s_p +
\sum_{i=p+2}^{n}(-1)^id_is_p \stackrel{(4')}{=}$$
$$\sum_{i=0}^{p-1}(-1)^id_is_p + \sum_{i=p+2}^{n}(-1)^id_is_p \stackrel{(3)}{=}$$
$$\sum_{i=0}^{p-1}(-1)^is_{p-1}d_i + \sum_{i=p+2}^{n}(-1)^is_pd_{i-1}.$$
Clearly $\sum_{i=0}^{p-1}(-1)^is_{p-1}d_i(C_{n-1})$ belongs to $F^{p-1}_n$.

The formulas above lead to the bicomplex with $E^0_{p,q}=M_{p,q}=F^p_n/F^{p-1}_n$, where $n=p+q$, \
$d^h=\sum_{i=0}^{p-1}(-1)^id_is_p$ and $d^v=\sum_{i=p+2}^{n}(-1)^id_is_p.$
The equality $d^hd^h=0=d^vd^v$ and $d^hd^v=-d^vd^h$ follows directly from the weak simplicial module structure.

If we replace the filtration $F^p_n$ by $\hat F^p_n$ we see that the spectral sequence is
modified; it is analogous, but not the same, as when $d^h$ is replaced by $d^v$.

\begin{remark}\label{Remark 10.4}
An acute observer will notice immediately\footnote{Victoria Lebed studied this before me in context of her
prebraided category.} that we deal not only with a bicomplex but also with pre-bisimplicial category (set or module).
For completeness I recall definitions after \cite{Lod-1}, page 459:\\
We define a bisimplicial object but in a same vain we can define pre-bisimplicial category, and weak bisimplicial
category:
`` By definition a bisimplicial object in a category $\mathcal C$ is a functor
$$\mathcal{X}: \Delta^{op}\times \Delta^{op} \to \mathcal C.$$
Such a bisimplicial object can be described equivalently by a family of objects $M_{p,q}$, $p\geq 0,q\geq 0$, together
with horizontal and vertical faces and degeneracies:
$$ d_i^h: M_{p,q} \to M_{p-1,q},\ \ \ s_i^h: M_{p,q} \to M_{p+1,q}, \mbox{ where } 0\leq i \leq p$$
$$d_i^v: M_{p,q} \to M_{p,q-1},\ \ \ s_i^v: M_{p,q} \to M_{p,q+1}, \mbox{ where } 0\leq i \leq q$$
which satisfy the classical simplicial relations horizontally and vertically and such that horizontal and vertical
operations commute. For any bisimplicial set $X$ there are three (homeomorphic) natural ways to make geometric realization,
$|X|$ of $X$. Loday notes that any bisimplicial set $X$ gives rise to the bisimplicial module $RX$ and
$H_*(|X|,R) = H_*(Tot(RX))$,  \cite{Lod-1}.
\end{remark}

\begin{example}\label{Example 10.5}
 A natural example of a (pre)-bisimplicial set or module is obtained by a Cartesian (or tensor)
product of (pre)-simplicial sets (or modules). Namely:
\begin{enumerate}
\item[($\times$)] Let $M_{p,q}= C_p\times C'_q$ where $(C_n,d_i)$ and $(C'_n,d'_i)$ are (pre)simplicial sets.
We define $d_i^h= d_i\times Id_{C'_q}$ and $d_i^v= Id_{C_p} \times d'_i$. In the case $(C_n,d_i,s_i)$ and
$(C'_n,d'_i,s'_i)$
are (weak)-simplicial sets we get $M_{p,q}$ a (weak)-simplicial set with
$s_i^h= s_i\times Id_{C'_q}$ and $s_i^v= Id_{C_p} \times s'_i$.
\item[($\otimes$)]  Let $M_{p,q}= C_p\otimes C'_q$ where $(C_n,d_i)$ and $(C'_n,d,_i)$ are (pre)simplicial modules.
Then $M_{p,q}$ is a (pre)simplicial modules with $d_i^h= d_i\otimes Id_{C'_q}$ and $d_i^v= Id_{C_p} \otimes d'_i$.
Similarly in the case $(C_n,d_i,s_i)$ and $(C'_n,d'_i,s'_i)$ are (weak)-simplicial modules.
\\

\end{enumerate}
\end{example}
\subsection{Right filtration of degenerate distributive elements}\label{Subsection 10.1}\ \\
We restrict ourselves here to a weak simplicial module yielded by a distributive structures.

Let $(X;*)$ be a spindle that is a magma which is right distributive ($(a*b)*c=(a*c)*(b*c)$)
and idempotent ($a*a=a$).

\begin{definition}\label{Definition 10.6}
\begin{enumerate}
\item[(i)]
Let $\hat s_i= s_{n-i}: C_n \to C_{n+1}$ is given by
$$\hat s_i (x_0,x_1,...,x_n)= (x_0,..,x_{n-i-1},x_{n-i},x_{n-i}, x_{n-i+1},...,x_n),$$ that is we double the letter
on the position $n-i$ (or $i$ from the end) if we count from zero.
\item[(ii)] We define $\hat F^p_n = span(\hat s_0(C_{n-1}), \hat s_1(C_{n-1}),...,\hat s_{n-1}(C_{n-1}))$
in $C_n(X)$. $\hat F^p_n$ form a boundary coherent filtration of $C_n(X)$:
$$ 0\subset \hat F^0_n \subset  \hat F^1_n\subset  \hat F^{n-1}_n=  C^D_n.$$
\item[(iii)] Let $ \hat Gr_p^n$ be the associated graded group: $ \hat Gr_p^n=  \hat F^p_n/ \hat F_{p-1}^n$.
\end{enumerate}
\end{definition}
If, as before, we define face maps $\hat d_i= d_{n-i}$ then $(C_n,\hat d_i,\hat s_i)$ is a
weak simplicial module. Thus $\hat F^p_n$ is a graded filtration ripe for the spectral sequence.
(We already noticed that $\hat\partial_n=\sum_{i=0}^n(-1)^i\hat d_i = (-1)^n\partial_n$.)

We consider the spectral sequence of the filtration starting from the initial page 
$ \hat E^0_{p,q}=  \hat Gr^p_{p+q}$. It is the first main observation of \cite{Pr-Pu-2} that the spectral 
sequence $(\hat E^r_{p,q},\hat \partial^r_{p,q})$ stabilizes on the first page, and eventually one term 
spindle homology can be computed easily from the normalized part.

\subsection{Integration maps $\hat u_i: \hat F^p_n \to \hat F^{p-1}_n$}\ \\
The main tools to show that the spectral sequence $\hat E^r_{p,q}$ stabilizes on the first page are 
the maps (which we can call integration)
 $\hat u_i: \hat F^p_n \to \hat F^{p-1}_n$, illustrated in Figure 10.1 below\footnote{In braid
notation, $\hat u_i$ can be expressed as $\sigma_p\sigma_{p-1}...\sigma_{i+1}\sigma_{i}$.}, they serve to show that
the right degenerated filtration spectral sequence has all $d^r$ ($r>0$) trivial and that
homology splits.

\centerline{\psfig{figure=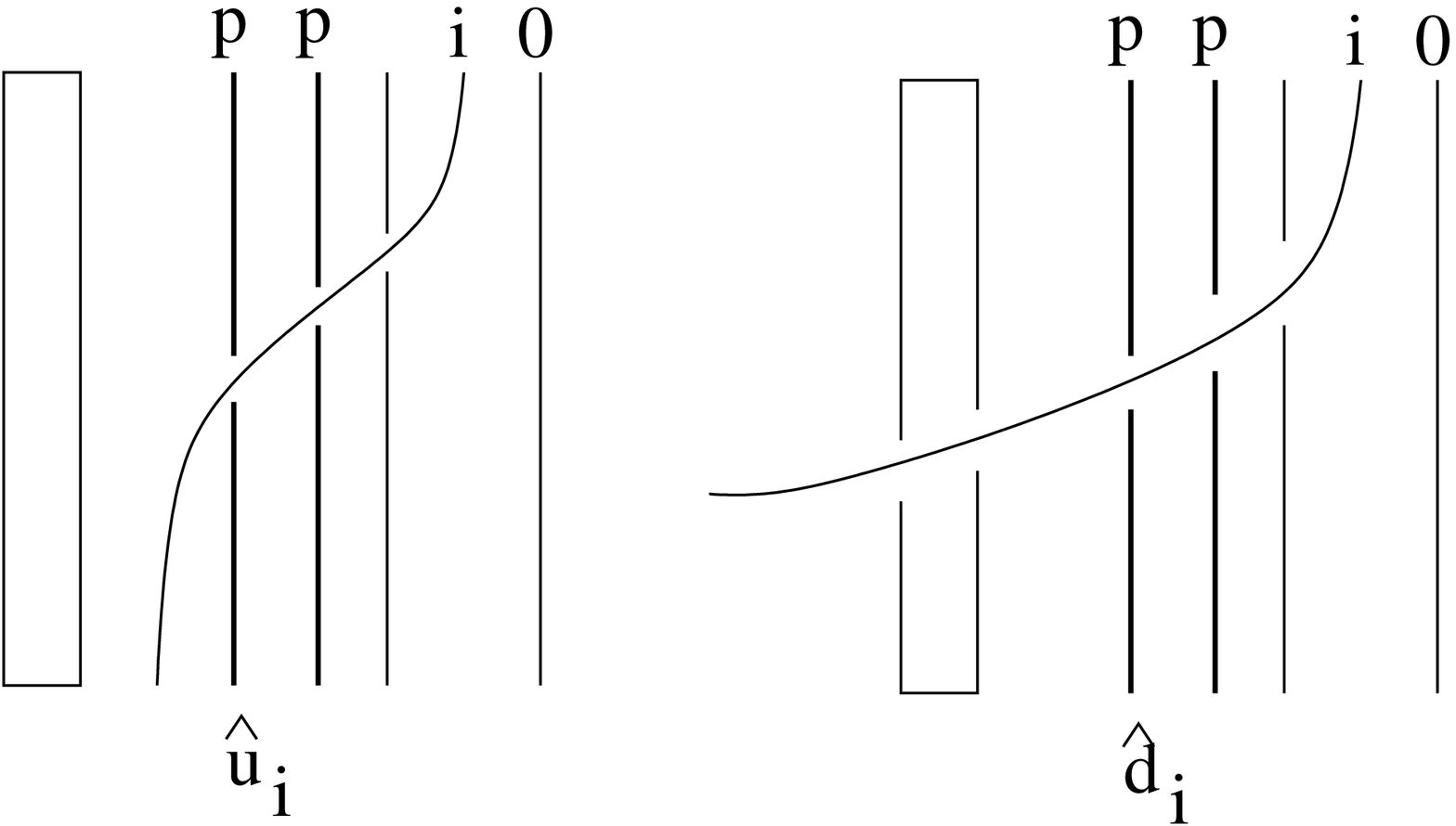,height=5.9cm}}
\centerline{Figure 10.1; The maps $\hat u_i$ and $\hat d_i$}

We check that the maps $\hat u_i$ satisfy:\\
(1) for $i<p$: $\hat d_i(y)= \hat d_{p+1}\hat u_i(y)$, where $y\in \hat s_p(C_{n-1})$.\\
(2) For $p> i_2 >i_1$:
$\hat d_{p+1}\hat u_{i_2-1}\hat u_{i_1}= \hat d_{i_1}\hat u_{i_2}$.\\
from this follows \\
(2') For $p> i_2 >i_1$:  $\hat d_{p}\hat u_{i_2-1}\hat u_{i_1} = \hat d_{i_2-1}\hat u_{i_1}$.\\
(2) and (2') are illustrated in Figure 10.2 and 10.3 below.
$$ (j) \mbox{ For } p>i_j > i_{j-1}>...>i_1\geq 0 \mbox{ one has }
\hat d_{p+1}\hat u_{i_j-j+1}...\hat u_{i_2-1}\hat u_{i_1} = \hat d_{i_1}\hat u_{i_j-j+2}...\hat u_{i_3-1}\hat u_{i_2}.$$
From this follows:\\
$$(j\geq k)\ \ \ \ \ \ \ \hat d_{p+1-k+1}\hat u_{i_j-j+1}...\hat u_{i_k-k+1}...\hat u_{i_2-1}\hat u_{i_1} = $$
 $$\hat d_{i_k-k+1}\hat u_{i_j-j+2}...\hat u_{i_{k+1}-(k+1)+2}\hat u_{i_{k-1}-(k-1)+1}...\hat u_{i_2-1}\hat u_{i_1}.$$

\centerline{\psfig{figure=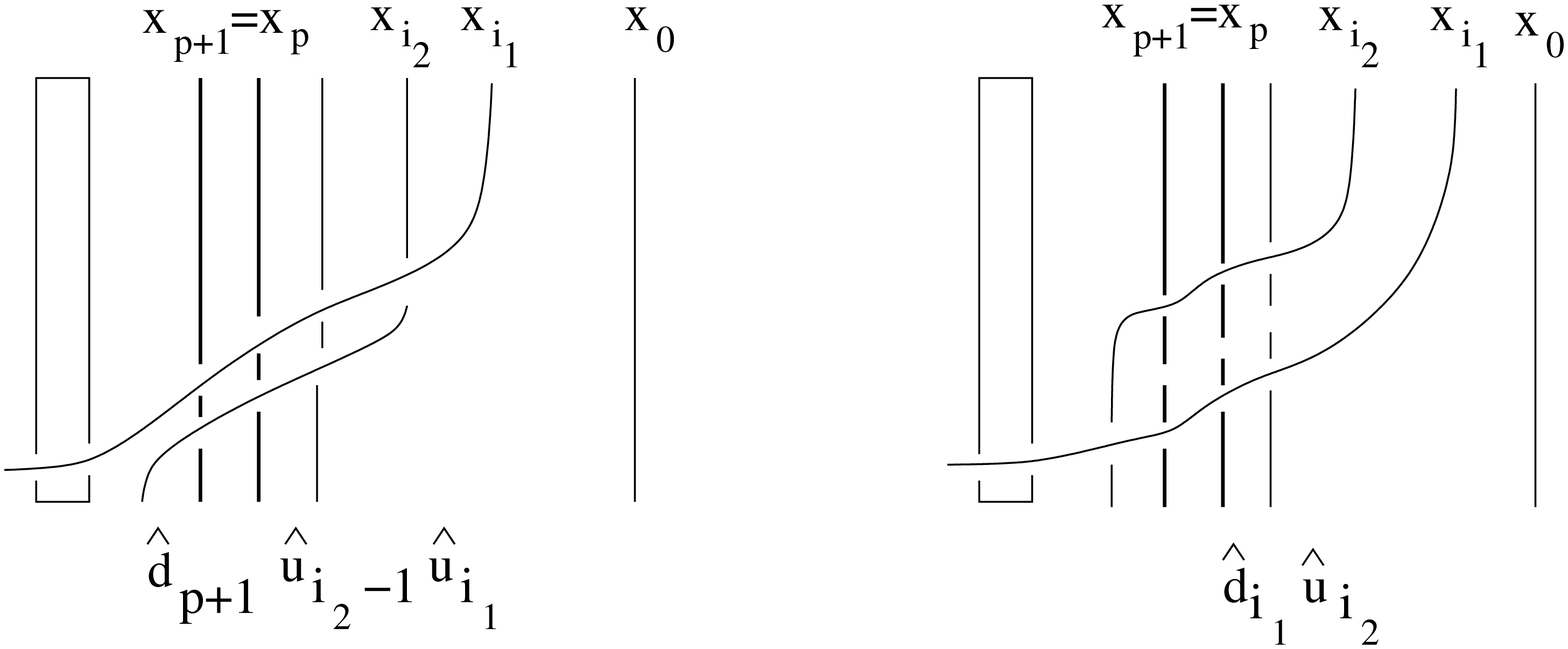,height=6.1cm}}
\centerline{Figure 10.2; Property (2): $\hat d_{p+1}\hat u_{i_2-1}\hat u_{i_1}= \hat d_{i_1}\hat u_{i_2}$ }
\ \\ \ \\
\centerline{\psfig{figure=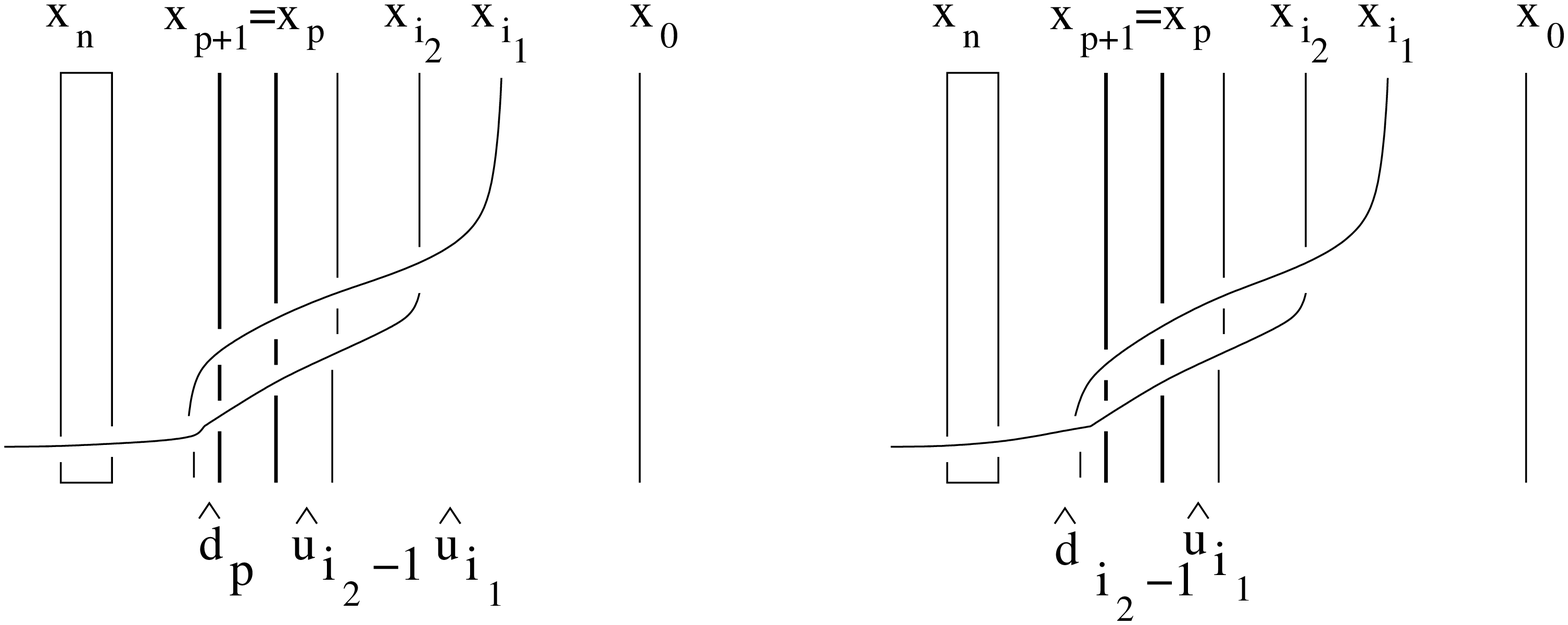,height=6.3cm}}
\centerline{Figure 10.3; Property (2): $\hat d_{p}\hat u_{i_2-1}\hat u_{i_1}= \hat d_{i_2-1}\hat u_{i_1}$ }

\begin{remark}\label{Remark 10.7}
We wrote formulas for $\hat d_s\hat u_{i_k-k+1}...\hat u_{i_1-1}\hat u_{i_1}$ only in the case of
$ p+1-k \leq s\leq p+1$ as
it is needed to compute the degenerate part of one term distributive homology of a spindle.
\end{remark}


\subsection{Weak simplicial modules with integration }
Here we formalize above equations to define {\it weak simplicial module} with integration
for which spectral sequence  stabilizes on $E^1_{p,q}$ and homology splits.

We consider a weak simplicial module $(C_n,d_i,s_i)$ with an additional structure,
$u_i: s_p(C_{n-1}) \to s_{p-1}(C_{n-1})$ for $i<p$, where maps $u_i$ satisfy condition (j) for any $j$.
This additional structure allows us to split  degenerate part.

\begin{definition}\label{Definition 10.8}
We say that  $(C_n,d_i,s_i,u_i)$ is a weak simplicial module with integration if
 $(C_n,d_i,s_i)$ is a weak simplicial module, $u_i: F^p_n \to F^{p-1}_n$ ($0\leq i<p$) and
the following hold:
\begin{enumerate}
\item[(1)] $d_{p+1}u_i=d_i$ ($i<p$).
\item[(j)] $d_{p+1}u_{i_j-j+1}...u_{i_2-1}u_{i_1} = d_{i_1}u_{i_j-j+2}...u_{i_3-1}u_{i_2}$,\\ where
$p> i_j >...>i_1\geq 0$.
\end{enumerate}
\end{definition}
A weak simplicial module with integration leads to bicomplex which stabilizes on $E^1_{p,q}$ and eventually
splits using the maps $f^p_n: F^p_n/F^{p-1}_n \to \bigoplus_{i=0}^{p}F^i_n/F^{i-1}_n$.

\section{Degeneracy for a very weak simplicial module}\label{Section 11}

We have considered, previously, the degenerate subcomplex in the case of a weak simplicial module, 
however if $(C_n,d_i,s_i)$ is only a very weak simplicial module, that is $d_is_i$ is not necessarily 
equal to $d_{i+1}s_i$, we can still construct the analogue of a degenerate subcomplex (and degenerate 
filtration (compare Remark. 3.4 in \cite{Prz-5}).

Let ${\mathcal C}=(C_n,d_i,s_i)$ be a very weak simplicial module, that is axioms (1)-(3) of 
Definition \ref{Definition 3.3} hold.
We do not necessarily have the condition $d_is_i-d_{i+1}s_i$ is equal to zero so it is of
interest to study an obstruction to zero: $t_i= d_is_i-d_{i+1}s_i$.
We have;
\begin{lemma}
Let $t_i:C_n \to C_n$ where $t_i= d_is_i-d_{i+1}s_i$ in a very weak simplicial module, then:
\begin{enumerate}
\item[(i)]
$$
d_it_j= \left\{ \begin{array}{rl}
 t_{j-1}d_i &\mbox{ if $i<j$} \\
  0 &\mbox{ if $i=j$ } \\
t_{j}d_{i} &\mbox{ if $i>j$}
       \end{array} \right.
$$
\item[(ii)] $t_it_j=t_jt_i$.
\item[(iii)] It follows from (i) that we have boundary preserving  filtrations:
$$ F_0^t= t_0(C_n) \subset F_1^t=span(t_0(C_n),t_1(C_n)) \subset...\subset F_n^t=span(t_0(C_n),...,t_n(C_n))=F^t$$

\item[(iii)] We have also boundary preserving  filtrations:
$$F_0^{tD}=span(t_0(C_n),s_0(C_{n-1})\subset...\subset$$
$$ F_{n-1}^{tD}C_n=
span(t_1(C_n),s_1(C_{n-1}),...,t_{n-1}(C_n),s_{n-1}(C_{n-1})
\subset $$
$$ F_{n}^{tD}C_n= span(t_1(C_n),s_1(C_{n-1}),...,t_{n-1}(C_n),s_{n-1}(C_{n-1}),t_n(C_n))=F^{tD}.$$
Or the filtration $$0 \subset (F_0^D+F^{t})/F^{t} \subset ...\subset (F_{n-1}^D+F^{t})/F^{t}= F^{tD}/F^{t}$$
\item[(iv)] $(\partial_nt_p- t_p\partial_n)(t_p(C_n)) \subset t_{p-1}(C_n)$. In particular, $t_p$ is a chain map 
on $F^t_p(C_n^t)/ F^t_{p-1}(C_n^t)$.
\end{enumerate}
\end{lemma}

$F^{tD}$ is likely the best proxy of degenerated subchain complex so I call it
the generalized degenerated subchain complex of a very weak simplicial module.
The quotient $C_n/F_{n}^{tD}C_n$ is an analogue of a normalized chain complex
(in quandle theory the quandle chain complex for any distributive structure, not necessarily spindle or quandle).

\begin{proof} (i) ($i<j$ case): we have here, $d_it_j= d_id_js_j - d_{i}d_{j+1}s_j= 
d_{j-1}d_is_j - d_jd_is_j = d_{j-1}s_{j-1}d_i - d_js_{j-1}d_i = t_{j-1}d_i$, \\
($i=j$ case)a we have, : $d_jt_j = d_jd_js_j - d_{j}d_{j+1}s_j= 0$, \\
 ($i>j$ case): we have, $d_it_j= d_id_js_j - d_{i}d_{j+1}s_j= $
DO\\
(ii)
First we show that:
$$\mbox{ (ii') } t_is_j=s_jt_i \mbox{ for } i<j  \mbox { we have, stressing which
property is used:}$$
$$t_is_j =  (d_is_i- d_{i+1}s_i)s_j = (d_i-d_{i+1})s_is_j \stackrel{(2)s_is_j=s_{j+1}s_i}{=} 
(d_i-d_{i+1})s_{j+1}s_i =$$ 
$$ d_is_{j+1}s_i - d_{i+1}s_{j+1}s_i  \stackrel{(3)}{=} s_jd_is_i - s_jd_{i+1}s_i =
s_jt_i.$$
(Similarly we prove that $t_is_j =s_jt_i$ for $i>j$. 
Now we complete the proof that $t_it_j=t_jt_i$. We have, assuming $i<j$:
$$t_it_j= t_i(d_js_j -d_{j+1}s_j) \stackrel{(i)}{=}  d_jt_is_j -d_{j+1}t_is_j  \stackrel{(ii')}{=} $$
$$=d_js_jt_i - d_{j+1}s_jt_i = t_jt_i \mbox{ as needed.}$$
(iii)
We start from computing $ \partial_n t_p$ stressing each time which 
property is used:
$$ \partial_n t_p = \sum_{i=1}^n (-1)^id_it_p = \sum_{i=0}^{p-1}(-1)^id_it_p + 0 + \sum_{i=p+1}^n(-1)^id_it_p=$$
$$\sum_{i=0}^{p-1}(-1)^it_{p-1}d_i  + \sum_{i=p+1}^n(-1)^it_pd_i$$
This is proving that filtration $F_p^n$ is boundary preserving and also that we deal with bicomplex.
Thus we can construct the spectral sequence from the bicomplex.

\end{proof}

Again if we work with a shelf and filtration from the right then it seems to stabilize at $E^1_{p,q}$
and homology splits, like in the case of right degenerate filtration of a spindle, $\hat F^p_n$.
See Figure 11.1 below for graphical interpretation of $\hat t_i$. $\hat u_i^t$ will be
defined  by analogy to $\hat u_i$.

\centerline{\psfig{figure=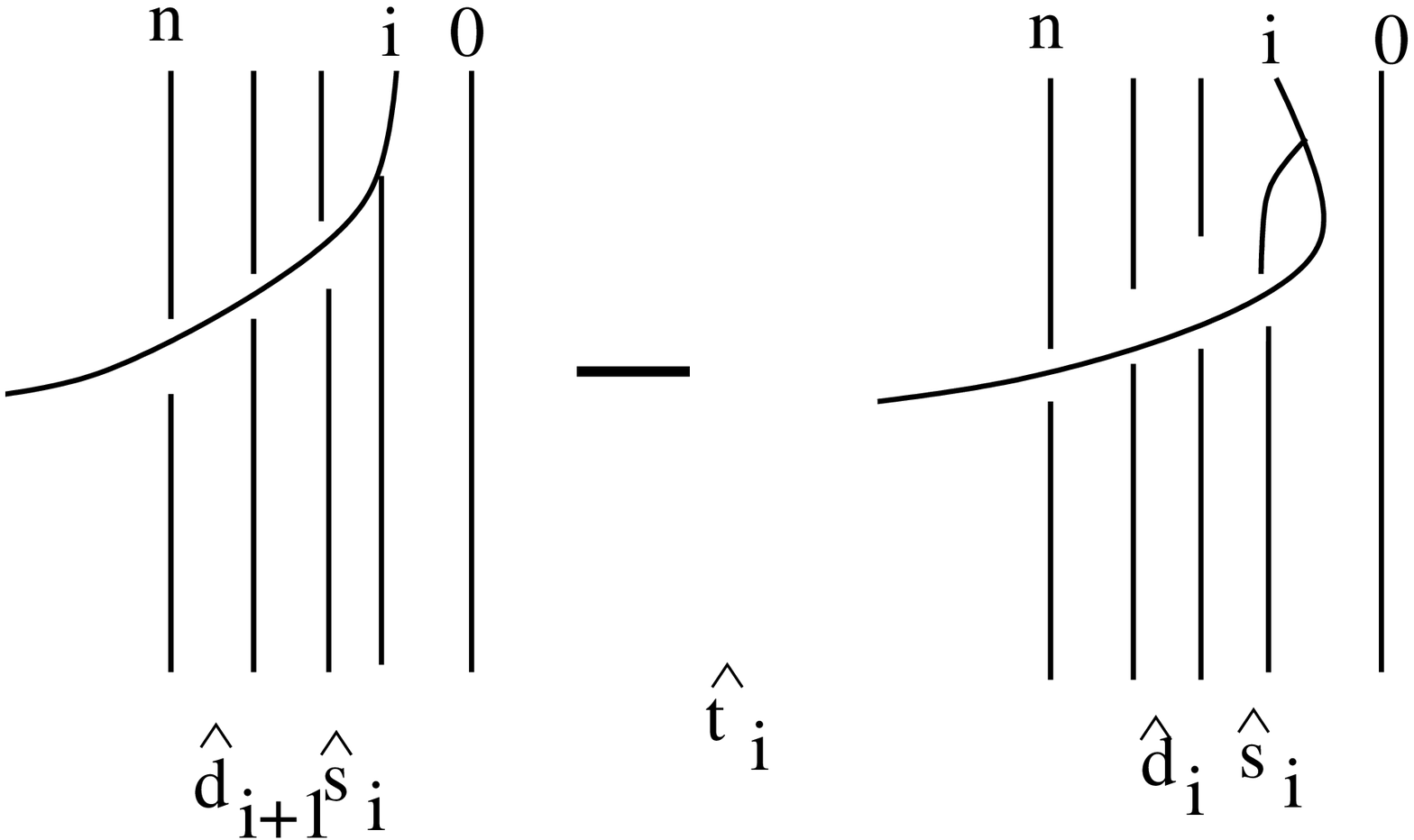,height=6.3cm}}
\centerline{Figure 11.1; The maps  $\hat t_i$ as the difference of $\hat d_{i+1}\hat s_i$ and $\hat d_{i}\hat s_i$}
\ \\

We leave to a reader the development of these ideas, specially in the case of multiterm distributive 
homology (compare \cite{Pr-Pu-2}).

\subsection{Introduction to t-simplicial objects}

We can extract properties of maps $t_i$ to obtain a new version of a simplicial object which
we will call $t$-simplicial object.
\begin{definition} Let ${\mathcal C}$ be a category and $(X_n,d_i,t_i)$ the sequence of objects $X_n$, $n\geq 0$,
and morphisms $d_i: X_n \to X_{n-1}$, $t_i: X_n \to X_{n}$, $0\leq i \leq n$ . We say that $(X_n,d_i,t_i)$ is
a $t$-simplicial object if the following four conditions hold (the first is the condition of a
presimplicial object):
\begin{enumerate}
\item[($1^t$)] $d_id_j=d_{j-1}d_i$ for $i<j$.
\item[($2^t$)] $t_it_j=t_jt_i$.
\item[($3^t$)] $$
d_it_j= \left\{ \begin{array}{rl}
 t_{j-1}d_i &\mbox{ if $i<j$} \\
t_{j}d_{i} &\mbox{ if $i>j$}
       \end{array} \right.
$$

\item[($4^t$)] $d_it_i=0$.
\end{enumerate}
\end{definition}

Let $(C_n,d_i,t_i)$ be a t-simplicial module, with $\partial_n=\sum_{i=0}^n(-1)^id_i$ then $(C_n,d_i,t_i)$
leads to a bicomplex $(C_{p,q},d^h,d^v)$ where $d^h$ and $d^v$ are defined up sign/shift by:
$$d^h= \sum_{i=0}^{p-1}(-1)^id_i \mbox{ and } d^v=\sum_{i=p+1}^n(-1)^id_i.$$

We also can define a $t$ analogue of a bisimplicial object, which we call a $t$-bisimplicial object,
$(C_{p,q},d_i^h,d_j^v)$, $0 \leq i \leq p$, $0\leq j \leq q$, with $d_i^h=d_i$ and $d^v_j=d_{p+1+i}$
(however some adjustment is needed to have $p+q=n$ (or just $p+q=n-1$).


\begin{problem}
The relation $d_it_i=0$ is crucial. Can one find for it general setting (say, $t_i$ as
the marker for horizontal and vertical parts of a bicomplex)? Other applications?
\end{problem}

\section{From distributive homology to Yang-Baxter homology}\label{Section 12}

We can extend the basic construction from the introduction, still using very naive point of view, as follows:
Fix a finite set $X$ and color semi-arcs of $D$ (parts of $D$ from a crossing to a crossing) by elements of $X$
 allowing different weights from 
some ring $k$ for every crossing (following statistical mechanics terminology we call these weights Boltzmann weights). 
We allow also differentiating between a negative and a positive crossing; see Figure 12.1. \\ \ \\
\centerline{\psfig{figure=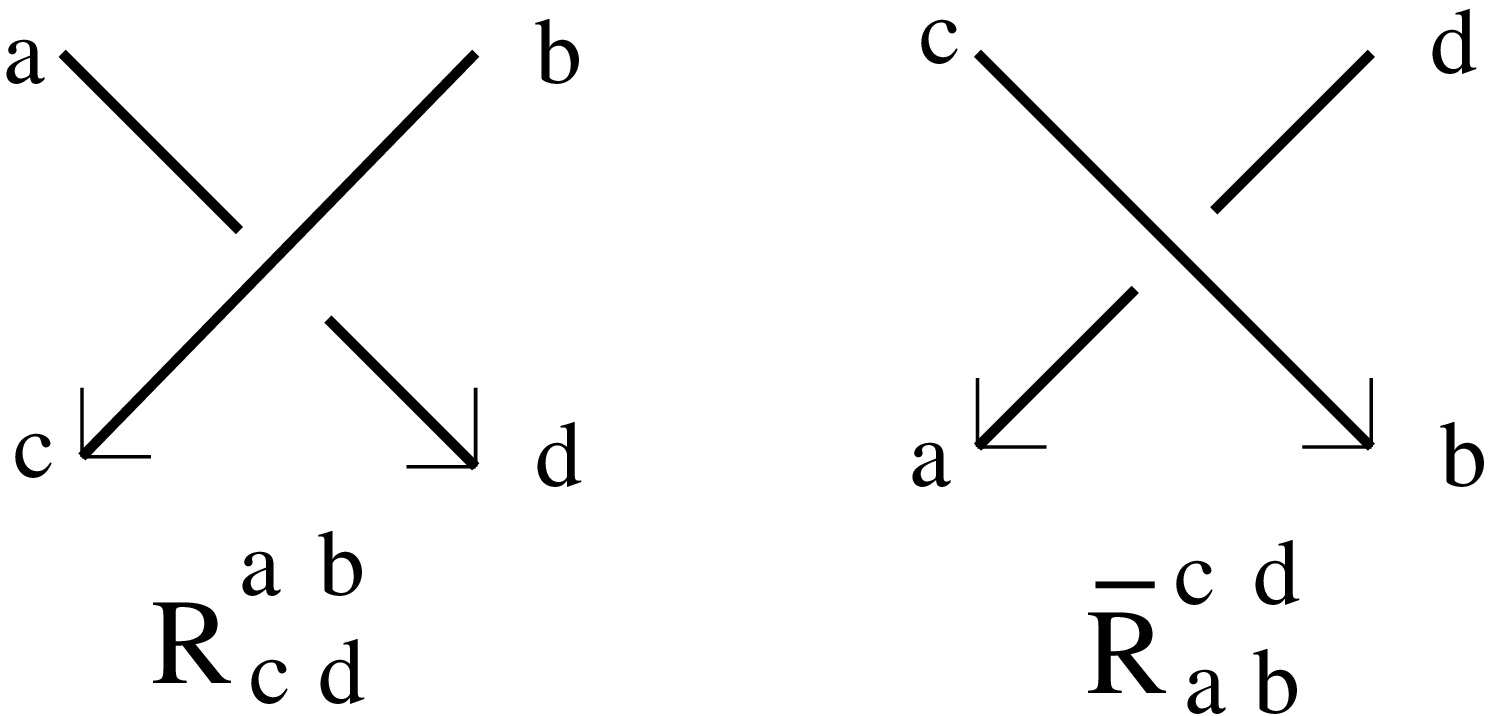,height=3.3cm}}\ \\ \ \\
\centerline{Figure 12.1; Boltzmann weights $R^{a,b}_{c,d}$ and $\bar R_{a,b}^{c,d}$ for positive and 
negative crossings}
\ \\ 
We can now generalize the number of colorings to state sum (basic notion of statistical physics) 
by multiplying Boltzmann weight over all crossings and adding over all colorings:
$$col_{(X;BW)}(X)= \sum_{\phi\in col_X(D)}\prod_{p\in \{crossings\}}\hat R^{a,b}_{c,d}(p)$$
where $\hat R^{a,b}_{c,d} $ is $R^{a,b}_{c,d}$ or $\bar R^{a,b}_{c,d}$ depending on 
whether $p$ is a positive or negative crossing (see \cite{Jones}). 
Our state sum is an invariant of a diagram but to get a link invariant we should test it on Reidemeister moves.
To get analogue of a shelf invariant we start from the third Reidemeister move with all positive crossings. 
Here we notice that, in analogy to distributivity, where passing through a positive crossing was 
coded by a map $R:X\times X \to X\times X$ with $R(a,b)=(b,a*b)$. Thus in the general case 
passing through a positive crossing is coded by a linear map $R:kX\otimes kX \to kX\otimes kX$ 
and in basis $X$ the map $R$ is given by the $|X|^2 \times |X|^2$  matrix with entries $(R^{a,b}_{c,d})$.
The third Reidemeister move leads to the equality of the following maps 
$V\otimes V \otimes V \to V\otimes V \otimes V$ where $V=kX$:
$$ (R\otimes Id)(Id\otimes R)(R\otimes Id)= (Id\otimes R)(R\otimes Id)Id\otimes R)$$
This is called the Yang-Baxter equation and $R$ is called a pre-Yang-Baxter operator.
If $R$ is additionally invertible it is called a Yang-Baxter operator. If entries of $R^{-1}$ are 
equal to $\bar R^{a,b}_{c,d}$ then the state sum is invariant under ``parallel" (directly oriented)
 second Reidemeister move.\footnote{We should stress that to find link invariants it suffices to use 
directly oriented second and third Reidemeister moves in addition to both first Reidemeister moves,
as we can restrict ourselves to braids and use the Markov theorem. This point of view was used in \cite{Tur}.}
 
For a given pre-Yang-Baxter operator we attempt to find presimplicial module, from which homology 
will be derived. 

Figure 12.2 below illustrate various graphical interpretation of the generating morphism $d_i$ of the  
presimplicial category $\Delta_{pre}^{op}$. They are related to homology of a set-theoretic 
Yang-Baxter equation of Carter-Kamada-Saito \cite{CES-2} and Fenn \cite{FIKM}, 
and to homology of Yang-Baxter equation of Eisermann 
\cite{Eis-1,Eis-2}. We should also acknowledge stimulating observations by Ivan Dynnikov.\\ \ \\
\centerline{\psfig{figure=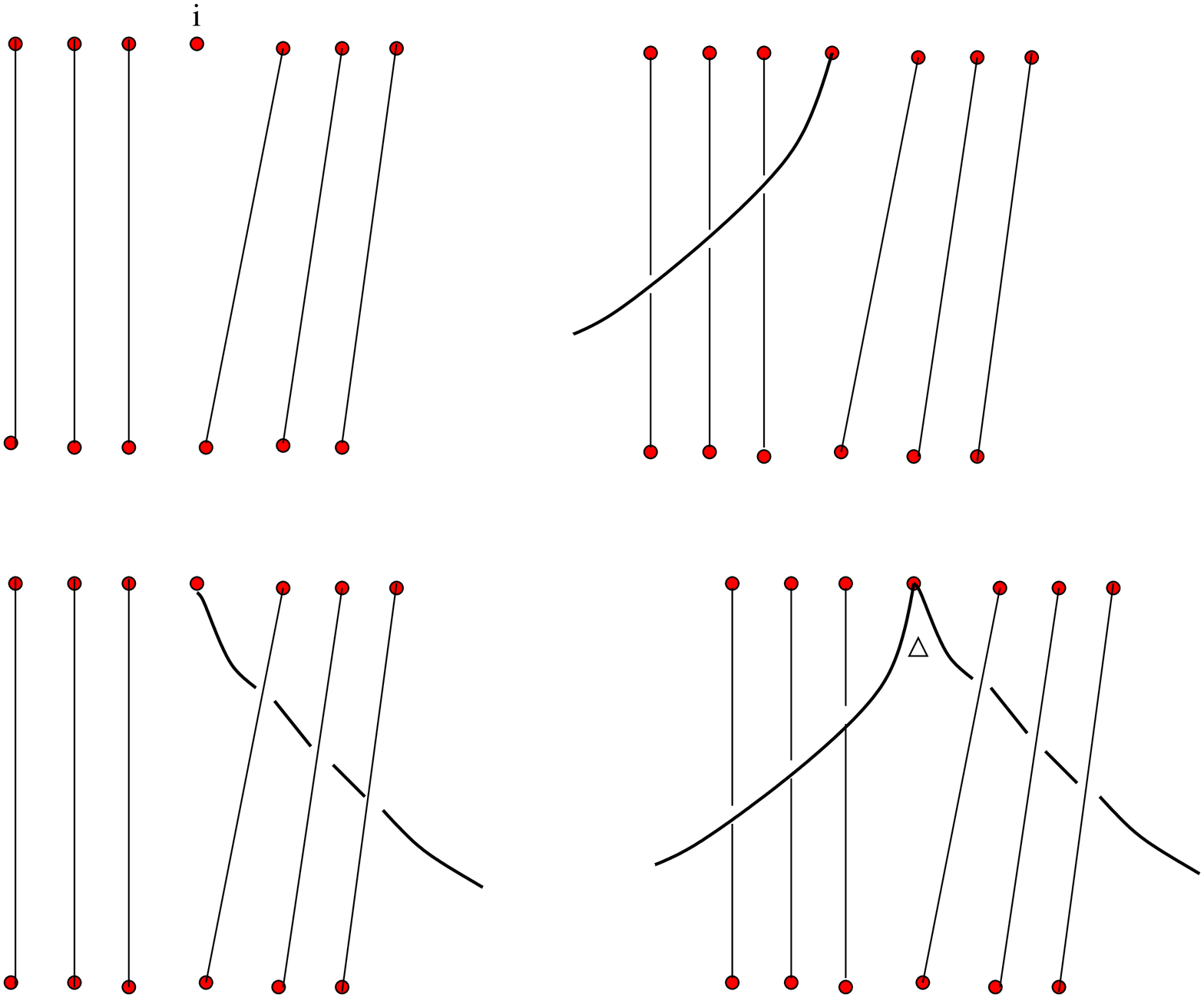,height=6.3cm}}
\centerline{Figure 12.2; Various interpretation of the graphical face map $d_i$}

\subsection{Graphical visualization of Yang-Baxter face maps}
The presimplicial set corresponding to (two term) Yang-Baxter homology has the following 
visualization. In the case  of a set-theoretic Yang-Baxter equation we recover the homology of 
\cite{CES-2}.
\ \\
\centerline{\psfig{figure=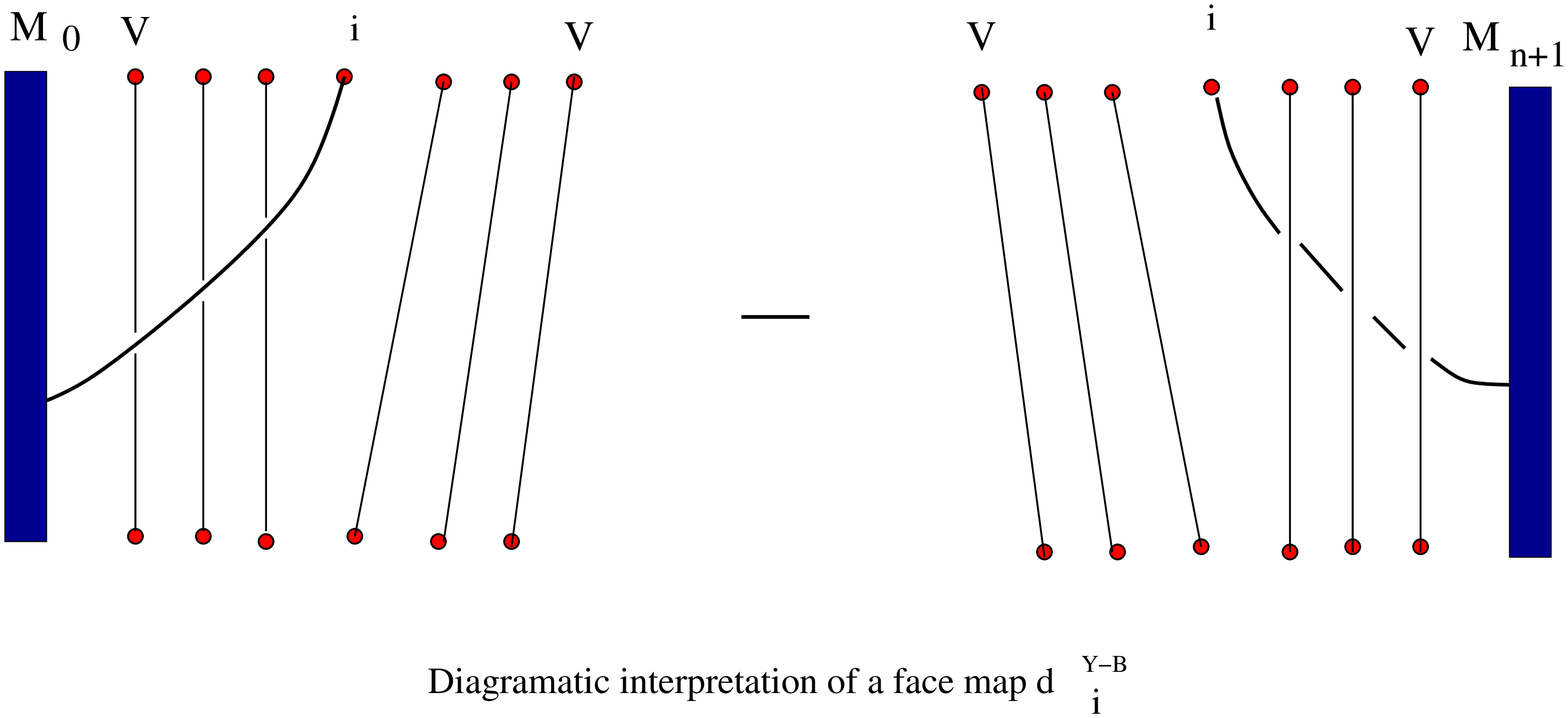,height=6.3cm}}
\centerline{Figure 12.3; Graphical interpretation of the face map $d_i$}
\ \\ \ \\
Our graphical model allows easy calculation:
\begin{example}\label{Example 5.1}
 Assume $R:X\times X \to X\times X$ generates set-theoretic Yang-Baxter operator 
with $R(x,y)= (R_1(x,y),R_2(x,y))$. Then
$$\partial^{YB}(x_1,x_2,x_3,x_4)=\partial^{\ell} -\partial^{r}  $$
$$ \partial^{\ell}(x_1,x_2,x_3,x_4)= ((x_2,x_3,x_4) - (R_2(x_1,x_2),x_3,x_4)  +$$
$$  (R_2(x_1,R_1(x_2,x_3),R_2(x_2,x_3),x_4)
- (R_2(x_1,R_1(x_2,R_1(x_3,x_4), R_2(x_2,R_1(x_3,x_4),R_2(x_3,x_4))$$
$$ \partial^{r}(x_1,x_2,x_3,x_4)= (R_1(x_1,x_2),R_1(R_2(x_1,x_2),x_3),R_1(R_2(R_2(x_1,x_2),x_3),x_4) - $$
$$(x_1,R_1(x_2,x_3),R_1(R_2(x_2,x_3),x_4) + (x_1,x_2,R_1(x_3,x_4)) - (x_1,x_2,x_3).$$
We have generally  for any $n$: $$ \partial^{\ell}_n=\sum_{i=1}^n(-1)^{i-1} d_i^{\ell} 
\mbox{ with }$$
$$d_i^{\ell}(x_1,...,x_n)=$$ 
$$(R_2(x_1,R_1(x_2,R_1(x_3,...,R_1(x_{i-1},x_i)))),...,R_1(x_{i-1},x_i)),x_{i+1},...,x_n).$$ 
Similarly we have, directly from Figure 12.3, for any $n$: $$ \partial^{r}_n=\sum_{i=1}^n(-1)^{i-1} d_i^{r} \mbox{ with }$$
$$d_i^{r}(x_1,...,x_n)=$$
$$(x_1,...,x_{i-1},R_1(x_i,x_{i+1}),...,R_1(R_2(R_2(...(R_2(x_i,x_{i+1}),x_{i+2}),...,x_{n-1})x_n)))).$$

\end{example}


\section{Geometric realization of simplicial and cubic sets}\label{Section 13}

A simplicial (or presimplicial) set (or space, more generally) can be a treated as an instruction of how
 to glue a topological 
space from pieces (simplexes in the most natural case). the result is CW complex or more precisely a $\Delta-$complex 
in the terminology used in \cite{Hat}. That is an object in $X_n$ is a name/label for 
a an $n$ dimensional simplex, and maps $d_i$ (and $s_i$ in a weak simplicial case) 
are giving glueing instruction. Precise description
(following \cite{Lod-1}) is given below. The similar construction for a cubic (or pre-cubic) set is described at 
the end of the section\footnote{The definition comes under the general scheme of a {\it co-end}, p. 371 of \cite{BHS}.}. 
We speculate also what should be a natural generalization of (pre)simplicial and 
(pre)cubic categories.

Let $\mathcal X$ be a simplicial space (e.g. simplicial set with discrete topology),
and $\mathcal Y$ a cosimplicial space (e.g. $\blacktriangle $). We define their product over $\Delta$
similar to the tensor product as follows:
$$ {\mathcal X} \times_{\Delta} {\mathcal Y} = \bigcup_{n\geq 0} (X_n \times Y_n)/\sim_{rel}$$
where $\sim_{rel}$ is an equivalence relation generated by $(x,f_Y(y)) \sim_{rel} (f_X(x),y)$,
$f:[m]\to [n]$, $f_X$ is the image of $f$ under contravariant functor $\Delta \to {\mathcal X}$ and
$f_Y$ is the image of $f$ under the covariant functor $\Delta \to {\mathcal Y}$.
\begin{definition}\label{Definition 13.1}
The geometric realization of a simplicial space $\mathcal X$ is, by definition, the space
$$|{\mathcal X}|= {\mathcal X} \times_{\Delta} {\blacktriangle} = \bigcup_{n\geq 0} (X_n \times \Delta^n)/\sim_{rel}$$
We restrict our topological spaces in order to have $|{\mathcal X}\times {\mathcal Z}|= |{\mathcal X}|\times |{\mathcal Z}|$.

We can perform our construction also for a presimplicial set (or space). The gluing maps are then limited to 
that induced by $d_i$.

If ${\mathcal X}$ is a simplicial set we may consider only nondegenerate elements in $X_n$ (that is elements
which are not images under degeneracy maps) and build $|{\mathcal X}|$ as a CW complex: We start from the union
of $n$-cells $\Delta^n$ indexed by nondegenerate elements in $X_n$; the face operations tells us how
these cells are glued together to form $|{\mathcal X}|$.
\end{definition}

\begin{example}\label{Example 13.2}
 Let ${\mathcal X}$ be an abstract simplicial complex $X=(V,P)$. If we order its vertices then
${\mathcal X}$ is a presimplicial set with $X_n$ being the set of $n$ simplexes
of ${\mathcal X}$ and face are defined on each simplex in a standard way $d_i(x_0,...,x_n)=
(x_0,...,x_{i-1},x_{i+1},...,x_n$
A copresimplicial space is here the category $\blacktriangle $ with objects $Y_n= \Delta^n=\{(y_0,...,y_n)\in \R^{n+1}\ | \ 
\sum_{i=0}^ny_i=1, y_i\geq 0 \}$, one object for any $n\geq 0$. The coface maps $d^i:Y_n\to Y_{n+1}$ are given by 
$d^i(y_0,..,y_n)=(y_0,...,y_{i-1},0,y_i,...,y_n)$ (of course $d^jd^i=d^id^{j-1}$ for $i<j$).
Then the topological realization $|{\mathcal X}|$ is a standard geometric simplicial complex associated to ${\mathcal X}$,
that is 
$$|X|= \bigcup_{n\geq 0} (X_n \times \Delta^n)/(x,d^i(y)=(d_i(x),y) \mbox{ for $x\in X_n$ and $y\in \Delta^{n-1}$,}$$ 
$X_n$ with discrete topology and $|X|$ with quotient topology.
\end{example}

\subsection{Geometric realization of a (pre)cubic set}\label{Subsection 13.1}
First we define a precubic category $\Box_{pre}$.
\begin{definition} the precubic category $\Box_{pre}$ has as objects non-negative integers
$[n]$ interpreted as $n$ points ($[0]$ has the empty object), see Figure 13.1. 
Thus the objects are the same as in presimplicial 
$\Delta_{pre}$ category except the grading shift (now $[n]$ is in grading $n$).
\end{definition}
\ \ \\
\centerline{\psfig{figure=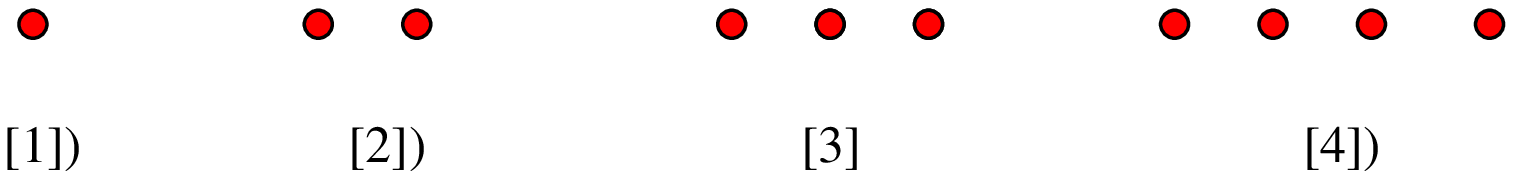,height=1.5cm}} \ \\
\centerline{Figure 13.1;  interpreting objects $[1],[2],[3]$ and $[4]$ in $\Box_{pre}$ category }
\ \ \\

Morphisms are strictly increasing maps that is $f\in Mor([m],[n])$ if 
$f:(1,2,...,m) \to (1,2,...,n)$ and $i\j$ implies $f(i)<f(j)$ with an additional data that points which are not in 
the image of $f$ have $0$ or $1$ associated to them. Morphisms are generated by maps $d^i_{\epsilon}(1,2,...,n)=
(1,...,i-1,i+1,...,n)$ and the point $i$ has marker $\epsilon$ (on the picture marker $0$ is denoted by $\leftarrow$) 
(here $1\leq i \leq n+1$, $\epsilon = 0$ or $1$); see Figure 13.2. \\ \
The presimplicial category $\Delta_{pre}$ is the quotient of the pre-cubic category $\Box_{pre}$,
but this functor is not that interesting in applications.
The proper functor is related with triangulation of a cube \cite{Cla}. \\ \
\centerline{\psfig{figure=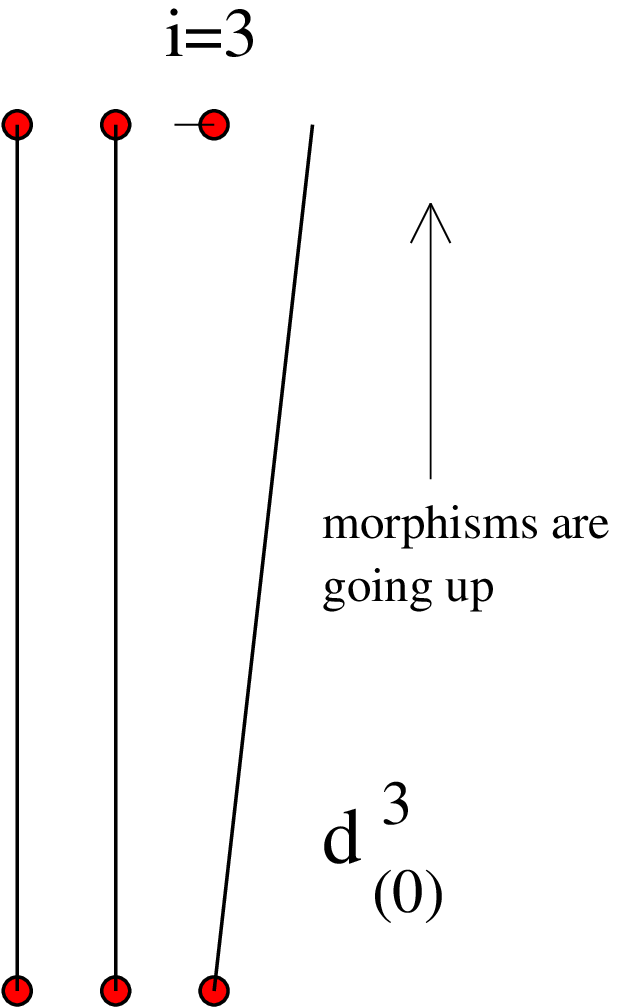,height=4.8cm}} \ \\
\centerline{Figure 13.2;  morphism (co-face map) $d^3_{(0)}$ from $[3]$ to $[4]$ }
\ \ \\
\begin{definition}
\begin{enumerate}
\item[(i)] A pre-cubic category is a contravariant functor from a pre-cubic category $\Box_{pre}$ to 
a given category, ${\mathcal C}$, or, equivalently, a covariant functor ${\mathcal F}: \Box_{pre}^{op} \to  
{\mathcal C}$. We can also say that a pre-cubic category is a sequence of objects $X_n$ in $Ob({\mathcal C})$,
and morphisms $d_i^{\epsilon}: X_n \to X_{n-1}$ satisfying $d_i^{\alpha}d_j^{\beta}=
d_{j-1}^{\beta}d_i^{\alpha}$ for $i<j$. The category $\Box_{pre}^{op}$ is visualized by the same diagrams 
as the category $\Box_{pre}$ except morphisms will be read from the top to the bottom.
\item[(ii)] A co-pre-cubic category is a covariant functor from a pre-cubic category $\Box_{pre}$ to 
a given category, ${\mathcal C}$. A basic example of a co-pre-cubic space (${\mathcal C}= TOP$) is 
given by choosing $X^{n}= I^n=\{(x_1,...,x_n)\in \R^n\ | \ 0\leq x_i \leq 1 \}$ and morphisms:
$d^i_{\epsilon}(x_1,..,x_n)= (x_1,...,x_{i-1},\epsilon,x_{i+1},...,x_n)$. Let us denote this co-pre-cubic 
space by $\blacksquare_{pre}$.
\end{enumerate}
\end{definition}

The basic co-pre-cubic space, $\blacksquare_{pre}$, can be enriched by degenerate (projection) 
 morphisms $s^i: [n]\to [n-1]$ for $1\leq i \leq n$,
given by $s^i(x_1,...,x_n)= (x_1,...,x_{i-1},x_{i+1},...,x_n)$. We denote the new category by $\blacksquare$.
Then the morphisms $d^i_{\epsilon}$ and $s^i$ (with the domain $X_n=I^n$) satisfy:
\begin{enumerate}
\item[(1)] $d^j_{\beta}d^i_{\alpha}= d^i_{\alpha}d^{j-1}_{\beta} \mbox{ for } i < j$,
\item[(2)] $s^js^i= s^is^{j+1}\mbox{ for } i \leq j$,
\item[(3)] $ s^jd^i_{\epsilon}= \left\{ \begin{array}{rl}
 d^i_{\epsilon}s^{j-1} &\mbox{ if $i<j$} \\
d^{i-1}_{\epsilon}s^j &\mbox{ if $i>j$}
       \end{array} \right.
$
\item[(4)] $ s^id^i_{\epsilon}= Id_{I^n}. $
\end{enumerate}

The axioms of the cubic category are modeled on axioms (1)-(4) more precisely:

\begin{definition}\label{Definition 13.5}
\begin{enumerate}
\item[(i)] The cubic category $\Box$  is composed of objects $[n]$, face maps $d^i_{\epsilon}$ (as in pre-cubic category 
$\Box_{pre}$) and degeneracy maps $s_i: [n-1]\to [n]$ ($1\leq i \leq n$ and relations between morphisms 
in the category are given by axioms (1)-(4).
\item[(ii)]  The cubic category is a functor ${\mathcal F}: \Box^{op} \to {\mathcal C}$. 
The classical example (of a cubic space) is giving by an approach to singular cubic homology. 
Here, for a topological space $T$, ${\mathcal F}([n])$ is the set of all continuous maps $f:I^n \to T$.
$d_i^{\epsilon}: {\mathcal F}([n]) \to {\mathcal F}([n-1])$ is given by 
$d_i^{\epsilon}(f)= fd^i_{\epsilon}: I^{n-1} \stackrel{d^i_{\epsilon}}{\to} I^n \stackrel{f}{\to}  T$.
 
\item[(iii)] The co-cubic category is a functor ${\mathcal F}: \Box \to {\mathcal C}$. The classical example
$(I^n,d^i_{\epsilon},s^i)$ was described before as a motivation for a cubic category, $\blacksquare$.
This example is used as a building block of a geometric realization of a pre-cubic and cubic set (Definitions 
\ref{Definition 13.6}, \ref{Definition 13.8}). 
\end{enumerate}
\end{definition}

As we observed already, the pre-cubic category leads to two presimplicial 
categories $(X_n,d_i^{(0)})$ and $(X_n,d_i^{1)}$
(with shifted grading). Conversely, if we have two presimplicial categories so that $d_i^{\alpha}d_j^{\beta}=
d_{j-1}^{\beta}d_i^{\alpha}$ then we combine them to in pre-cubic category (with grade shift).

With degenerate maps situation is not that clear as a condition (3)-(4) of a simplicial category
 only partially agree with the analogous conditions of a cubic category.

The geometrical realization of a cubical and precubical set (or space) is analogous to that for 
simplicial or presimplicial sets (space). We write below a formal definition only in the case of 
a pre-cubic set (space) as degeneracies of a cubic set are not necessarily the one used in 
knot theory.

\begin{definition}\label{Definition 13.6}
The geometric realization of a precubic set is a CW complex defined as follows (notice that 
$X_n$ is indexing cubes and pre-cubic structure gives an instruction how to glue the cubes together):
$$|{\mathcal X}|= {\mathcal X} \times_{\Box_{pre}} {\blacksquare} = \bigcup_{n\geq 0} (X_n \times I^n)/\sim_{rel}$$
where $\sim_{rel}$ is an equivalence relation generated by $(x,d^i_{\epsilon}(y)) \sim_{rel} (d_i^{\epsilon}(x),y)$,
and, as before $d^i_{\epsilon}: I^{n-1}\to I^n$ and $d_i^{\epsilon}: X_i\to X_{i-1}$, $x\in X_n$, and $y\in I^{n-1}$.
\end{definition}

More generally:
\begin{definition}\label{Definition 13.7}
Let $\mathcal X$ be a pre-cubic space (e.g. pre-cubic set with discrete topology),
and $\mathcal Y$ a co-pre-cubic space (e.g. $\blacksquare_{pre} $).
The we define 
$${\mathcal X} \times_{\Box_{pre}} {\mathcal Y}= \bigcup_{n\geq 0} (X_n \times Y_n)/\sim_{rel}.$$
where $\sim_{rel}$ is an equivalence relation generated by $(x,d^i_{\epsilon}(y)) \sim_{rel} (d_i^{\epsilon}(x),y)$.
\end{definition}

If $(X_n,d^{\epsilon}_i,s_i)$ is a cubic set (or space) we define its geometric realization similarly to 
Definitions \ref{Definition 13.6} and \ref{Definition 13.7}, but including also $s_i,s^i$ as gluing morphisms 
(effectively dividing by degenerate part, which is not necessary acyclic). 

\begin{definition}\label{Definition 13.8} 
Let $\mathcal X$ be a cubic set and $\mathcal Y$ the co-cubic space $(I^n,d_{\epsilon}^i,s^i)$ then we define 
the geometric realization $|\mathcal X|$ of $\mathcal X$ as
$$|{\mathcal X}|= {\mathcal X} \times_{\Box} {\blacksquare} = \bigcup_{n\geq 0} (X_n \times I^n)/\sim_{rel}$$
where $\sim_{rel}$ is an equivalence relation generated by $(x,d^i_{\epsilon}(y)) \sim_{rel} (d_i^{\epsilon}(x),y)$,
and analogously with $s_i$ and $s^i$.
\end{definition}

\section{Higher dimensional knot theory $M^n \to \R^{n+2}$ }\label{Section 14}

Many ideas described in this paper can be applied to higher dimensional knot theory,
were we study embeddings of $n$ dimensional (mostly orientable) manifolds in $\R^{n+2}$.
For this we direct readers to \cite{CKS} and \cite{Pr-Ro-2}.

\section{Acknowledgements}
J.~H.~Przytycki was partially supported by the  NSA-AMS 091111 grant,
by the GWU-REF grant, and Simons Collaboration Grant-316446.

\ \\
Department of Mathematics,\\
The George Washington University,\\
Washington, DC 20052\\
e-mail: {\tt przytyck@gwu.edu},\\
University of Maryland College Park,\\
and University of Gda\'nsk, Poland

\end{document}